\documentclass[11pt, reqno]{article}

\usepackage{fullpage}
\usepackage{enumerate}
\usepackage{setspace}

\usepackage{upgreek, mathtools,bm}

\usepackage{amsmath,amsfonts,amssymb,graphics,amsthm, comment}
\usepackage{hyperref}
\hypersetup{
    colorlinks=true,
    linkcolor=blue,
    citecolor=red,
    urlcolor=blue,
    pdfborder={0 0 0}
}
\usepackage{yfonts}

\usepackage{graphicx}
\usepackage{xcolor}
\usepackage{bbm}
\usepackage{wrapfig} % for figures inserted on the side of the text

\usepackage[font=sf, labelfont={sf,bf}, margin=1cm]{caption}

\usepackage{cleveref}
  \crefname{theorem}{Theorem}{Theorems}
  \crefname{thm}{Theorem}{Theorems}
  \crefname{thm*}{Theorem*}{Theorems}
  \crefname{lemma}{Lemma}{Lemmas}
  \crefname{lem}{Lemma}{Lemmas}
  \crefname{remark}{Remark}{Remarks}
  \crefname{prop}{Proposition}{Propositions}
\crefname{notation}{Notation}{Notations}
\crefname{claim}{Claim}{Claims}
  \crefname{defn}{Definition}{Definitions}
  \crefname{corollary}{Corollary}{Corollaries}
  \crefname{section}{Section}{Sections}
  \crefname{figure}{Figure}{Figures}
    \crefname{assumption}{Assumption}{Assumptions}

\newtheorem{thm}{Theorem}[section]
\newtheorem{thm*}{Theorem*}[section]

\newtheorem{lemma}[thm]{Lemma}
\newtheorem{corollary}[thm]{Corollary}
\newtheorem{prop}[thm]{Proposition}
\newtheorem{defn}[thm]{Definition}

\numberwithin{equation}{section}

\theoremstyle{definition}
\newtheorem{remark}[thm]{Remark}
\newtheorem{example}[thm]{Example}
\def \di {\mathsf{}}

\def\cY{\mathcal{Y}}
\def\cX{\mathcal{X}}

\def \h {\mathfrak{h}}
\def\cT{\mathcal{T}}
\def\cS{\mathcal{S}}

\def\cP{\mathcal{P}}

\def\cM{\mathcal{M}}

\def\cH{\mathcal{H}}

\def\cA{\mathcal{A}}
\def \ve {\varepsilon}

\def \ms {\mathsf}
\def \sF {\mathsf{F}}
\def \se {\mathsf{e}}

\def\P{\mathbb{P}}

\def\E{\mathbb{E}}
\def\C{\mathbb{C}}
\def\R{\mathbb{R}}
\def\Z{\mathbb{Z}}

\def\D{\mathbb{D}}

\def\H{\mathbb{H}}

\def  \p- {p\textunderscore}

\def \bs {\boldsymbol}

\def\eps{\varepsilon}

\def \d {{\# \delta}}

\def\Euc{\textsf{euc}}

\DeclareMathOperator{\vol}{d\mu}
\DeclareMathOperator{\Ptemp}{\mathbb{P}_{Temp}}

\DeclareMathOperator{\Pwils}{\mathbb{P}_{Wils}}
\DeclareMathOperator{\Ewils}{\mathbb{E}_{Wils}}
\DeclareMathOperator{\Etemp}{\mathbb{E}_{Temp}}

\DeclareMathOperator{\Pwwils}{\tilde{\mathbb{P}}_{Wils}}

\DeclareMathOperator{\Ztemp}{Z_{Temp}}

\DeclareMathOperator{\Zwils}{Z_{Wils}}
\DeclareMathOperator{\Zwwils}{ \tilde{Z}_{Wils}}

\renewcommand{\i}{\text{\textnormal{int}}}

\newcommand{\note}[1]{{\color{red}{[note: #1]}}}

\newcommand{\abs}[1]{ \lvert #1 \rvert}

\newcommand{\blue}[1]{{\color{blue}{#1}}}

\def\crsf{\vec{\mathcal{T}}}

\newtheorem*{lemma*}{Lemma}

\def \ext {\textsf{ext}}

\begin{document}
\title{Dimers on Riemann surfaces I:\\
 Temperleyan forests}

\author{Nathana\"el Berestycki\thanks{University of Vienna. Email: nathanael.berestycki@univie.ac.at} \and Benoit Laslier\thanks{Université Paris Cité and Sorbonne Université, CNRS, Laboratoire de Probabilités, Statistique et Modélisation. Email: laslier@lpsm.paris} \and Gourab Ray\thanks{University of Victoria.   Email:gourabray@uvic.ca}}

\maketitle

\abstract{This is the first article in a series of two papers in which we study the Temperleyan dimer model on an arbitrary bounded Riemann surface of finite topological type. The end goal of both papers is to prove the convergence of height fluctuations to a universal and conformally invariant scaling limit. In this part we show that the dimer model on the Temperleyan superposition of a graph embedded on the surface and its dual is well posed, provided that we remove an appropriate number of punctures. We further show that the resulting dimer configuration is in bijection with an object which we call Temperleyan forest, whose law is characterised in terms of a certain topological condition. Finally we discuss the relation between height differences and Temperleyan forest and prove that the convergence of the latter (which is the subject of the second paper in this series) implies the existence of a conformally invariant scaling limit of height fluctuations.}

\tableofcontents

%\note{Here is an idea of the content for the introduction :}
%\begin{itemize}
%	\item This is the second part of a project on dimer and imaginary geometry on surfaces
%	\item motivations to study dimers (in particular why surfaces)
%	\item motivation to study imaginary geometry (in particular why surfaces) \note{motivations for spanning tree and CRSF by themself are mentionned in the other introduction}
%	\item General idea of the approach from the original paper on the simply connected case
%	\item A short note on what needs to be added compared to that
%	\item division of the two papers with a short mention of what was proved in the previous one. \note{Where is the bijection proved ?}
%	\item More precise discussion of the content of this paper together with the organisation of the paper. \note{we need to merge the two setup sections.}
%\end{itemize}

\section{Introduction}

\subsection{Background and main result}

This paper is the first part of a series of two papers in which we show universality and conformal invariance of the Temperleyan dimer model on Riemann surfaces.% \noteb{I am not sure we can really claim this series is about imaginary geometry. Maybe claim `conformal invariance of dimer model on Riemann surfaces' }

%The dimer model is a classical model of statistical mechanics which is defined as taking uniformly at random a perfect matching of the vertices of some graph $G$ along its edges. 
\medskip Given a finite graph $G$ (with an even number of vertices) and nonnegative weights $w = (w_e)_{e \in E}$ on the edges, a dimer configuration is a perfect matching of the vertices along the edges of $G$. The dimer model is the random perfect matching obtained by sampling a dimer configuration $\bf m$ on $G$ with probability proportional to the product of the weights of the edges in the matching $\bf m$: 
 \begin{equation}\label{Gibbs}
\P({\bf m}) = \frac{\prod_{e \in {\bf m}}  w_e  }{Z},
\end{equation}
where $Z$ is the partition function. If all the weights are equal to 1, $\P$ is simply a uniformly picked perfect matching.  Following Thurston, a particularly convenient way to encode the dimer model on a \emph{planar} bipartite graph is through a notion of height function. Informally, it is a real valued function defined on the faces of the graph, which is uniquely determined by the dimer configuration. This turns the dimer model into a model of discrete random surfaces. A central question in the study of the dimer model is to describe the fluctuations around the mean of the height function. It turns out that the height function of the dimer model is extremely sensitive to boundary conditions (see for example the quote of Kasteleyn at the beginning of  \cite{CohnKenyonPropp}). In the so-called \textbf{liquid phase}, the Kenyon--Okounkov conjecture predicts that the fluctuations of the height function  around its mean are asymptotically given by a \textbf{Gaussian free field} in a certain nontrivial \emph{conformal gauge} (see \cite{KenyonOkounkov2} section 2.3 for details). Despite remarkable progress on this conjecture over the last decade, this remains an outstanding problem in all but a handful of cases: see \cite{BLR16,dubedat_torsion,KenyonGFF,petrov,marianna} for examples of models where this is already shown, and see \cite{BerestyckiPowell,GFFShe} for an introduction to the Gaussian free field. We emphasise that, given the extreme sensitivity of the model to its boundary conditions, the occurrence of a universal limit such as the Gaussian free field is to some extent very surprising.
%A remarkable property of the Gaussian free field (and hence also of the dimer model) is that it is conformally invariant so even though at the discrete level, the model is very sensible to boundary conditions and the shape of the underlying graph $G$, at large scale all simply-connected domains should be equivalent (at least as long as the boundary heights are such that there is a single liquid phase inside).

\medskip The goal of this paper is to describe the fluctuations of the height function for certain graphs embedded on Riemann surfaces of finite topological type, i.e., with finitely many handles and holes. Our assumptions on these graphs ensure that the dimer model is liquid on the entire Riemann surface {(as opposed to having frozen or gaseous regions, say as in the  Aztec diamond)} and the conformal gauge is simply the identity; we call the resulting dimer model \textbf{Temperleyan}, and we say that the sequence of graphs $G^\d$ is a Temperleyan discretisation of the surface. Even though this assumption simplifies the problem, conformal invariance in the scaling limit is by no means obvious; for instance, in the simply connected setting and for the square lattice, this is in fact precisely the content of Kenyon's landmark paper \cite{Kenyon_ci}. 

Before we give a simplified statement of the main theorem, we point out that in the context of a Riemann surface $M$, the height function becomes a height one-form $h^\d$: that is, the gradient $\nabla h^\d$ is well defined but this gradient does not derive itself from a single-valued function on the surface. Instead it can be viewed as the gradient of a (single-valued) function defined on the universal cover of the surface. For concreteness we will denote by $h^\d$ the associated function on the universal cover of the surface. Readers who prefer it can however equivalently think of $h^\d$ as the random one-form on $M$ associated to the dimer configuration. (See \cref{sec:ht,sec:statement} for precise definitions.) The function $h^\d$ is the central object of this sequence of papers, and we still refer to it as the height function.

Our main theorem, stated informally below, confirms that the dimer model admits a universal scaling limit on Riemann surfaces. This scaling limit is furthermore conformally invariant. 

\begin{thm}\label{T:main_intro}
	Let $M$ denote a bounded Riemann surface, possibly with a boundary, which is not the sphere or a simply connected domain. Let $G^\d$ be a sequence of Temperleyan discretisations of $M$ satisfying the Invariance principle and a crossing estimate for random walk described in \cref{sec:setup}. Let $h^\d$ denote the height function of the dimer model on $G^\d$. For all smooth compactly supported test functions $f$ on the universal cover of $M$,
	$$
	\int (h^\d - \E(h^\d)) f 
	$$
	converges as $\delta \to 0$ in law and the limit is conformally invariant  and independent of the sequence $G^\d$.
\end{thm}

The precise assumptions on the surface $M$ will be given in \cref{sec:surface_embedding}, the meaning of  Temperleyan discretisations and our precise assumptions on the sequence $G^\d$ can be found in \cref{sec:setup}. {Finally, note that the smooth test function $f$ is taken on the universal cover $\tilde M$ of $M$, so that the integration takes place on $\tilde M$ (equipped e.g. with Lebesgue measure).} 
See \cref{sec:ht,sec:statement} for details. We refer to \cref{thm:main_precise} for the complete version of the theorem.

Although we cannot find what the limiting object is, as an artefact of our proof, we can show that the limit is non trivial. Indeed, we prove later in \cref{lem:exp_tail} that the limiting field restricted to a small enough neighbourhood of $M$ away from the punctures can be coupled to the restriction of a Gaussian free field in a slightly bigger neighbourhood, with positive probability; in particular this field is nontrivial. We conjecture that the limiting field is a \textbf{compactified Gaussian free field} with appropriate parameters depending only on the surface $M$ and the positions of the punctures $(x_1, \ldots, x_k)$ on $M$.
\subsection{Relation with previous work}

One of the first rigourous works on the dimer model on Riemann surfaces is the inspiring paper by Boutillier and de Tili\`ere \cite{BdT_torus}, who were motivated by nonrigourous physics predictions based on the Coulomb gas formalism (see e.g. \cite{Coulomb_gas} and the references in \cite{BdT_torus}). They described the topological part of the fluctuations in the dimer model on the honeycomb lattice on the torus. This was complemented some years later by the fundamental paper of Dub\'edat \cite{dubedat_torsion}, in which the full law of the limiting height fluctuations was described for isoradial graphs on the flat torus, and identified as a so-called compactified Gaussian free field. Dub\'edat also conjectured a greater form of universality beyond the isoradial setting; such a universality statement was subsequently proved by Dub\'edat and Gheissari \cite{DubedatGheissari} but only for the topological part of fluctuations on the torus. Note that in particular our Theorem \ref{T:main_intro} settles this conjecture of Dub\'edat. 

\medskip 
We point out that all techniques in the paper mentioned above, based on Kasteleyn theory (in contrast with our approach), are very specific to the case of the torus and cannot easily be adapted to general Riemann surfaces. Cimasoni \cite{cimasoni} shows that for general surfaces of genus $g$ (and no boundary), the partition function of the dimer model becomes a weighted (in fact, signed) sum of $2^{2g}$ determinants of Kasteleyn matrices in which the orientation has been reversed along some
of the $2g$ cycles forming a basis of the first homology group of the surface. The signs in front of the determinants themselves form a nontrivial algebraic topological invariant of the surface which Cimasoni, remarkably, was able to identify in terms of the so-called Arf invariant. (In the case of the torus and the honeycomb lattice, this goes back to the work of Boutillier and de Tili\`ere \cite{BdT_torus}.) This illustrates the difficulty in generalising the Kasteleyn approach to general surfaces. {See also \cite{ahn2021lozenge,cimasoni2007dimers,cimasoni2008dimers,costa2002dimers} for related results.
}

\medskip The above difficulty illustrates the well known fact that the dimer models becomes ``less integrable'' as the genus of the surface increases, or at any rate integrability becomes harder to exploit. {In this paper, we will therefore follow the general strategy of our previous work \cite{BLR16}, which relies instead on Temperley's bijection between the dimer model and uniform spanning trees (in fact, see the overview of the proof in \cref{sec:organization} for a discussion of this strategy and the relation between the two papers).}
%Our approach is not dependent on integrability or solvability \blue{instead relying on (an extension of) Temperley's bijection following our previous work \cite{BLR16}. In fact, since the general strategy of \cite{BLR} is central to this paper} and this
{The increased robustness of this method}
 accounts in large part for the fact that we are able to tackle generic Riemann surfaces (in contrast with methods based on Kasteleyn theory or interacting particle systems).  This is not to say that the difficulties caused by the increased complexity of the surface vanish entirely, but they manifest themselves in a different manner, which boils down to the following. As we will describe in more details in \cref{sec:background}, a Temperleyan discretisation of the surface is the graph resulting from the superposition of a graph $\Gamma$ embedded on the surface and its dual. The resulting graph is bipartite, with its black vertices being given by the vertices and faces of $\Gamma$, and the white vertices being given by the edges of $\Gamma$. However, a straightforward calculation using Euler's formula shows that this does not admit a dimer covering unless we remove $2g + b -2 = |\chi|$ many white vertices from the graph (where $g$ is the genus and $b$ the number of boundary components, and $\chi$ is the Euler characteristic); these can be thought of as \textbf{punctures} in the surface, or monomers. Note that the number of punctures increases with the complexity of the surface, and when the surface is a torus or an annulus no such puncture needs to be removed. {Roughly speaking, we will see that the locally tree-like random subgraphs  appearing when we generalise Temperley's bijection, which we will call \textbf{Temperleyan self-dual pair}, have one topological constraint for each puncture. In the case of a torus or an annulus where there is no puncture, we will see that the Temperleyan self-dual pair boils down to the simpler notion of cycle rooted spanning forest studied by Kassel and Kenyon in \cite{KK12}.} 
 
 %In the general case, it is easy to believe that every extra topological constraint makes in a sense the model less integrable, even though it turns out that the only big gap in complexity comes with the first puncture.
%Roughly speaking, each puncture or monomer is responsible for an extra topological condition which the Temperleyan forest is required to satisfy; in particular in the case of a torus or an annulus, the Temperleyan forest is not conditioned on any degenerate event and boils down to the simpler notion of cycle-rooted spanning forest studied by Kassel and Kenyon in \cite{KK12}.

\medskip The necessity to remove a certain number of punctures from the surface in order to formulate the Temperleyan dimer model has a geometric flavour which is reminiscent of \textbf{Liouville conformal field theory}, in which correlation functions are only well defined with the appropriate number of singularities (known in that context as insertions). The required number of insertions is specified by the so-called Seiberg bounds, see e.g.  \cite{DKRV} on the sphere and \cite{GKRV} for Riemann surfaces (see also \cite{BerestyckiPowell} for an introduction). As is the case here, the number of punctures is entirely specified by the topology of the surface, but in a way that is in some sense opposite to ours: indeed the required number of insertions is an increasing function of $\chi$ in Liouville CFT, whereas it is decreasing here.

This difference can be understood informally as follows. To first order, Liouville CFT may be viewed as yielding a random metric which is approximately hyperbolic (this can for instance be made rigourous in the semiclassical limit where the parameter $\gamma \to 0$, see \cite{LacoinRhodesVargas}). However, {on the Riemann sphere, for which the theory is of particular interest, and on other surfaces of non-negative Euler characteristic}, it is impossible to equip the surface with a metric of constant negative {scalar} curvature without first adding conical singularities at a certain number of points, as follows from the Gauss--Bonnet theorem. By contrast, the Temperleyan dimer model is trying to equip the surface with a \textbf{flat metric}, a fact which will be particularly apparent when we relate the height function to the winding of the Temperleyan forest in \cref{sec:winding_ht} (note that the winding is computed on the universal cover, where we ignore the effects of curvature, see \Cref{SS:universalcover}). Yet, on a hyperbolic surface (i.e., when $\chi<0$) no flat metric exists, again by the Gauss--Bonnet theorem. In summary, Liouville CFT may be viewed as trying to build a hyperbolic metric on a surface which may be parabolic or elliptic, whereas the Temperleyan dimer model tries to impose a flat metric on a surface which may be hyperbolic. In both cases the difficulty is resolved by puncturing the surface at an appropriate number of singularities. (The above parallel between Liouville CFT and dimer model may perhaps also be compared with the two couplings between SLE and Gaussian free fields given respectively by the so-called forward and reverse couplings; the forward coupling describes Liouville conformal field theory whereas the reverse coupling describes imaginary geometry \cite{IG1,IG4}).

\subsection{Overview of the proof and organisation of the paper}\label{sec:organization}

%\paragraph{Brief outline.} 
{As mentioned above, }our general approach to study height fluctuations on Riemann surfaces is motivated and inspired by our earlier work \cite{BLR16} connecting the dimer model to \textbf{imaginary geometry} in the simply connected setting. 
Let us recall the main idea of that work briefly in the original simply connected setting, before explaining how this relates with and differs from the case of Riemann surfaces.

\begin{itemize}

\item The starting point of \cite{BLR16} is Temperley's bijection, which relates the dimer model on a Temperleyan discretisation of a simply connected domain to a pair of dual uniform spanning trees with respectively wired and free boundary conditions. Furthermore, the height function of the dimer model 
can be determined relatively simply from the associated pair of spanning trees: roughly speaking, the height difference between two points is simply the total winding (sum of turning angles) of the unique path connecting these points in either tree. 

\item The second observation is that, as follows from the landmark paper of Lawler, Schramm and Werner \cite{LSW}, paths between points in the wired uniform spanning tree (UST) have a scaling limit which may be described in terms of Schramm--Loewner Evolutions with parameter $\kappa = 2$ or SLE$_2$ for short. 

\item Finally, in the continuum, there is a coupling known as {imaginary geometry} between SLE$_\kappa$ curves and an appropriate multiple of the Gaussian free field. In this coupling, the values of the GFF may informally be identified as the winding of the associated SLE curves (note that this requires careful interpretation, since the Gaussian free field is not defined pointwise and SLE$_\kappa$ paths are not smooth). Nevertheless, this coupling can be thought of as a continuous form of Temperley's bijection in the case $\kappa =2$.

\end{itemize}

The strategy of \cite{BLR16} is to exploit Temperley's bijection on the one hand, and imaginary geometry on the other hand, to show that the winding of paths in a UST is indeed given by the appropriate multiple of the Gaussian free field asymptotically (essentially, one has to show that we can exchange the order of taking limits, or that a diagram commutes). This strategy is very robust; in particular it does not appeal to the more solvable aspects of the dimer model. Indeed, the convergence of loop-erased random walks (which describe the branches of a wired UST at the discrete level) to SLE$_2$ is known in great generality.

\medskip Roughly speaking, we wish to implement a similar approach in the case of Riemann surfaces. However, it should be clear that each part will be substantially different in this new setup. 

\begin{enumerate}

\item To begin with, we observe that Temperley's bijection is local and so may be applied on a Riemann surface to output a pair of random subgraphs that are dual to one another and locally tree-like, and which we call \textbf{Temperleyan self-dual pair}. While in the simply connected case, these are just dual spanning trees, the {combinatorial structure} (i.e. state space) of the Temperleyan self-dual pair on a Riemann surface is a priori not entirely clear, and its law even less so. Our first result will be a characterisation of the Temperleyan self-dual pair. We also describe its law explicitly, which has a simple and concise expression. It turns out that one of the two components, which we call 
the \textbf{Temperleyan cycle rooted spanning forest} or Temperleyan CRSF for short, determines the other (and hence the entire dimer configuration) almost uniquely (up to some choice of orienting cycles), and also has a tractable law. Essentially, a Temperleyan CRSF is a random spanning subgraph in which every cycle is non-contractible, and satisfying a certain further topological condition. The introduction and study of Temperleyan CRSFs take up a substantial part in both the articles in this series.

\item A second problem is that the connection between height function and winding is not as straightforward as in the simply connected case; in particular, in addition to the already mentioned fact that only the gradient of the height function is well defined (leading to a multivalued height function), there may not be a path connecting two given points in the Temperleyan CRSF, even when they are neighbours. We thus need to develop a systematic method for computing the height function given the Temperleyan CRSF.

%Then one needs to refine the link between the winding of these objects and the dimer height function. Here note that since a graph on a Riemann surface is non-planar, the definition of the height function only works locally so it is possible to accumulate a non-trivial amount of height as one makes a non-contractible loop around the surface. Hence one needs to either view the height as defined on the universal cover or to use a version of the Hodge decomposition to write it as both a true function \textbf{scalar component} and an harmonic one-form called the \textbf{instanton component}.

\item The next difficulty is to obtain a scaling limit result for our Temperleyan CRSF. Let us immediately note here that there are specific difficulties with {even }defining variants of SLE on surfaces since the simply connected uniformisation theorem lies at the heart of the standard definition of SLE.
\item %In addition, 
{Finally} if we were to follow the blueprint of \cite{BLR16} one would also need to establish a version of imaginary geometry for Riemann surfaces, an endeavour which seems out of reach with current techniques. Fortunately, we are in fact able to entirely bypass this step because 
it turns out that the (discrete) {``error'' }estimates needed anyway to exchange limits are strong enough to imply convergence regardless of any information in the continuum. Clearly, the downside of this approach that we cannot hope to describe the law of the limit in this manner. 
%(Another related reason is that our approach to establish the scaling limit for Temperleyan CRSF does not give any usable description of the limit law.)

%\note{I am not completely happy with the formulation here. It is true that the lack of imaginary geometry on surfaces is a neat place where we see that we will not get any description of the limit. On the other hand I feel like the main difficulty if we wanted to obtain this description in the spirit of \cite{BLR16} would be in proving the scaling limit to the right version of SLE, rather than establishing the continuum coupling.}

\end{enumerate}

\medskip The topological and geometric structures of the surface $M$ raise substantial specific difficulties at each of the steps of this roadmap. In this first paper we develop the relevant discrete/combinatorial arguments corresponding to the first two items above { and present the conclusion according to the last item}. The second paper in this sequence, \cite{BLR_Riemann2}, is devoted to the {third bullet point, i.e the }proof of convergence of the Temperleyan CRSF in the scaling limit. This is the most involved step which requires several new ideas, but can be read essentially independently. 
{
At this point, it is also useful to discuss a bit more precisely the role that the different structures associated to a Riemann surface will play in our paper. In order to even discuss convergence on the surface $M$ one needs in principle to specify a Riemannian metric on $M$, and it is natural to require that this metric be consistent with the conformal structure of $M$.
However, ultimately, we want to prove convergence to a conformally invariant object so 
in principle only the conformal class of this metric matters (here we mean that two metric tensors are equivalent if we obtain one from the other by multiplying by a smooth function $e^\varphi$.)
  
%ideally we would like to think of $M$ (or rather the punctured surface $M'$) as only a conformal class. Because of this, it will be important to insist that the lift from $M$ (or $M'$) to their cover are conformal but we are still free to chose any metric (in the same conformal class) on the cover. 

It turns out however that, after lifting to the universal cover (which in most of the interesting cases for this paper will be the unit disc, see below for more details), it will be more convenient for us to equip this cover with
the standard Euclidean metric instead of the \emph{a priori} more natural hyperbolic metric. 

 %can for example be seen in the way the Invariance principle and crossing estimates are formulated
This impacts the way we formulate some of our assumptions in \cref{sec:setup}. Unfortunately, some of our technical estimates involve indirectly the Euclidean metric on the unit disc: for instance, through our assumptions that the graph $G$ has a bounded density and each edge has a bounded winding (\cref{boundeddensity,embedding} in \cref{sec:setup}). Nevertheless, changing the metric to another one in the same class only changes the constants involved in these assumptions, and so does not impact the main results. {We explicitly prove this fact in \Cref{lem:conf_inv_assumption}}. 
\bigskip

\subsection*{Organisation of the paper}
%\note{I just copied that from the previous version since it should be essentially good, check that it fits the new perspective where we have a full result on all surfaces.}
\begin{itemize}
	%\item
	
	%\medskip \noindent
	\item In \cref{sec:background}, we recall some basic facts about dimers on surfaces and in \cref{sec:winding}, we recall the notions of windings of curves which were developed in our previous article \cite{BLR16}. \cref{sec:surface_embedding,SS:universalcover} contain small reminders about Riemann surfaces. We introduce precisely our discrete setup in \cref{sec:setup}. As one of the main differences with the classical case is that the height function now becomes multivalued we recall in \cref{sec:ht} the language of height one-forms on graphs which is a classical way to handle such multivalued functions. In \cref{sec:ht}, we introduce another way to deal with multivalued functions by lifting them to single-valued functions on the universal cover. We also illustrate how an application of the Hodge decomposition theorem allows us to decompose the (multivalued) height function on the surface into a single-valued function (scalar component) and a canonical representative of the ``multivalued (or topological) part", the so-called instanton component. 
	
	\item In \cref{sec:Temp}, we introduce the generalisation of Temperley's bijection to Riemann surfaces{, essentially covering point 1 above}. \cref{SS:Tempforest} contains the bijection itself.
 %which, as mentioned above, uses a new type of object which we call a \textbf{Temperleyan CRSF} or Temperleyan forest for short. 
 We also provide in \cref{SS:reducCRSF} a simple criterion for a CRSF to be Temerleyan which is the basis of the second paper in the series \cite{BLR_Riemann2}.
	
	\item In \cref{sec:winding_ht}, we carefully develop the relation between
	the height differences and the winding of Temperleyan forests, i.e point 2 of the overview. This is in the spirit of the work of Kenyon--Propp--Wilson \cite{KPWtemperley}, who treated the (planar) simply connected case. In that case of course the bijection is with a uniform spanning tree.
	Compared to that work there are two additional points that we need to handle. The first one is that the edges of the graph cannot be assumed to be straight lines as in \cite{KPWtemperley}. The second and more significant one is that there are additional terms coming from the fact that the forest is not connected: given points $x$ and $y$ on the universal cover, the height difference $h(x) - h(y)$ (which is unambiguously defined) is essentially given by the intrinsic winding of any path connecting $x$ to $y$, \emph{plus} additional discontinuities every time the path jumps between components. The resulting key formula is stated in \cref{lem:number_crossing} (see also \cref{lem:winding_arbitrary_path}).

	\item {In \cref{sec:local_coupling_grand} we extend our local and global coupling results from \cite{BLR16} to the framework of Riemann surfaces. This coupling is a key ingredient of our approach, which allows us to show that given a finite number of points $(z_i)_{1\le i \le k}$ on the surface (or, rather, their lifts to the universal cover), the respective geometries of the CRSF in neighbourhoods of these $k$ points are almost independent. This is crucial because certain types of ``error terms'' in our approach vanish in expectation but are large in absolute value.}

	\item In \cref{sec:ht_convergence}, we conclude by proving our main convergence result: \cref{thm:main_precise}. \cref{sec:apriori,sec:full_spine} contain some technical a priori estimates on winding of the CRSF branches on the universal cover.
	%Because Wilson's algorithm only reveals one copy of the branch in the cover, we actually need some input from the classical theory of Riemann surfaces  (\cref{app:spine}) along with some significant amount of additional work compared to the planar case.
	Finally in \cref{sec:winding_convergence}, we finish the proof of \cref{thm:main_precise}. %This proof shares similarities with the argument in \cite{BLR16}, but with the difference that we cannot rely on imaginary geometry to identify the limit.
	
	\item Finally, \cref{app:spine} contains some geometric facts about \emph{spines} in the CRSF (that is, the lifts to the universal cover of a cycle of the CRSF). It is given as a appendix because the proofs relies on slightly more involved elements of the classical theory of Riemann surfaces than the rest of the paper and we wanted to avoid overloading the already lengthy setup section.%It is shown that such spines, in the hyperbolic case, form simple chords in the unit disc (simple paths joining two boundary points) or loops touching the boundary at a single point. This basic fact (stated as \cref{lem:spine:boundary} but proved in the appendix) underlies much of the discussion connecting winding to height differences. The proof relies on some arguments from the classical theory of Riemann surfaces.
	
\end{itemize}

\subsection*{Acknowledgements}

The authors are grateful to Antoine Dahlqvist and Julien Dub\'edat  for a number of useful discussions. NB's is supported by FWF grant P33083, ``Scaling limits in random conformal geometry''. GR's research is supported by NSERC 50311-57400. BL's research is supported by ANR-18-CE40-0033 “Dimers”.
The paper was finished while NB was in residence at the Mathematical Sciences Research Institute in Berkeley, California, during
the Spring 2022 semester on \emph{Analysis and Geometry of Random Spaces}, which was supported by the National Science Foundation under Grant No. DMS-1928930. Finally, we thank the anonymous referee for valuable feedback on an earlier version of the manuscript.

\section{Background and setup}\label{sec:background}

\subsection{Riemann surfaces and embedding}\label{sec:surface_embedding}
In this article, we work with a {Riemann surface $M$ (a connected, one dimensional complex manifold)} satisfying the following properties.
\begin{itemize}
\item $M$ is of finite topological type, meaning that
  the fundamental group $\pi_1(M)$ is finitely generated. In other
  words, we assume that the surface has finitely many ``handles'' and
  ``holes''.
\item  $M$ can be compactified by
  specifying a \textbf{boundary} $\partial M$. We denote by $\overline
  M$  the compactified Riemann surface with the boundary. More
  precisely, every point is either in the interior and hence has a
  local chart homeomorphic to $\C$ or is on the boundary and has a
  local chart homeomorphic to the closed upper half plane $\bar \H$. Also there are finitely many such charts which cover the boundary. Note that this condition implies that $M$ has no punctures. (However for future reference, we note here that we will later introduce punctures on $M$ which will align with the removed vertices of Temperleyan graphs. The resulting punctured manifold will be denoted by $M'$.)
\end{itemize}

We say $M$ is \textbf{nice} if $M $ satisfies the above properties. 
\medskip

{ Given such a Riemann surface $M$, it is possible to equip it with a Riemannian metric $g$ (with associated distance function denoted by $d_M$) which extends continuously to the boundary, and turns $M$ into a smooth Riemannian surface. This Riemannian metric can be constructed by considering the hyperbolic metric on the universal cover of the surface if it has no boundary, see \cite[Section 2.4]{Jost} in combination with the uniformisation theorem (\cite[Theorem 4.4.1]{Jost}, also recalled and discussed below).
If the surface has boundary components, we may apply the same result to the surface obtained by gluing it to itself along the boundary -- also called the double.

Note that the choice of the metric above exclude surfaces such as the hyperbolic plane. This will simplify certain topological issues later when we deal with the Schramm topology.

%\note{Talk about extension to $\bar M$.}
%%In other words, recall that this metric can be represented locally in isothermal coordinates as $e^{\rho}|dz|^2$ for a smooth function $\rho$ which extends continuously to the boundary.
}

\medskip \paragraph{Classification of surfaces.}

Riemann surfaces can be classified into the following classes depending on their conformal type (see e.g. \cite{donaldson2011riemann} for an account of the classical theory):

\begin{itemize}

\item Elliptic: this class consists only of the Riemann sphere, i.e., $M \equiv \hat{  \mathbb{C}}$

\item Parabolic: this class includes the torus, i.e., $ M \equiv \mathbb T =  \C / (\Z + \tau \Z)$ where $\Im(\tau >0)$, the cylinder $M \equiv \mathbb{C} \setminus \{0\}$, and the complex plane itself (or the Riemann sphere minus a point), $M \equiv \C$.

\item Hyperbolic: this class contains everything else. This includes examples such as the two-torus, the annulus, as well as proper simply connected domains in the complex plane, etc.
\end{itemize}
%\note{This sentence makes no grammatical sense, we say is conformally equivalent to elliptic.}
The proofs in this paper (and its companion \cite{BLR_Riemann2}) are concerned with the hyperbolic case (subject to the above conditions) as well as the case of the torus. So from now on we always assume that $M$ is such a manifold. We note that the case of simply connected {proper} domain in $\C$ is covered in our previous work \cite{BLR16}. %These are representative of the main difficulties that arise. We also remark that, in the case of simply connected domains in $\C$, this result is a  special case of our previous work \cite{BLR16}.

% The other cases in the elliptic and parabolic classes (i.e. the whole complex plane, the punctured plane and the Riemann sphere) can all be treated with identical arguments as in our previous work \cite{BLR16}. We will explain in \cref{sec:remainder} how the proofs carry over for these cases.

\subsection{Universal cover}
\label{SS:universalcover}

{ The universal cover $\tilde M$ of the Riemann surface $M$ will play an important role in our analysis.
Recall that any Riemann surface admits a (regular) covering map by a covering space which is simply connected. This covering space may furthermore be endowed with a conformal structure, with respect to which the covering map is then analytic. Being simply connected, this covering space is (by the uniformisation theorem, see Theorem 4.4.1 and Theorem 2.4.3 in \cite{Jost}) conformally equivalent to a Riemann surface $\tilde M$, which is either the unit disc $\D \subset \C$ (hyperbolic case), the whole plane (parabolic case), or the Riemann sphere $\hat \C$ (elliptic case). Altogether we obtain a map\footnote{For concreteness one can fix a point $x_0$ in $M$ and consider $p$ such that $p(0) = x_0$. In the hyperbolic case fix an arbitrary choice of rotation as well.} $p:\tilde M \to M$ which is both analytic and a regular covering map of $M$. By definition of a covering map, $p$ is a local homeomorphism: for every $z \in M$, there exists a neighbourhood $N$ containing $z$ so that $p$ is injective in every component of $p^{-1}(N)$.}
        
%We further recall here that the covering map acts discretely, in the sense that for every $z \in M$, there exists a neighbourhood $N$ containing $z$ so that $p$ is injective in every component of $p^{-1}(N)$.

{
 Recall further 
 that 
 %the classification of %Riemann surfaces is a corollary of the  Poincar\'{e}--Koebe Uniformisation theorem, which states that every simply connected Riemann surface is conformally equivalent to either the Riemann sphere $\hat \C$ (elliptic case), the complex plane $\C$ (parabolic case), or the unit disc $\D$ in the complex plane (hyperbolic case). Recall further that these three spaces 
 in each of the three cases (elliptic, parabolic, hyperbolic), $\tilde M$
 can be endowed with a metric (compatible with the complex structure) of constant Gaussian curvature equal to $1,0$ and $-1$ respectively. 
 Furthermore, there is a subgroup $F$ of conformal automorphisms on $\tilde M$ acting  properly discontinuously and freely -- i.e., without fixed points --  (see, e.g., \cite[Section 2.4]{Jost} for definitions),
 such that $p$ descends to a conformal equivalence between $\tilde M/F$ and $M$. Put it more simply, $F$ is a discrete subgroup of the corresponding set of M\"obius transformations describing the conformal automorphisms of $\tilde M$, and $M$ is conformally equivalent to  $\tilde M /F$ (i.e., to either $\hat \C / F$, $\C/F$ or  $\D/F$).

%We can apply the P\'{o}incare Koebe Uniformisation theorem to the universal cover $\tilde M$, and using the fact that $\tilde M$ is conformally equivalent to 

%as a corollary, we obtain the  classical \textbf{Riemann uniformisation theorem}: a Riemann surface $M$ is conformally equivalent to $\tilde M / H$ where $H$ is a subgroup of the automorphism group of $\tilde M$ which acts freely and properly discontinuously. 

%Recalling the fact that $\tilde M$ admits a conformally equivalent metric of constant curvature whose group of conformal automorphisms are simply the Mobius maps
%That is, in the hyperbolic case, we write $M = \D/F$ and $p:\D \to M$ is the canonical projection which is then conformal. In the case of the torus, $p: \C \to M$ is the standard projection from the plane onto the torus and is also conformal.
%$\tilde M/F$ where $\tilde M$ is endowed with a metric of constant curvature (0 if $\tilde M  =\mathbb C$ and $-1$ if $\tilde M = \D$) and $F$ is a discrete subgroup of the M\"obius group.

In case of the torus, this discrete subgroup is isomorphic to $\Z^2$ and the generators specify translations in the two directions of the torus. In the hyperbolic case, this class of subgroups is much bigger and are known as \textbf{Fuchsian groups} (see e.g. \cite{donaldson2011riemann} for a general account). This particular representation of a hyperbolic Riemann surface, sometimes called a Fuschian model, will be particularly convenient because it allows us to describe the scaling limits of Temperleryan or cycle rooted spanning forests \emph{in the universal cover}, rather than on the surface itself. This will allow us to import directly a number of the ideas and result from \cite{BLR16} on the simply connected case. %{\color{cyan}Alternate: Concretely, since in the Fuchsian representation, the dimer and spanning tree lift to just a planar graph embedded in a disc, we can look at this disc with the Euclidean metric in which the sum of angles around a point is $2\pi$ and re-use all the relations between winding and height function proved in the simply connected case in \cite{BLR16} (see also Lemma \ref{lem:int_to_top} below).}

{Concretely, one of the main advantages of the Fuchsian model of a hyperbolic Riemann surface -- and more generally of a lift to the universal cover which inherits the Euclidean metric -- is that, \emph{equipped with this Euclidean metric}, we have a flat embedding in which the sum of angles around a point is $2\pi$, the total winding of a simple loop is $2\pi$ (see also Lemma \ref{lem:int_to_top} below). The Fuchsian structure itself plays a more technical role, see for instance \cref{app:spine}. We rely on this for a canonical description of the height function in terms of the extended Temperley bijection.

To conclude, moving to the universal cover allows us in a sense to trade a natural law on a graph possessing a nontrivial topology, for a law with complicated periodicity-like constraints but on a planar graph.
}
%As we will see in \Cref{app:spine}, the Fuchsian representation will be particularly useful to extract certain topological information about the Temperleyan CRSF. On the other hand, as will be explained in \Cref{sec:winding_ht}, it is enough to calculate winding with the Euclidean metric in the cover (with a proper choice of reference flow), so various results and ideas from the simply connected case \cite{BLR16} can be carried over.
}

\subsection{Intrinsic and topological winding}\label{sec:winding}

The goal of this section is to recall several notions of windings of
 curves drawn in the plane, which we use in this paper. We refer to \cite{BLR16} for a more detailed exposition. A self-avoiding (or simple) curve
 in $\C$ is an injective continuous map $\gamma:[0,T] \to \C$ for some $T \in
[0,\infty]$.
%For now, let us assume $T<\infty$.
%We call  a curve smooth if the map $\gamma$ is smooth (continuously

\medskip The \textbf{topological winding} of such a curve $\gamma$ around a point $p \notin \gamma[0,T]$ is defined as follows. We first
write
\begin{equation}
 \gamma(t)  - p = r(t)e^{i \theta(t)},
\end{equation}
where the function $\theta: [0,\infty) \to \R$ is taken to
be continuous, which means that it is unique modulo a global additive constant multiple of $2\pi$. We define the
winding for an interval of time $[s,t]$, denoted
$W(\gamma[s,t],p)$, to be
$$
W(\gamma[s,t],p) = \theta(t) - \theta(s)
$$ (note that this is uniquely defined).
Notice that if the curve has a derivative at an endpoint of $\gamma$, we can take $p$ to
be this endpoint by
defining
$$W(\gamma[0,t],\gamma(0)) := \theta(t)- \lim_{s \to 0} \theta(s)$$
and
similarly
$$W(\gamma[s,T],\gamma(T)) := \lim_{t
\to T} \theta(t)-\theta(s).
$$

With this definition, winding
is
additive: for any $0 \le s \le t \le T$
\begin{equation}
 W(\gamma[0,t],p) = W(\gamma[0,s],p) + W(\gamma[s,t],p).
\end{equation}

%and in this situation, we shall drop the $p$ from
%the notation and just write $W(\gamma[0,t])$ to be the winding up to time $t$.
%Sometimes for a curve $\gamma:[0,t] \to \C$, we simply write
%$W(\gamma,p)$ to mean the total winding from the starting point to the
%endpoint: $W(\gamma[0,t],p)$.
The notion of \textbf{intrinsic winding} we describe
now, also discussed in \cite{BLR16}, is perhaps a more natural definition of windings of curves. This notion is
the continuous analogue of the discrete winding of non backtracking paths in $\Z^2$
which can be defined just by the number of anticlockwise turns minus the number
of clockwise turns. Notice that we do not need to specify a reference point with
respect to which we calculate the winding, hence our name ``intrinsic" for this
notion.
% \begin{wrapfigure}{R}{.4\textwidth}
%  \centering
% \includegraphics[scale=.58]{topint2.pdf}
% \caption{An example illustrating \cref{lem:intrinsic->top}. The intrinsic
% winding of the curve is $-5\pi/2$ and the topological winding is $-3\pi$. Here
% \note{$\arg ->$?}$\Arg_{\gamma(0) - D} (\gamma(0) - \gamma(t)) = \pi/2$ and
% $\Arg(-\gamma'(0)) =
% 0$.}
% \end{wrapfigure}

%What is  the analogue of this discrete notion in the continous?

We call  a curve smooth if the map $\gamma$ is smooth (continuously
differentiable).
Suppose $\gamma$ is smooth and for all $ t, \,\gamma'(t) \neq 0$. We write
$\gamma'(t) = r_\i(t)e^{i \theta_\i(t)}$ where again
$\theta_\i : [0,\infty) \to \R$ is taken to be continuous. Then
define the intrinsic winding in the interval $[s,t]$ to be
\begin{equation}\label{E:windingsmooth}
W_{\i}(\gamma[s,t]) :
= \theta_\i(t)- \theta_\i(s).
\end{equation}
The total intrinsic winding is again defined to be
$\lim_{t \to T} W_{\i}(\gamma,[0,t])$ provided this limit exists. Note that this
definition does not depend on the parametrisation of $\gamma$ (except for the assumption of
non-zero derivative). The following topological lemma
from \cite{BLR16} connects the intrinsic and topological windings for smooth curves.

% \begin{question}
%  Is the sum of topological winding seen from each point is the same as the
% Intrinsic winding?
% \end{question}

\begin{lemma}[Lemma 2.1 in \cite{BLR16}]\label{lem:int_to_top}
 Let $\gamma[0,1]$ be a smooth {self-avoiding} curve in $\C$ then,
$$
 W_{\i}(\gamma[0,1]) = W(\gamma[0,1],\gamma(1)) + W(\gamma[0,1],\gamma(0)).
$$
\end{lemma}
Note that the right hand side in the above expression makes sense as soon as the endpoints of $\gamma$ are smooth. From now on we therefore extend the definition of intrinsic winding to all such curves using the equality of \cref{lem:int_to_top}. Since the topological winding is clearly a continuous function on the curve away from the point $p$, this extension is consistent with any regularisation procedure.
We also recall the following deformation lemma from \cite{BLR16} (see Remark 2.5).
\begin{lemma}\label{lem:winding_change_conformal}
Let $D$ be a domain and $\gamma \subset \bar D$ a simple smooth curve (or piecewise smooth with smooth endpoints). Let $\psi$ be a conformal map on $D$ and let $\arg_{\psi'(D)}$ be any realisation of argument on $\psi'(D)$. Then
\begin{equation*}
W_{\i}( \psi( \gamma ) )-  W_{\i}(\gamma) = \arg_{\psi'(D)}(\psi'(\gamma(1) ) ) - \arg_{\psi'(D)}( \psi'(\gamma(0) ) ).
\end{equation*}
\end{lemma}

\subsection{Discretization setup}\label{sec:setup}
{We will require the graphs we are working with to be embedded on $M$ in a way that preserves its topology and geometry, we make this precise now.

Let $M$ be a Riemann surface with $b$ holes and $g$ handles satisfying the assumptions of \Cref{sec:surface_embedding}. We say that a graph $\Gamma$ is \textbf{faithfully} embedded in $M$ if it satisfies the following (see \cref{fig:tem_example}):
\begin{itemize}
	\item the embedding is proper, i.e. edges do not cross; 
	\item $\Gamma$ is connected;
 \item Every face of $\Gamma$ not intersecting $\partial M$ has the topology of a disc.
	\item $\Gamma$ has a marked vertex (called a boundary vertex, the set of which is denoted by $\partial \Gamma$) for each component of $\partial M$;
	
\end{itemize}
Note that with this convention, we can think of each boundary vertex as being `delocalised' along the whole boundary component and indeed later we will consider cycle rooted spanning forests (CRSFs) with a wired condition on each boundary component. This `delocalised' boundary vertex is best illustrated to be sitting inside the hole, see \Cref{fig:tem_example}.

Let $\Gamma^{\dagger}$ denote the dual graph of $\Gamma$, defined as follows. Every face of $\Gamma$ contains a single vertex of  $\Gamma^\dagger$ and two vertices of $\Gamma^\dagger$ are connected by an edge if their corresponding faces share an edge of $\Gamma$.  We also assume the following for $\Gamma^{\dagger}$:
\begin{itemize}
\item The embedding is proper, i.e., the edges do not cross.
\item $\Gamma^{\dagger}$ is connected. 
\item Every face of $\Gamma^\dagger$ corresponding to a non-boundary vertex of $\Gamma$ is a topological disc.
\item The vertices of $\Gamma^{\dagger}$ corresponding to the faces of $\Gamma$ which contain a given boundary vertex form a simple cycle homotopic to the corresponding boundary component of $\partial M$. We call this the \textbf{boundary cycle}, and we denote this by $\partial \Gamma^\dagger$.
\end{itemize}
}

\begin{figure}[h]
	\centering
	\includegraphics[width = \textwidth]{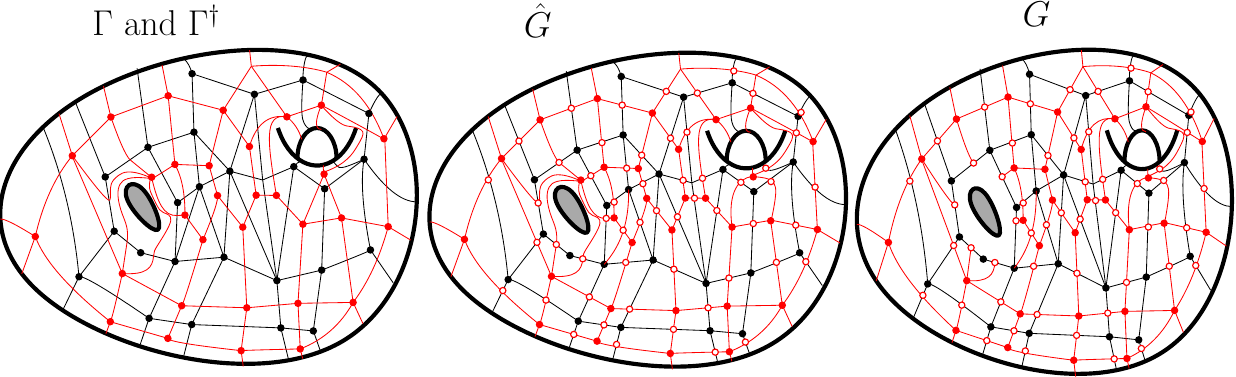}
	\caption{The graphs $\Gamma, \Gamma^\dagger, \hat G$, and $G$. The graph $\Gamma$ is pictured in red and $\Gamma^\dagger$ in black.}\label{fig:tem_example}
\end{figure}

We assign an \textbf{oriented} weight $w(e)= w_\Gamma(e)\ge 0$ to every oriented edge of $\Gamma$. This defines a continuous time Markov chain on $\Gamma$ where for every edge $(x, y)$ in $\Gamma$, the chain jumps from $x$ to $y$ at rate $w( x,y)$, and with a slight abuse of terminology we will from now on refer to this chain as the \emph{random walk} on $\Gamma$. Note that we defined it in continuous time for simplicity but we will in fact always be interested only by the geometry of its paths so its time parametrisation is mostly irrelevant. With a slight abuse of notations, we will often identify the trajectory of the random walk (or other processes such as loop-erased paths) with the continuous path obtained from the set of edges used by the walk, either seen as a path or as a subset of $M$. Similarly, with a slight abuse of notation, we will use direct image notations to denote sub-paths of the random walk, i.e if $s < t$ are two times, we will denote by $X[s, t]$ the path followed by the random walk between times $s$ and $t$. Finally we will always stop the random walk whenever it reaches a boundary vertex so we assume that all oriented edges starting at such a vertex have weight $0$. We emphasize that there are no weights on the dual graph $\Gamma^\dagger$ and indeed we will never consider the random walk on the dual graph.

\begin{remark}To explain the asymmetry between $\Gamma$ and $\Gamma^\dagger$, we point out that $\Gamma$ can be thought of as a graph that is wired at the boundary of $\partial M$ whereas $\Gamma^\dagger$ has free boundary conditions.% In $\hat G$ the extra wired vertices are present and these are removed in $G$.
\end{remark}

In order to apply locally the bijection between dimers and (a variant of) spanning trees, we need to define the superposition of $\Gamma$ and $\Gamma^\dagger$. We introduce a new set of vertices $W$ at each point where some edge $e$ and dual edge $e^\dagger$ intersect. Define the graph $\hat G$ whose vertices are $ V = V(\Gamma) \cup V(\Gamma^\dagger) \cup W$ and whose edges are obtained by joining every vertex in $W$  which is on the edge $e$ with the endpoints of $e$ and $e^\dagger$ (in other words, each pair of edges $e$ and $e^\dagger$ corresponds to four edges in $\hat G$). We call $\hat G$ the \textbf{superposition graph}. Then we let $G$ denote the graph where we remove from $\hat G$ the boundary vertices $\partial \Gamma$ and its corresponding half edges. See \cref{fig:tem_example} for an illustration. Clearly, $G$ is a bipartite graph where we take  $W$ to be the set of white vertices and the rest to be black. Also every non-outer face of $G$ is bounded by $4$ edges (i.e. every non-outer face is a quadrangle). We can naturally define \textbf{un}oriented weights on $G$ from the ones on $\Gamma$. If $w$ is the white vertex in the middle of the edge $(x,y)$ of $\Gamma$, we set $w_G( x,w) = w_\Gamma( x, y)$ and $w_G( y,w) = w_\Gamma( y,x)$ while $w_G(e) =1$ if $e$ is part of an edge of $\Gamma^\dagger$.
{It will be useful later for the definition of the height (see \cref{sec:tree_embedded}) to also choose in each face $f$ of $\hat G$ a \textbf{diagonal} $d(f)$ which is just a curve connecting the primal and dual vertices adjacent to $f$ without exiting $f$, together with a \textbf{midpoint} $m(f)$ which is just a point in the interior of the diagonal. If the primal  vertex adjacent to $f$ is a boundary point choose the diagonal as a path connecting the boundary and the dual vertex. It will be clear from \cref{sec:tree_embedded} that the choice of these diagonals and midpoints only affects some of the arbitrary conventions needed to define the height function and that any statement on the centered height $h - \E( h)$ is independent of this choice.}

{We will prove in \cref{sec:Temp} that if the manifold $M$ is not a torus or an annulus, one must introduce \textbf{punctures} in the graph $G$ for it to even admit dimer coverings. More precisely we will remove finitely many white vertices from $G$ and call a  removed vertex and the edges associated with it a \textbf{puncture} (equivalently, the puncture may be thought of as forcing the vertex to be a monomer instead of in a dimer edge). Note that each puncture creates an octagon  $G$ consisting of four white vertices, two black vertices from $\Gamma$ and two black vertices from $\Gamma^\dagger$ (see \cref{edge_removal}). We will call the resulting graphs respectively $G',\Gamma',(\Gamma^\dagger)'$ and we call the manifold obtained by removing the white vertices $M'$. We will see later in \cref{sec:Temp} that the number of punctures we remove must be $ {\sf k} = |\chi|$ where $\chi$ is the Euler characteristic of $M$. When this is the case we will call $G'$ a \textbf{Temperleyan graph} on $M$ or \textbf{Temperleyan discretisation} of $M$.}

% and we call the resulting graph $G'$. We will call . Removing such a puncture produces an octagon in $G$ consisting of four white vertices, two black vertices from $\Gamma$ and two black vertices from $\Gamma^\dagger$ (see \cref{edge_removal}). Call these new graphs respectively $G',\Gamma',(\Gamma^\dagger)'$. The graph $G'$ resulting from this sequence of operations is what we call a \textbf{Temperleyan graph} on $M$.

%The graph $G$ does not necessarily have a dimer cover, in fact it was shown in \cite{BLR_Riemann2} that a dimer cover is possible if and only if the Euler characteristic $\chi$ of $M$ is strictly negative. Thus, in case the $G$ does not have a dimer cover, we remove $|\chi|$ many white vertices from $G$ (which are macroscopically separated) and we call the resulting graph $G'$. We call the removed white vertex and the edge associated with it a \textbf{puncture}. Removing such a puncture produces an octagon in $G$ consisting of four white vertices, two black vertices from $\Gamma$ and two black vertices from $\Gamma^\dagger$ (see \cref{edge_removal}). Call these new graphs respectively $G',\Gamma',(\Gamma^\dagger)'$. The graph $G'$ resulting from this sequence of operations is what we call a \textbf{Temperleyan graph} on $M$.
%Note that in the case of a torus or an annulus, $2g+b =2$, thus $G'=G$. We denote by $F(G)$ the number of non-outer faces of $G$, i.e., the faces of $G$ which are homeomorphic to open discs. Call the vertices in $\Gamma^\dagger$ corresponding to the outer face of $G$ the \textbf{boundary vertices} of $\Gamma^\dagger$.
\begin{figure}[h]
	\centering
	\includegraphics[scale = 0.5]{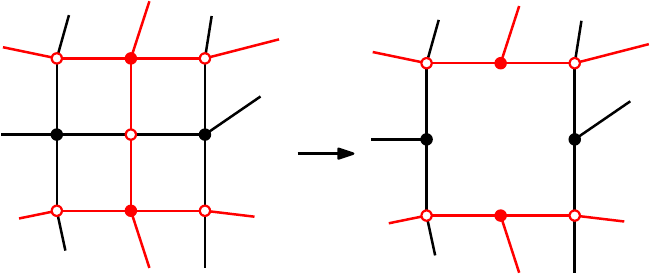}
	\caption{Removing a white vertex (puncture) to create the dimer graph $G'$.}
	%\note{Bad figure! both left and right are hard to understand.}
	\label{edge_removal}
\end{figure}

%We assign \emph{oriented} weights $w_e$ to each edge $e$ in $\Gamma'$. However on the dual graph $(\Gamma^\dagger)'$, we will set $w_e =1$ for every edge in $e$. It is easy to see that this turns $G'$ into a weighted \emph{un}oriented graph. Indeed, if $e = (x,y)$ is an oriented edge of $\Gamma'$, let $w$ denote the white vertex in the middle of $e$. Then we assign to the unoriented edge $\{x,w\}$ of $G'$ the weight $w_{(x,y)}$ and to the edge $\{w,y\}$ of $G'$ the weight $w_{(y,x)}$.

When considering scaling limits, we will of course consider a sequence of graphs $\Gamma^\d, (\Gamma^\dagger)^\d$, {$ G^\d$} from the setup above. We now introduce our assumptions on such a sequence. 
%We emphasise that these only make sense for a sequence so all the statements given for a fixed graph (in particular the bijection from \cref{sec:Temp} and the relation between winding and height function from \cref{sec:winding_ht}) are valid without any further assumption. 

%We now consider a sequence of graphs $\Gamma^\d, (\Gamma^\d)^\dagger$ embedded faithfully on a nice Riemann surface $M$ (see \cref{sec:surface_embedding}), which is not the sphere, the punctured plane or the full plane. Recall also that $\tilde \Gamma^\d$ denote the lift of $\Gamma^\d$ to the universal cover $\tilde M$ (which is either the unit disc $\D$ or the complex plane $\C$). We
%assume the following about $\Gamma^\d$ apart from it being embedded faithfully (see \cref{sec:surface_embedding}). Let $d_M$ denote the Riemannian metric in $M$ and let $d_{\bar M}$ be this metric continuously extended to $\bar M$. We assume that we are given a Markov chain on $\Gamma^\d$ respecting the graph structure which can be non-uniform and nonreversible: in other words, we have some weights $w_{(x,y)}$ for each oriented edge $(x,y)$ in $\Gamma^\d$ and the Markov chain moves from $x$ to $y$ at rate $w_{(x,y)}$ in continuous time. With an abuse of terminology we will call this Markov chain \textbf{random walk} on $\Gamma^\d$. (We think of it in continuous time for simplicity, but all properties of interest to us concern the geometry of its path up to time change, and so the precise time-parametrisation is completely irrelevant.)

Let $p: \tilde M \to M$ be a map which is both analytic and a regular covering of $M$  by its universal cover $\tilde M$, which is either the unit disc $\mathbb D \subset \C$ or the whole plane $\C$ (recall the discussion in  \Cref{SS:universalcover} for the existence of this map.)
 In the end the choice of this covering map does not affect the following assumptions, see \cref{lem:conf_inv_assumption} for a precise discussion about this covering map and also about the covering map for the punctured manifold.

{Recall from \Cref{sec:surface_embedding} that we equipped $M$ with a distance function $d_M$ derived from a Riemannian metric, which extends to the closure $\bar M$; we still denote this extended distance function by $d_M$.}

%\notet{I am not sure what is meant by \emph{the} conformal covering map. After reading some literature, here is what I understand. Any Riemann surface admits a covering by a simply connected space. Any covering space can be endowed with a conformal structure so that the covering map is analytic. Being simply connected, the covering space is conformally equivalent by uniformisation to one of the three types of surfaces. Presumably, you cannot find covering spaces of a given surface $M$ which are conformally equivalent to say the disc and other covering spaces of $M$ which are conformally equivalent to say the plane, though this is not relevant. }

\begin{enumerate}[{(}i{)}]
	\item \label{boundeddensity} \textbf{(Bounded density)} There exists a constant
	$C$ independent of $\delta$ such that for any $x \in M$, the number
	of vertices of $\Gamma^\d$ in the ball $\{z \in M: d_M(x,z) <\delta\}$ is smaller
	than $C$.
	\item \label{embedding} \textbf{(Good embedding)} The edges and diagonals of the graph are embedded
	as smooth curves and for every compact set  $K \subset \tilde M$, the intrinsic winding of every edge or diagonal in the lift $\tilde \Gamma^\d$ intersecting $K$ is bounded by a constant $C=C_K$ depending only on $K$. (Note that this allows edges to wind quite a bit near holes.)
	
	\item \label{InvP} \textbf{(Invariance principle)} As $\delta
	\to 0$, the continuous time random walk $\{\tilde X_t\}_{t \ge 0}$ on
	$\tilde \Gamma^\d$
	started from a nearest vertex to $0$ satisfies:
	$$
	{( \tilde X_{t/\delta^2})_{t \ge 0}} \xrightarrow[\delta \to 0 ]{(d) }
	(B_{\phi(t)})_{t \ge 0},
	$$
	where $(B_t, t \ge 0)$ is a two dimensional standard Brownian motion in $\tilde M$ (killed when it leaves $\tilde M$, if $\tilde M = \D$) started from
	$0$, and $\phi$ is a nondecreasing, continuous, possibly random function
	satisfying $\phi(0) = 0 $ and $\phi(\infty) = \infty$. The above convergence
	holds in law in Skorokhod topology.
	
	We remark that the above condition is equivalent to asserting that simple random walk from some fixed vertex converges to Brownian motion on the Riemann surface itself up to time parametrisation (see e.g. \cite{hsu2002}).

	\begin{figure}[h]
		\centering
		\includegraphics[scale = 0.5]{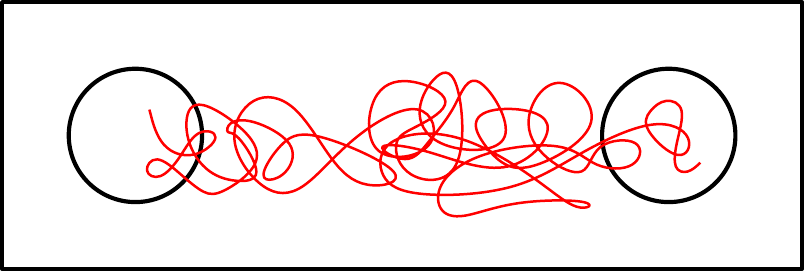}
		\caption{An illustration of the crossing condition.}\label{fig:crossing}
	\end{figure}

	\item \label{crossingestimate} \textbf{(Uniform crossing estimate).}
	Let $R$
	be the horizontal rectangle $[0,.3]\times
	[0,.1]$ and $R'$ be the vertical
	rectangle $[0,.1]\times [0,.3]$.
	Let $B_1 :=
	B((.05,.05),0.025)$ be the
	\emph{starting ball} and $B_2:=
	B((.25,.05),.025) $ be the \emph{target ball}
	(see \cref{fig:crossing}).

	The uniform crossing condition is the following. There exist universal constants $\delta_0,\alpha_0>0$ such that for every compact set $K  \subset  M$, there exists a $\delta_K$ such that for all $\delta \in (0,\delta_K)$  the following is true. Let $\tilde K = p^{-1}(K)$ be the lifts of $K$. Let $R''$ be a set of the form $cR+z$, where $c \ge \delta/\delta_0$ and $z \in \R^2$ (i.e. a scaling and translate of $R$), and is such that $R'' \subset \tilde K$.

%	{\color{gray}subset of one of the connected components of $\tilde K$ and is a translation of $cR$ or $cR'$ where $c \ge \delta/\delta_0$.} 
	
	Let $B_1'' = cB_1+z$ and  $B_2'' = cB_2+z$. For all $v \in \tilde \Gamma^\d \cap B_1''$,
	\begin{equation}
	\P_{v}(\tilde X \text{ hits }B''_2 \text{ before exiting } R'')
	>\alpha_0.\label{eq:cross_left_right}
	\end{equation}
	{We emphasise that this crossing condition is defined in the Euclidean metric in the disc, and not the (perhaps more standard) hyperbolic metric when dealing with the universal cover of a hyperbolic surface.

	In what follows, sometimes for a compact set $S \subset \tilde M$, we will write $\delta_S$ to mean $\delta_{p(S)}$ as defined above.}
	%\note{I think maybe we need RSW for the lift of the punctured manifold also.}
	%Note that this assumption also entails that the same crossing condition holds for all the preimages: $\{(p^{-1}(p(R)),p^{-1}(p(B_1)),p^{-1}(p(B_2))\}$ and $\{(p^{-1}(p(R')),p^{-1}(p(B_1')),p^{-1}(p(B_2'))\}$ for the same choice of $\delta$.
	%
	
\item \textbf{(punctures)} \label{punctures} We remove $\sf k = |\chi|$ many points from (the interior of) $M$ and call the resulting surface $M'$, where $\chi$ is the Euler characteristic of $M$. We assume that each of the ($\sf k $ many) monomers (white vertices removed from $\hat G$ to obtain $G'$) converge to a unique puncture as $\delta \to 0$ which is a given point in the interior of the Riemann surface. { We also assume without loss of generality that 
for all $\delta$, the punctures are sufficiently far (say at graph distance at least some large fixed constant) from the boundary.}

{ Finally, we also suppose that the following holds. Let $u_i,v_i$ be the endpoints of the edge of $\Gamma$ corresponding to the white monomer removed for $1 \le i \le \sf k$. Then there exist paths $\gamma_{u_i}, \gamma_{v_i}$ (viewed as a ordered collection of adjacent edges $((x_1,x_2), (x_2,x_3),\ldots, (x_{r-1},x_{r}))$ or as a continuous path depending on context) satisfying the following. The paths start with $x_1 = u_i$ or $x_1=v_i$ respectively for $\gamma_{u_i}, \gamma_{v_i}$, and we assume
\begin{equation}
M \setminus \bigcup_{1 \le i \le \sf k} (\gamma_{u_i} \cup \gamma_{v_i } \cup (u_i,v_i))\label{eq:necessary_cond}\end{equation}
is topologically a disjoint union of annuli.

% for $1 \le i \le \sf k$ satisfying the following. The path  $\gamma_x$ is viewed as a ordered collection of adjacent edges $((x_1,y_1), (x_2,y_2),\ldots, (x_r,y_r))$  starting from either $x_1 = u_i$ or $x_1=v_i$, $(i = 1, \ldots, \sf k)$. We identify $\gamma_x$ with a continuous path on $M$,  
%\begin{itemize}
%\item  Either $\gamma_x$ is an oriented path starting at $x$ and every vertex on this path has exactly one outgoing edge,
%\item or, $\gamma_x$ is an oriented path started at $x$ and ending in $\partial \Gamma$.
%\end{itemize}
%In the former case, $\gamma_x$ necessarily consists of a simple oriented cycle along with an oriented simple path connecting $x$ to the cycle.
%In the latter case, we add a segment to $\gamma_x$ disjoint from everything else connecting its endpoint to $\partial M$, and call the new path $\gamma_x$ as well admitting an abuse of notation. We require:

\begin{figure}[h]
		\centering
		\includegraphics[scale = 0.5]{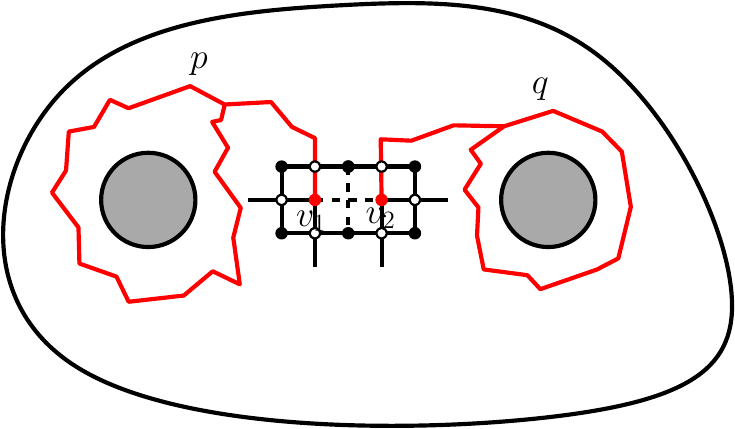}\hspace{.5cm}
		\includegraphics[scale = 0.5]{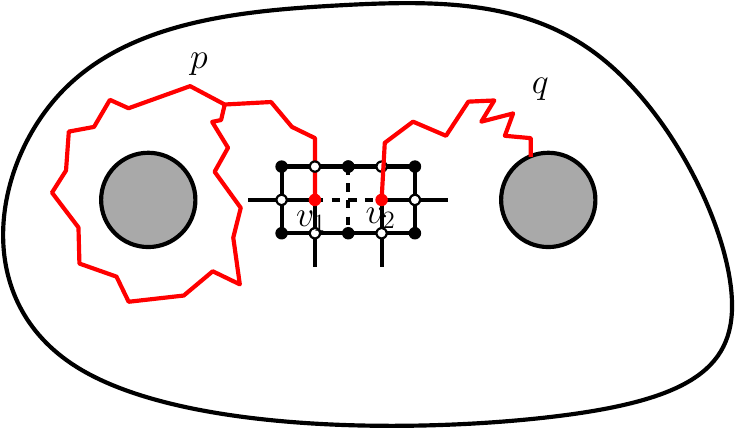}
		\caption{ {An illustration of the assumption in \Cref{punctures}. The paths $\gamma_{u_1}, \gamma_{v_1}$ consists of $p$ and $q$, which decompose the surface into a number of disconnected annuli: three in the first case, two in the second.
			The dotted edges are the ones removed from $G$ to get $G'$.}} \label{F:pants}
\end{figure}
}

%\note{Add a remark saying : this is not actually an assumption in the limit so this is only relevant for the combinatorics for ``large'' delta.}

	%\textcolor{gray}{\item \label{boundary} \textbf{(Boundary convergence).} In case $\partial M \neq \emptyset$, recall that the set of boundary cycles $(\partial \Gamma^\dagger)^\d$ corresponds to the connected components of $\partial M$. We assume that each boundary cycle converges in the Hausdorff metric (induced by $d_{\bar M}$) to the associated component of $\partial M$. }\note{Isn't that implied by the CLT ?}\noteb{Agreed. Made this into an observation below from an assumption. G.}	%In case $M$ is not simply connected, we assume each component of the boundary forms a non-contractible simple loop. Also if $M$ does not have a boundary then $\partial \Gamma^\d = \emptyset$ for all $\delta$.
	
	%\item \label{beurling} \textbf{(Beurling type estimate).} For all $\ve,\ve'>0$ there exists $\eta>0$ such that for any $\delta <\delta(\eta)$ and for any vertex $v \in G^\d$ such that $d_M(v,\partial M) <\eta$, the probability that a simple random walk exits $B_M(z,\ve ) := \{z\in M: d_M(z,v) <\ve\}$ before hitting $\partial G^\d$ is at most $\ve'$.
\end{enumerate}

{\begin{remark}
Typically the existence of the paths $\gamma_{u_i}$ and $\gamma_{v_i}$ for $\delta $ small enough is guaranteed by the crossing condition and the \emph{pants decomposition} of $M$. In words, the paths decompose the surface into disjoint annuli, see Figure \ref{F:pants}. However, it is possible to construct pathological examples (for large $\delta$) where there is no paths such that \eqref{eq:necessary_cond} holds, so this needs to be made as a separate assumption. In fact, we will see that (v) is equivalent to the graph $G$ admitting a dimer cover (see Figure \ref{fig:tem_example}), so we could equivalently assume the latter instead -- this would be more elegant but less concrete.
\end{remark}	
}

{\begin{lemma}[Conformal invariance of assumptions]\label{lem:conf_inv_assumption}
The assumptions (\ref{embedding}), (\ref{InvP}), (\ref{crossingestimate}) on the sequence $(\Gamma^\d)_{\delta>0}$ are invariant with respect to the choice of the covering map $p$, in the sense that if $p'$ is any other analytic and regular covering of $M$ then these assumptions are satisfied for some (possibly different) choices of $\delta_0,\alpha_0,(\delta_K)_{K \subset M}, (C_K)_{K \subset \tilde M}$.

Furthermore, if $(\Gamma^\d)_{\delta}$ satisfies the assumptions (\ref{boundeddensity})-(\ref{punctures}) for a choice of covering map $p$ and constants $C,\delta_0,\alpha_0, (\delta_K)_{K \subset M}, (C_K)_{K \subset \tilde M}$ then these assumptions are invariant with respect to conformal transformations in the sense that if $\psi: (M,x_1,x_2,\ldots, x_{\sf k}) \to (N,x'_1,x_2',\ldots, x'_{\sf k})$ is a conformal bijection mapping $M$ to $N$ and $x_i$ to $x'_i$ for $1 \le i \le \sf k$, then
$(\psi(\Gamma^\d))_{\delta>0}$ also satisfies assumptions (\ref{boundeddensity})-(\ref{punctures}) for the choice of the covering map $p\circ \psi$ and constants $C,\delta_0,\alpha_0, (\delta_{\psi^{-1}(K)})_{K \subset M'}$, and $(C_K)_{K \subset \tilde M}$.
\end{lemma}
\begin{proof}
We start with the second of these two assertions, for which items (\ref{embedding})-(\ref{punctures}) are trivial. For (\ref{boundeddensity}), if $M$ has no boundary then $d_M (x,y) = d_N (\psi(x), \psi(y))$ (essentially because M\"obius maps are isometries with respect to the hyperbolic metric on the unit disc). Thus (i) holds. 
If $M$ has a boundary, then note that the conformal bijection $\psi$ between $M$ and $N$ extends to a conformal bijection $\bar \psi$ between the doubles of $ M, N$; and we are back to the previous case.

Let us now prove the first assertion. Note that by the uniformisation theorem, there exists a conformal bijection (M\"obius map) $\phi: \tilde M \to \tilde M$ such that $p\circ \phi = p'$. Note $\phi$ is a M\"obius map from the unit disc to itself in the hyperbolic case and from the complex plane to itself in the torus case. Assumption (\ref{embedding}) is trivially preserved in the torus case as M\"obius maps are translations and rotations. For the hyperbolic case, we can employ the change in winding under conformal map formula (\Cref{lem:winding_change_conformal}), and note that derivatives of M\"obius maps are bounded on compact subsets of the open unit disc.  In assumption (\ref{InvP}), we only require convergence of the random walk up to time change which is preserved under M\"obius maps.  Assumption (\ref{crossingestimate}) is easily seen to be preserved in the  torus case as a M\"obius map maps a rectangle to another rotated and translated rectangle, which can be crossed by concatenating bounded number of vertical and horizontal rectangles of smaller scales. In the hyperbolic case, a rectangle is mapped to a domain bounded by four circular arcs, and the starting and target discs are mapped to discs inside this domain. Furthermore, compact sets are mapped to compact sets. Thus concatenating domains of this type, it is easy to see that uniform crossing estimate holds for $M$, perhaps for different positive constants $\delta_0, \alpha_0, (\delta_K)_{K \subset M}$.
\end{proof}
}
   % Let us now prove the first assertion. Note that by the uniformisation theorem, $p\circ p'^{-1}$ is a conformal bijection between the universal covers, which is a M\"obius map from the unit disc to itself in the hyperbolic case and from the complex plane to itself in the torus case.  Assumption (\ref{embedding}) is trivially preserved in the torus case as M\"obius maps are translations and rotations. For the hyperbolic case, we can employ the change in winding under conformal map formula (\Cref{lem:winding_change_conformal}), and note that derivatives of M\"obius maps are bounded on compact subsets of the open unit disc. 
   % In assumption (\ref{InvP}) we only require convergence of the random walk up to time change which is preserved under M\"obius maps.  Assumption (\ref{crossingestimate}) is easily seen to be preserved in the  torus case as a M\"obius map maps a rectangle to another rotated and translated rectangle, which can be crossed by concatenating bounded number of vertical and horizontal rectangles of smaller scales. In the hyperbolic case, a rectangle is mapped to a domain bounded by four circular arcs, and the starting and target discs are mapped to discs inside this domain. Furthermore, compact sets are mapped to compact sets. Thus concatenating domains of this type, it is easy to see that uniform crossing estimate holds for $M$, perhaps for different positive constants $\delta_0, \alpha_0, (\delta_K)_{K \subset M}$. 

	{\begin{remark}\label{rmk:start_point}
		The Invariance principle assumption (\cref{InvP}) actually implies something stronger in combination with the other assumptions: for any point $x$ in $\tilde M$, the random walk started from a vertex $x^\d$ nearest to $x$ converges to a Brownian motion started from $x$ up to a time change as above. This is a consequence of the fact that random walk from 0 comes close to $x$ with uniformly positive probability using the crossing estimate and the strong Markov property of Brownian motion.
	\end{remark}
	}
	
In case $\partial M \neq \emptyset$, recall that the set of boundary cycles $(\partial \Gamma^\dagger)^\d$ corresponds to the connected components of $\partial M$. One consequence of the Invariance principle assumption (\cref{InvP}) is that each boundary cycle converges in the Hausdorff metric (induced by $d_{ M}$) to the associated component of $\partial M$.%\noteb{Not sure why this follows from CLT. N.}\note{Brownian motion on a manifold with boundary hits the boundary almost surely. On the other hand we assume that in the universal cover, random walk converges to a Brownian motion on the unit disc. Combining the two yields the result. What are we missing? G.}

Sometimes we drop the superscript $\delta$ from $\Gamma^\d, (\Gamma^\dagger)^\d$ for clarity, when there is no possibility of a confusion.

%\begin{comment}
%commented out by NB on Feb 12, 2023.
%By the uniformisation theorem of Riemann surfaces, we know that there exists a conformal map from the Riemann surface $\tilde M/ F$ to $M$ where $F$ is a Fuchsian group which is a discrete subgroup of the group of M\"obius transforms on $\tilde M$\footnote{$F$ is discrete if and only if for every $x \in M$, $\exists$ a neighbourhood $V $ of $x$ so that $fV \cap V \neq \emptyset$ for finitely many $f \in F$.}. In the case of the torus, this subgroup is simply a group of translations isomorphic to $\Z^2$. In the hyperbolic case $F$ is a subgroup of the group of M\"obius transforms of the unit disc $\D$. Such a conformal map is unique up to conformal automorphisms (i.e. M\"obius transforms) of the unit disc. In other words, if $F,F'$ are two Fuchsian groups such that $M$ is conformally equivalent to both $\D / F$ and $\D/ F'$ then there exists a M\"obius map $\phi: \D \mapsto \D$ such that $F'  = \phi^{-1}\circ F \circ \phi$. Since we have fixed a {\color{blue} covering map} $p$, we have defined $F$ uniquely.
%commented out by NB on Feb 12, 2023. This used to be in the remark below. 
%\end{comment}	

\begin{remark}\label{rem:lift} 
	%{\color{gray}We also remark that in \cref{InvP}, we only require convergence up to time change, and hence this assumption depends only on the conformal type of the metric.
	%Finally, one could be worried about the fact that in \cref{crossingestimate}, the probability $\alpha_0$ is uniform over the position and scale of the rectangle despite the distortion between two copies of the same set $K \subset M$. However note that the image of a rectangle by a M\"obius transform is made of $4$ circular arcs and is crossed by Brownian motion with the same probability as the original rectangle, so our assumption is natural in this sense.}

 In the hyperbolic case, note also that while we have stated the assumptions~(\ref{embedding}), (\ref{InvP}) and (\ref{crossingestimate}) on the universal cover of $M$, {these assumptions are also valid for the sequence $(\Gamma^\d)_{\delta>0}$ on the universal cover of the punctured surface $M'$ as well if we endow it with the metric it inherits from $M$ via inclusion.} This can be checked using the fact that there is a map from the universal cover of $M'$ to $\tilde M \setminus p^{-1} (\{x_1, \ldots, x_{\sf k}\})$ (where $x_1, \ldots, x_{\sf k}$ is the set of punctures) which is locally a conformal bijection. { Also note that (\ref{boundeddensity}) is also satisfied in $M'$ trivially by the choice of the metric.} 
 	
	%{\color{gray}In summary, we could work with any choice of {\color{blue}(regular, analytic) covering map} both for the punctured and the non-punctured manifold and we fix a particular choice of this map and call it $p$ (with a small abuse of notation we will use $p$ throughout in both the punctured and unpunctured cases).} 
\end{remark}

%\note{Compare the lift in punctured manifold and the original one.}

\subsection{Height function and forms}\label{sec:ht}
  A \textbf{flow} $\omega$ is a real valued antisymmetric function on the oriented edges $\vec{E}$ of $G$, i.e., for every oriented edge $(u,v)$, $\omega(u,v) = -\omega (v,u)$.  The total flow out of a vertex $v$ is defined to be $\sum_{w \sim v} \omega(v,w)$. Similarly, the total flow into a vertex $v$ is defined to be $\sum_{w \sim v}  \omega (w,v)$. A flow $f$ is a \textbf{closed 1-form} if the sum over any oriented contractible cycle is 0: i.e.,
% More precisely, a function $f:\overrightarrow E \mapsto \R$ is a flow if for any oriented edge $\overrightarrow e = (e^-,e^+)$,
%$$
%f(e^-,e^+) = - f(e^+,e^-).
%$$
%The flow $f$ is also a closed 1-form if
for any oriented cycle $(v_0,v_1,\ldots, v_n=v_0)$ in $G$ so that the embedding of $\cup_{i=0}^{n-1}(v_i,v_{i+1})$ in $M$ forms a contractible loop,
$$
\sum_{i=0}^{n-1} \omega(v_i,v_{i+1}) =0.
$$
It is clear that if $M$ is simply connected, then there exists a function $f$ on the vertices of $G$ (uniquely defined up to a global constant) such that $f(v)-f(u) = \omega(u,v)$ for all $u \sim v$.

We now associate to any dimer configuration $\bf m$ on $G$ a closed 1-form on $\vec{E}$. Let $\bf m$ be a dimer configuration on $G$, and let $\vec e = (w,b)$ be an oriented edge, where $w$ is a white vertex and $b$ a black vertex.
We define the flow $ \omega_{\bf m}$ by setting $\omega_{\bf m} (\vec e) =1_{\{ e \in \bf m\} }$. Also, $\omega_{\bf m}$ is defined in an antisymmetric way: $\omega_{\bf m}( (b,w)) = - \omega_{\bf m} ((w,b))$.
Note that the total flow out of a white vertex is 1 and that out of a black vertex is -1.

To any flow $\omega$ on oriented edges, one can associated a dual flow $\omega^\dagger$ defined on the oriented edges of the dual graph $G^\dagger$, where if $e^\dagger$ crosses the edge $e = (w,b)$ with $w$ on its right and $b$ on its left, then we set $\omega^\dagger (e^\dagger) = \omega(e)$.  Note also that if $\omega$ is divergence free (i.e., the flow out of every vertex is $0$), then $\omega^\dagger$ is a closed 1-form on $\vec{E^\dagger}$.

Consider any \textbf{reference flow} $\omega_0$ which
has total flow out of white vertex equal to $1$ and total flow out of a black vertex equal to $-1$. Then
$\omega = \omega_{\bf m} - \omega_0$ defines a divergence free flow on $\vec{E}$.
%
%One can define a closed 1-form $\omega=\omega(G,\bf m)$ on the dual of $G$. More precisely, for any oriented edge $(u,v)$ in
%$G$, let $(u',v')$ denote the edge in $G^\dagger$ which crosses
%$(uv)$ from right to left and define $\omega(v',u') := (\omega_{\bf m}
%- \omega_0 )(v) - (\omega_{\bf m} - \omega_0)(u)$. It is straightforward to see that $\omega$ is indeed a closed 1-form since $\omega_{\bf m} - \omega_0$ is divergence free.
We call $\omega^\dagger$ the \textbf{height 1-form} corresponding to $\bf m$ with reference flow $\omega_0$.

When $G$ is embedded on {a simply connected domain (so} that, in particular, no cycle in $G^\dagger$ is non-contractible), every closed 1-form $\omega$ on $\vec{E^\dagger}$ becomes exact: i.e., there exists a function on the faces $F(G)$ of $G$,
$h: F(G) \mapsto \R$, so that for any two adjacent faces $f,f'$,
$$h(f' ) - h(f) = \omega(f,f').$$ Observe further that this function is defined only up to a global constant. The function $h$ is then called the \textbf{height function} of the dimer $\bf m$, admitting an abuse of terminology.

We recall the following simple but useful observation.
A \textbf{path} in $G$ (or $G^\dagger$) is a sequence of vertices $(v_0,\ldots, v_n)$ (or faces $(f_0,\ldots, f_n)$) of $G$ so that $v_i$ is adjacent to $v_{i+1}$ (or $f_i$ is adjacent to $f_{i+1}$ in $G^\dagger$) for all $0 \le i \le n-1$.
\begin{lemma}[Unique path lifting property]
Let $\gamma = (f_0,f_1,\ldots, f_n)$ be a path (not necessarily simple) in $G^\dagger$. Let $\tilde f_0$ be a lift (i.e. one pre-image) of $f_0$ to $\tilde M$. Then there exists a unique path $\tilde \gamma = (\tilde f_0,\tilde f_1,\ldots, \tilde f_n)$ in $\tilde M$ so that $\tilde f_i$ is the lift of $f_i$. Further, $\sum_{i=0}^{n-1} \omega(\tilde f_i, \tilde f_{i+1}) = h(\tilde f_n ) - h(\tilde f_0)$.
\end{lemma}

\medskip We now turn to the definition of height function in the more complicated case when $M$ is no longer assumed to be simply connected. In that case, when we sum the values of the height 1-form (defined above)
 along any non-contractible cycle, we may get a nonzero value. One can use the \textbf{Hodge decomposition theorem}, to isolate out the part of the height 1-form which is encoded by the topology of the underlying surface. The Hodge decomposition theorem works in great generality, but in the present context, it takes the following simple form. For any function $f$ on the vertices of $G$ we define $df$ to be the closed 1-form defined on $\vec{E}$ as
$$
df(u,v) = f(v)-f(u).
$$
A  \textbf{harmonic 1-form} $\h$ is a closed 1-form which is divergence free, so that $\sum_{v\sim u} \h(u,v) =0$.
\begin{thm}[Hodge decomposition \cite{aru2015,bott,mercat07}]\label{thm:hodge}
For any closed 1-form $\omega$ on $G$ (or $G^\dagger$),  there exist a function $f$ on the vertices of $G$ and a harmonic 1-form $\h$ defined on $\vec{E}$ such that
$$
\omega = df +\h,
$$
and $f$ is unique up to an additive global constant, and $\h$ is unique. Furthermore, $\h$ is completely determined by summing $\omega$ over a finite set of oriented non-contractible cycles which forms the basis of the first homology group of $M$.
\end{thm}
%\note{Hodge decomposition is supposed to decompose $L^2(E)$ into orthogonal components. Does that tell us anything more about the harmonic form?
%
%One thing is that if we prove that the joint law is Gaussian, independence should follow for free.}

In this paper, we will analyze $(f,\h)$ corresponding to the divergence free flow $\omega_{\bf m} - \omega_0$, where $\bf m$ is dimer configuration chosen from the law \eqref{Gibbs} subject to certain natural conditions and $\omega_0$ is a carefully chosen reference flow (in fact, we will consider the height function see \cref{sec:ht_convergence} for a precise statement). We will call $\h$ the \textbf{instanton component}. We remark that changing the reference flow changes $(f,\h)$ by a deterministic additive factor, and in particular does not affect the fluctuations of $(f, \h)$ around their mean. To be more precise, our main theorem (\cref{thm:main_precise}) will be stated in terms of the single-valued function associated to $\omega_{\bf m} - \omega_0$ on the universal cover of $M'$. See Remark \ref{rmk:instanton_jt_conv} for a statement concerning both the scalar and instanton parts of the height function.
%\note{I think this is fine even though not strictly accurate.}

%\subsection{Height function on the universal cover}\label{sec:uni_cover}

Throughout the paper, rather than working with the scalar and instanton components of the height 1-form, it will be more convenient to lift the height 1-form $\omega$ to the universal cover of $M$. Since the latter is always simply connected, this allows us to work with actual functions without having to worry about the Hodge decomposition \cref{thm:hodge}. We will then check that the convergence of height function on the universal cover implies convergence of each of the components in the Hodge decomposition.

Our assumptions on the graph $G$ where the dimer model lives are such that
$\tilde G = p^{-1}(G)$ is a planar graph embedded on $\tilde M$.
Moreover, the height 1-form $\omega$ on the dual edges of $G$ lifts to a height 1-form $\tilde \omega$ on the dual edges of $\tilde G$. Since $\tilde M$ is simply connected, and since $\tilde \omega$ is a closed one-form on the dual edges of $\tilde G$ (this is a local property, so remains true when we lift to $\tilde G$), we can define a height function $h = h(\boldsymbol{m}, G)$ (up to a global constant) on the dual graph $\tilde G^\dagger$.
%The height function and the height 1-form are connected via \cref{thm:hodge}.
The instanton component $\h$ can be related to the height function $h$ on the universal cover by summing up the value of $\tilde \omega$ along any path in the dual graph of $\tilde G$ corresponding to a non-contractible loop in the dual edges of $G$. This is easier to explain on an example.

\begin{example}
If $M $ is the flat torus $\mathbb T := \C / (\Z + \tau \Z)$ for some complex number $\tau $ with $\Im (\tau) >0$, then the universal cover is the complex plane $\C$. The universal cover can be thought of as many copies of the fundamental domain (a parallelogram determined by $1$ and $\tau$).

Fix $v_0$ in the fundamental domain. Then by periodicity of $dh$, the height function $h$ on $\tilde G$ evaluated at a point $v = v_0 + m +\tau n$ (where $m,n \in \Z$) is given by
$$
h(v) = h(v_0  + m+in) = h(v_0)+\ms{a} m+\ms{b} n,
$$
for some $\ms{a}, \ms{b} \in \R$ which do not depend on either $v_0, m,n$. Let us describe what $\ms a, \ms b$ are. Consider
the two loops on the torus described by $L_1 := (t+1/2\tau: t \in [0,1])$ and $L_2:= (1/2 + t\tau: t \in
[0,1])$ in the fundamental domain (i.e., $L_1$ and $L_2$ are the two non-contractible loops in
 the torus which form the basis of the homology group). Then $\ms a$ is the sum of the values of $\omega $ along any loop in $\vec{E^\dagger}$ which is homotopic to $L_1$, whereas $\ms b$ is the sum of the values of $\omega$ along any loop which is homotopic to $L_2$. Clearly, the choice of
these curves in $G^\dagger$ do not matter since the height 1-form is closed. Furthermore, in the Hodge decomposition of \cref{thm:hodge}, the harmonic 1-form $\h$ is uniquely determined by the numbers $\ms a$ and $\ms b$.
\end{example}

\section{ Temperley's bijection on Riemann surfaces}\label{sec:Temp}

\subsection{Notion of Temperleyan cycle rooted spanning forest; bijection}
\label{SS:Tempforest}
%\note{Changed $\cT$ to $t$ and the reserving the notation $\cT$ for random CRSF as per referee's advice.}

The goal of this section is to extend the classical Temperley bijection to graphs embedded in surfaces of higher genera. The bijection is between dimer covers of certain natural class of graphs introduced in \Cref{sec:setup} and certain objects which we call \emph{Temperleyan forests}.
The main result is formulated in \Cref{prop:temp_bij}.

Before defining the objects in question, let us briefly recall the classical Temperley bijection on simply connected surfaces \cite{CohnKenyonPropp}. The bijection map is local: given a dimer cover, one `extends' (in the direction from black to white) the dimer to an oriented edge as depicted in \Cref{F:Temp_bij}. This yields two oriented trees which are dual to each other. In the opposite direction, given a pair of trees which are dual to each other, one can naturally orient them towards a pre-assigned `root vertex'; the dimer cover then consists of the first half of each oriented edge in the two dual trees. {The primary issue that needs to be dealt with in higher genus is that the relevant objects are no longer pairs of trees that are dual to each other, but pairs $(t, t^\dagger)$ of \textbf{arborescences} (i.e., collection of oriented edges such that there is a single outgoing edge out of every vertex, except for boundary vertices) on $\Gamma$ and $\Gamma^\dagger$ respectively, dual to each other, and where each component contains a non-contractible oriented cycle. Thinking only about the arborescence $t$ on the primal graph, which we aim to ultimately understand via a variant of Wilson's algorithm, the requirement that its planar dual $t^\dagger$ can be endowed with the orientation of an arborescence imposes
nontrivial topological constraints on $t$ which will be elucidated below.}

%The primary issue that needs to be dealt with in higher genus is that the relevant objects are no longer pairs of trees that are dual to each other, but pairs of forests dual to each other where each component contains a non-contractible cycle. The cycles in the dual do not come with a natural orientation, and hence the extension needs additional work.}

Let $\Gamma$ be a graph faithfully embedded on a surface with a certain specified set of boundary vertices. Assume every edge $e$ of $\Gamma$ comes with a specified weight $w(e)>0$. Before introducing the notion of Temperleyan forest, we start with the simpler notion of cycle rooted Spanning Forest.
 
\begin{defn}\label{def:CRSF}
	A \textbf{wired oriented cycle rooted spanning forest} (which we abbreviate: wired oriented CRSF) of $\Gamma$ with the specified boundary is an oriented subgraph $t$ of $\Gamma$ where
	\begin{itemize}
		\item Every non-boundary vertex of $\Gamma$ has exactly one outgoing edge in $t$. Every boundary vertex has no outgoing edge. (As a result, any cycle of $t$ must be consistently orientated).

		\item Every cycle of $t$ is non-contractible.
		
	\end{itemize}
	To each CRSF $t$, we may associate the weight $\prod_{e \in t} w(e)$. We will consider the probability measure on CRSFs such that the probability of $t$ is by definition proportional to its weight.
\end{defn}

This is equivalent to the notion of essential CRSF on a graph with wired boundary introduced by Kassel and Kenyon \cite{KK12}.
%although their non-contractible CRSF has free boundary conditions whereas as the name suggests the above object is wired at the boundaries. Moreover the orientation is important and the dual of the object they consider need not be wired.
Ignoring the orientation of $t$ gives an unoriented graph, its connected components will simply be called the \textbf{connected components} of $t$ without any additional precision. Note that if $t$ is a wired oriented CRSF, every connected component of $t$ contains at most one cycle: more precisely, every boundary component must have zero cycles, while every non-boundary component contains exactly one cycle.

%\item There is a connected component corresponding to every connected component of the boundary of the graph. In each such component, if we glue together all the boundary vertices into a vertex $v$, we obtain an oriented tree with each edge oriented towards $v$. \note{NB: we think this should be removed. Not sure what it means in fact...}
%\end{itemize}

We will refer to the set of all non-contractible cycles of a wired oriented CRSF to mean the set of unique cycles corresponding to each (non-boundary) component of the wired oriented CRSF.

\medskip
{
Let us come back to the setup of {\cref{sec:setup}} and recall that we had the graph $\Gamma$, its dual $\Gamma^{\dagger}$ and the superposition graph $G$ all embedded nicely in a Riemann surfaces with $g$ handles and $b$ holes. Compared to the planar setting, there is an immediate topological difficulty, which is that this graph $G$ in general does not admit a dimer cover. Indeed by Euler's formula, if $G$ has $v$ vertices, $e$ edges and $f$ many contractible faces, then $v-e+f = \chi := 2-2g-b$ (where $\chi$ is the Euler's characteristic). On the other hand if $G$ admits a dimer cover, one must have $e = v+f$ since $e$ is the number of white vertices which must match $v+f$ (the number of vertices and contractible faces of $\Gamma^{\dagger}$).
Thus we need to remove $\sf k = |\chi|$ many edges from the superposition graph for it to have any chance of having a dimer cover (see \cref{edge_removal}). These removed edges can be thought of as creating \textbf{punctures} in the surface. Call the graph with these punctures $G'$. (Recall that, as per \cref{punctures}, we assume throughout that the punctures are macroscopically far apart in $M$, and are away from the boundary as well, and converge as $\delta \to 0$ to a set of pairwise distinct points in $M$).
Note that if $M$ is a torus or an annulus then $\chi = 0$ so no punctures need to be removed.
See also Ciucu--Krattenthaler \cite{ciucu} and Dub\'edat \cite{dubedat_torsion}  for other situations where punctured dimers arise.
}

\medskip %{We now come to the crucial definition of this paper.
%, in which we introduce the notion of Temperlayan forest. Recall that Temperleyan forests will be shown to be in bijection with a dimer cover on $G'$ (and the bijection will also verify as a corollary that $G'$ indeed has a dimer cover).}
To every wired oriented CRSF $t$ of $\Gamma'$ with boundary $\partial \Gamma$, one we can associate a natural dual \textbf{free cycle rooted spanning forest} $t^\dagger$ (abbreviated free CRSF) of $(\Gamma^\dagger)'$ as follows. The  vertices of $t^\dagger$ are given by the vertices of $(\Gamma^\dagger)'$ (i.e., it spans $(\Gamma^\dagger)'$) and an edge $e^\dagger$ is present in $t^\dagger$ if and only if its dual $e$ is absent in $t$. Note that \emph{a priori} $t^\dagger$ does not come with an orientation and that its cycles can overlap (or equivalently a component might contain several cycles), see \cref{F:temperleyan} for an example. {In fact, it is not hard to see that in cases similar to \cref{F:temperleyan}, it is not possible to consistently orient the free CRSF, so there is no chance to apply any version of Temperley's bijection. This motivates the introduction of the following definition.}
%This is 
%the reason why not every CRSF is associated with a dimer configuration (recall that in the classical simply connected case both the primal and the dual trees come with a canonical orientation and the dimer configuration is obtained from the pair of dual oriented spanning trees in Temperley's bijection by placing a dimer on the ``first half" of each oriented edge in both trees). The following adds a condition to the CRSF which allows to assign an unambiguous orientation to the dual $t^\dagger$ of $t$. 
%definition is crucial for the rest of the paper and allows us to extend the theory to the setting of multiply connected surfaces.

\begin{defn}\label{D:temp}
	We say that the wired oriented CRSF $t$ is \textbf{Temperleyan} if each connected component of $t^\dagger$ contains exactly one cycle.% \note{Again, isn't this exactly one?} %Let
	%$\crsf^\dagger$ be the set of wired Temperleyan oriented CRSF on $(\Gamma^\dagger)'.$ % \begin{lemma}\label{L:orientation_dimer}
	% Suppose, we find an orientation of edges of $\Gamma$ and $\Gamma$
	% \end{lemma}
\end{defn}
%\note{NB shouldn't it be contain rather than consist?}
An example of a wired oriented CRSF $t$ that is \emph{not} Temperleyan is provided in Figure \ref{F:temperleyan}.

\begin{figure}
	\begin{center}
		\includegraphics{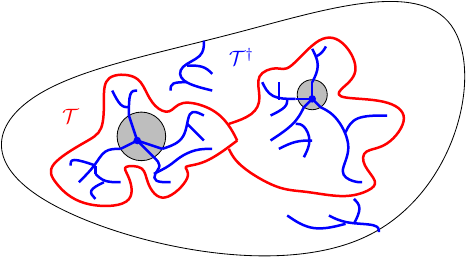}
		\caption{A non-Temperleyan CRSF (in blue). The surface $M$ is the ``pair of pants": a domain of the plane with two holes (in grey in the picture). In this example, $t$ does not contain any cycle, and hence any connected component flows to the boundary of $M$. Its dual $t^\dagger$ must contain a component with two cycles which overlap. The cycles go around each of the two holes, and must be connected, as otherwise there would have to be a path in $t$ separating them; however this is impossible as such a path would have to connect two distinct boundary points. So $t$ is \emph{not} Temperleyan. Note that $ \chi = 2-2g - b = -1 $ here. See Lemma \ref{lem:annulus} for a more general argument. }
		\label{F:temperleyan}
	\end{center}
\end{figure}

Temperleyan forests also come with a natural law, which is simply the law of the CRSF conditioned on the event that each component of the dual contains exactly one cycle. We will soon formulate a more explicit criterion for a CRSF $t$ to be Temperleyan (see \cref{SS:reducCRSF} and in particular \cref{L:top_condition}). For now, observe that if $t$ is a Temperleyan wired oriented CRSF and $t^\dagger$ is its dual, we can assign an orientation to each cycle in each component of $t^\dagger$ arbitrarily from one of the two possible choices. Then we orient all other edges of $t^\dagger$ towards the cycle of that component. 
%We let $\crsf$ be the set of pairs $(t, t^\dagger)$ where $t$ is a Temperleyan wired oriented CRSF, $t^\dagger$ is its dual (hence a free CRSF) for which an orientation of its cycles has been specified, which allows us to view $t^\dagger$ also as a free, \emph{oriented}, CRSF such that each vertex has a single outgoing edge attached to it.
%Note that if $(t, t^\dagger) \in \crsf$ then w
{We call $(t, t^\dagger)$ a Temperleyan pair or \textbf{Temperleyan self-dual pair} if $t$ is a Temperleyan CRSF and $t^\dagger$ is its dual with a choice of orientation as above. We warn the reader that the notation slightly obscures the choice of orientation needed to go from $t$ to $t^\dagger$ in order to highlight the duality. Hopefully this should not introduce any confusion since we will always want to think of our CRSF as being oriented.} %We hope this terminology will not be too confusing: $t$ is the object we will mostly work with, and $t$ determines $t^\dagger$ uniquely up to the orientation of its cycles.
%We will sometimes, with a small abuse of vocabulary, speak of the

%We call a graph $G'$ defined as above a \textbf{Temperleyan graph} embedded in the surface $M$.
%Armed with the fact that the set of dimer configurations of $G'$ is nonempty,
\medskip We now state the Temperley bijection for general surfaces. Recall that we assign \emph{oriented} weights $w_e$ to each edge $e$ in $\Gamma'$, no weight (or unit weight) to edges of $\Gamma^\dagger$ and that this turns $G'$ into a weighted \emph{un}oriented graph. Indeed, if $e = (x,y)$ is an oriented edge of $\Gamma'$, let $w$ denote the white vertex in the middle of $e$. Then we assign to the unoriented edge $\{x,w\}$ of $G'$ the weight $w_{(x,y)}$ and to the edge $\{w,y\}$ of $G'$ the weight $w_{(y,x)}$.

We define the measure $\Ptemp$ as the measure on Temperleyan pairs with weights inherited from the weights of $\Gamma$ and we will call $(\cT, \cT^\dagger)$ a generic associated random variable, i.e
\begin{equation}
\Ptemp((\cT, \cT^\dagger) =(t,t^\dagger)) = \frac{1}{\Ztemp} 1_{\{(t, t^\dagger) \text{ Temperleyan pair} \}}\prod_{e \in t} w(e),\label{eq:oriented CRSF}
\end{equation}
where $\Ztemp$ is the partition function. 
%We will associate to a Temperleyan pair $(T, T^\dagger)$ a dimer configuration $\textbf{m}$ on $G'$ with law \eqref{Gibbs}
%We define weights on $G'$ as follows $w(e)$ is defined as follows: we view $e$ as half a (unique) edge in $\Gamma^\dagger$ or $\Gamma$, and the weight of the edge $e$ is inherited from the weight in either $\Gamma^\dagger$ or $\Gamma$.

%for every half edge of $T$ formed by dividing the edge $e$ into two by a white vertex.

\begin{thm}[Temperley bijection on general surfaces]\label{prop:temp_bij}
	Let $M,\Gamma',(\Gamma^\dagger)',G'$ be as in \cref{sec:setup}. Then there exists a bijection $\psi$ between the set of Temperleyan pairs and the set of dimer configurations on $G'$. Furthermore if $(\cT, \cT^\dagger)$ has the law \eqref{eq:oriented CRSF} then $\textbf{m} = \psi((\cT, \cT^\dagger))$ has law \eqref{Gibbs} with unoriented weights on $G'$ described above.
\end{thm}
%\note{NB: add details, even if well known.} \note{Added}
%\note{NB: local figure needed.}
{We emphasise that it is unclear at this point whether the measure $\Ztemp\Ptemp$ (and consequently the measure on dimers) is not trivially zero as we have only deduced that the construction of $G'$ is necessary, but have not deduced that it is sufficient. We prove this later in \cref{L:dimerable}.}
\begin{figure}
	\begin{center}
		\includegraphics[scale=.5]{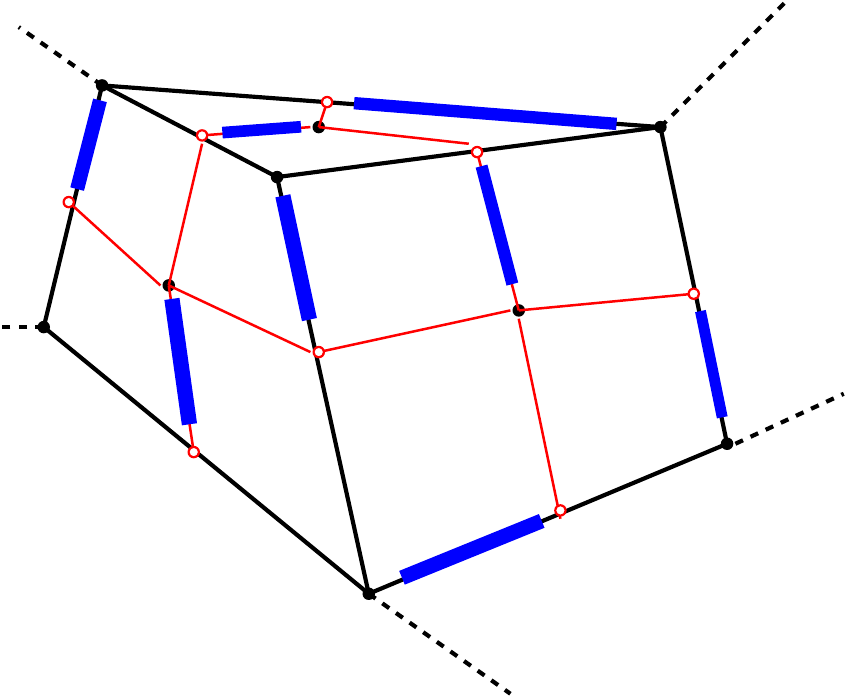} \hspace{.5cm}
		\includegraphics[scale=.5]{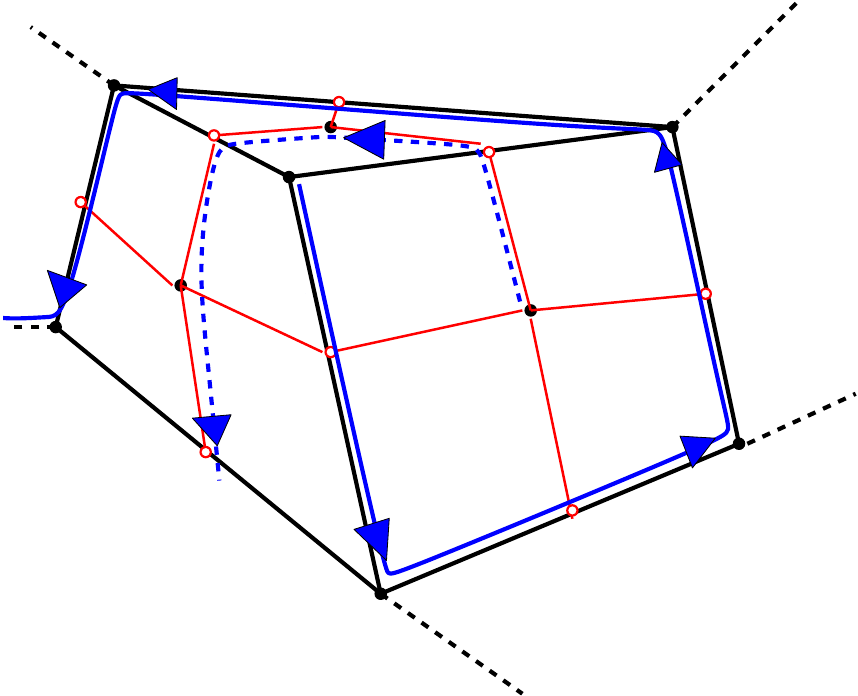}
		\caption{Illustration of the local transformation in Temperley's bijection. Left: a dimer configuration on a superposition graph. Right: a collection of oriented edges forming dual oriented CRSF.}
		\label{F:Temp_bij}
	\end{center}
\end{figure}

\begin{proof}[Proof of \cref{prop:temp_bij}]
We assume that the set of Temperleyan CRSF and the set of dimer covers are both nonempty (this assumption is validated in \cref{L:dimerable}).
	Given a Temperleyan pair $(t,t^\dagger)$, we obtain a configuration of edges $\textbf{m} = \psi((t, t^\dagger))$ as follows: for every oriented edge $\vec e \in t$ (resp. $\vec e \in t^\dagger$), we can write $\vec e = e_1 \cup e_2$ where $e_1, e_2$ are the first and second halves of $\vec e$ %(recall that $\vec e$ is oriented, whether in $t$ or in $t^\dagger$), 
 and set $e_1 \in \textbf{m}$ (see \cref{F:Temp_bij}). The resulting $\textbf{m}$ is a matching on $G'$ because every (non boundary) vertex has a unique outgoing edge in either $t$ or $t^\dagger$.
	Furthermore, since $t \cup t^\dagger$ spans the black vertices of $G'$, the matching is a perfect matching.
	
	Also $\psi$ is injective: if $(t_1, t_1^\dagger)$ and $(t_2, t_2^\dagger)$ are distinct, then there must a black vertex $v$ on $G'$ (i.e., a vertex of $\Gamma'$ or $(\Gamma^\dagger)'$) such that the unique outgoing edge from $v$ in $t_1$ or $t_1^\dagger$ is different from the unique outgoing edge from $v$ in $t_2$ or $t_2^\dagger$.
	Hence $v$ will be matched to two distinct white vertices in $\psi((t_1, t_1^\dagger))$ and $\psi((t_2, t_2^\dagger))$.
	
	We now check $\psi$ is onto. Given a matching $\textbf{m}$ of $G'$, we can obtain a pair $(t,t^\dagger)$ by extending the matched edges: {for every non-boundary black vertex $b$ of $G'$, let $w$ be the white vertex matched with $b$ in $\textbf{m}$. By construction $(bw)$ is a part of a unique edge $e=(bb')$ (either of $\Gamma$ of $\Gamma^\dagger$) and we include the edge $(bb')$ (oriented from $b$ to $b'$) in $t$ or $t^\dagger$ as appropriate.}
 %let the unique outgoing edge from $v$ be the edge of $\Gamma'$ or $(\Gamma^\dagger)'$ containing the white vertex to which $v$ is matched in $\textbf{m}$.
 The fact that neither $t$ nor $t^\dagger$ contain contractible cycles follows from the same argument as in the standard, planar case: if say $t$ has a contractible cycle $C$, then since $\Gamma \setminus C$ has a dimer cover, an elementary counting argument then shows that $v - e + f = 0$ where $v, e, f$ are the number of vertices, edges and faces of the contractible component $\Gamma \setminus C$. On the other hand, Euler's formula implies $v-e+f = 1$ (again excluding the outer face). This shows that all cycles are non-contractible. Since every vertex $v$ of $t^\dagger$ has a unique outgoing edge, $t^\dagger$ must have exactly one cycle per component and so $t$ is Temperleyan.
	{It also easily follows from the definition of the weights of the Temperleyan tree and the dimer covers that the bijection is weight preserving.}
	This concludes the proof.
	%include the unique oriented edge $\vec e = (e_-,e_+)$ such that $e_- = v$ and $v$ is matched with $w_e$. Clearly each component of an oriented subgraph with exactly one outgoing edge from every vertex can contain only a single cycle. Also the subgraph obtained spans $G'$ since the matching is a perfect matching. Clearly, the pair $(T,T^\dagger)$ obtained is unique and has the same weight as the corresponding dimer configuration. This completes the proof.
\end{proof}

%%\textcolor{gray}{We now claim that removing the white vertices as above indeed produces a dimerable graph.
%\begin{lemma}\label{L:dimerable}
%	The graph $G'$ obtained above has a dimer cover.
%\end{lemma}
%The proof of \cref{L:dimerable} will depend on our extension of Temperley's bijection to this setup and in particular the introduction of the notion of Temperleyan forest. This is what we now do, and we defer the proof of \cref{L:dimerable} until later.
%}

%\note{NB: now, explain that if $T^\dagger$ temperleyan exists, then dimerable. Hence annulus is dimerable. Then pants decomposition and hence proof of Lemma 3.1}. \note{Added extra explanation}

\subsection{Criterion for a wired CRSF to be Temperleyan}
\label{SS:reducCRSF}
{In this Section, we prove the following proposition:
\begin{prop}\label{L:dimerable}
The graph $G'$ obtained in \cref{sec:setup,} has a dimer cover. In particular, the measure $\Ptemp$ defined in \cref{eq:oriented CRSF} is a probability measure.
\end{prop}As a by product of the proof of \cref{L:dimerable}, we derive a simple criterion for a CRSF to be Temperleyan (\cref{L:top_condition}).
}
% First of all note that if we can find a Temperleyan CRSF of $\Gamma'$ then by \cref{prop:temp_bij}, $G'$ has a dimer cover. Next  The converse is also obviously true when we do not remove edges:
We start with a lemma about the $\chi=0$ case.
\begin{lemma}\label{lem:annulus}
	Let $M$ be nice with $g$ handles and $b $ boundary components and $\Gamma,\Gamma^\dagger,G$ be embedded as above (i.e. without punctures). A wired Temperleyan oriented CRSF of $\Gamma$ exists if and only if $M$ has the topology of either a torus or an annulus. Furthermore in these cases, all oriented CRSF are Temperleyan.
\end{lemma}
\begin{proof}
Note that if $M$ has the topology of an annulus (with all the nice properties of \cref{sec:surface_embedding}), then finding a Temperleyan oriented CRSF is straightforward. Indeed, \emph{any} wired spanning forest in the annulus (where both boundaries are wired) is Temperleyan: the dual is a graph containing a single cycle separating the two components touching each boundary.
Also notice that for a torus, every oriented CRSF is Temperleyan \cite{DubedatGheissari}: essentially an oriented CRSF must contain a cycle (since there is no boundary on a torus) and cutting along this cycle gives us a (bounded) cylinder or equivalently an annulus.
	
 For the converse,  observe that the extended Tempereley's bijection (\cref{prop:temp_bij}) allows us to construct a dimer configuration from the Temperleyan CRSF and its dual (by endowing each cycle with arbitrary orientation and orienting every other edge towards the unique cycle of its component). Notice that although \cref{prop:temp_bij} is written for the punctured graph $G'$, since in the case of torus and annulus $G=G'$, the same proof goes through. 
Now recall that a dimer configuration exists without removing any punctures only if $\chi=0$ or equivalently $2g+b=2$ (as discussed in \cref{SS:Tempforest}). This equation has only two feasible solutions: $g=1,b=0$ (i.e. a torus) and $g=0,b=2$ (i.e. an annulus). This completes the proof of the `only if' part.

For the last assertion, note that every oriented CRSF in the annulus divides the annulus into several smaller disjoint annuli. The rest  follows easily from the arguments in the first paragraph above.
\end{proof}

{\begin{proof}[Proof of \cref{L:dimerable}]
	In light of \cref{lem:annulus}, we assume $M$ is neither a torus, nor an annulus. Using assumption \Cref{punctures}, %particularly \eqref{eq:necessary_cond}, find paths satisfying the properties described there.
	recall that there exist a collection of paths $\gamma_{u_i}, \gamma_{v_i}$ on $\Gamma'$ for $1 \le i \le \sf k$
	starting respectively from $u_i$ and $v_i$, such that 
	%satisfying the following. 
	%The path  $\gamma_x = (w_1, \ldots, w_r)$ is viewed as a collection of adjacent vertices starting from either $u_i$ or $v_i$, $(i = 1, \ldots, \sf k)$. We identify $\gamma_x$ with a continuous path on $M$,  
%
%and suppose:
$$
M \setminus \bigcup_{1 \le i \le \sf k} (\gamma_{u_i} \cup \gamma_{v_i } \cup (u_i,v_i))
%\label{eq:necessary_cond}
$$
is topologically a disjoint union of annuli when we view the paths as continuous curves on $M$ and $(u_i, v_i)$ denotes the edge from $u_i$ to $v_i$. We start building a dimer configuration $\mathbf{m}$ on $G'$ as follows. First, for each path $\gamma $ in this collection, which by definition starts from either $u_i$ or $v_i$ for some $1\le i \le \sf k$, we orient it away from the puncture from which it emanates, i.e. away from $u_i$ or $v_i$. We then consider the collection of dimer edges naturally associated with $\gamma$ formed by the first halves of each of the oriented edges of $\gamma$, as in \cref{F:Temp_bij} and the proof of \cref{prop:temp_bij}. Then include all these edges in $\mathbf{m}$.

	%In each case, for each oriented edge, we can add a dimer similar to the usual Temperleyan bijection as described in the proof of \Cref{prop:temp_bij}. 
	
	The remaining disjoint annuli are nice in the sense of \cref{sec:surface_embedding}, so we can apply \Cref{lem:annulus} to obtain a Temperleyan forest on each of these annuli, and hence by \cref{prop:temp_bij} a dimer cover on the part of $G'$ corresponding to these annuli. Patching these together gives us a dimer configuration $\mathbf{m}$ on all of $G'$ as desired.
	%{\color{red} Not clear enough.}
\end{proof}

\begin{remark}\label{rmk:puncture}
Note that this implies that the assumption in \Cref{punctures} is, as already mentioned earlier, necessary in order for $G'$ to admit a dimer cover. Indeed if there is a dimer cover then by applying the local mapping in \Cref{F:Temp_bij}, starting from the vertices $u_i$ and $v_i$ adjacent to the punctures, we can obtain paths as described in \Cref{punctures}. Each component of the complement of these paths also has a dimer cover, but no puncture. According to \Cref{lem:annulus}, this is only possible if   \eqref{eq:necessary_cond} is satisfied. 
%It is worth pointing out that we could not find a nice combinatorial property of the graph which ensures \eqref{eq:necessary_cond}. Indeed, one can attach a large piece of a nice graph (e.g. square lattice) by a single edge to a vertex of $\Gamma$ and embed it inside a face of $\Gamma$. It is easy to see that if we put too many punctures inside this piece, then it is impossible to satisfy \eqref{eq:necessary_cond} as the paths from the punctures have to merge before leaving the piece.
%It is an interesting question to figure out exactly what combinatorial properties ensure the validity of \eqref{eq:necessary_cond}. However it is unclear whether such a property would be necessarily easier to check than simply the condition as stated. Since this point is a bit tangential to the current agenda of the paper, we do not pursue this in this article.
\end{remark}
}	

We now deduce from the above an extremely convenient criterion for a wired oriented CRSF to be Temperleyan. Let us suppose $M$ is not the torus or an annulus, whence ${\sf k} = 2g +b - 2>0$ (we already know that every CRSF is Temperleyan otherwise). Define the branch starting from a  primal vertex $v$ of $\Gamma'$ to be the path obtained by going along the unique outgoing edge from each vertex (which necessarily ends when a loop is formed or a boundary is hit).
Recall that, at a puncture (i.e., a white monomer), there are exactly two vertices $u_{i},v_{i}$ of $\Gamma'$ on either side of the puncture. Let $\mathfrak B_{i1},\mathfrak B_{i2}$ be the branches in $\Gamma'$ of $u_{i},v_{i}$ for $1 \le i \le {\sf k}=2g+b-2$. We call the union of these branches, together with the removed edges $e_i$, the \textbf{skeleton} $\mathfrak{s}$ of $t$. That is,
$$ 
\mathfrak{s} = \bigcup_{i=1}^{\sf k}(\mathfrak B_{i1}\cup e_i \cup \mathfrak B_{i2}).
$$
Simply put, the skeleton of $t$ is the union of the branches emanating out of the punctures.

\begin{thm}\label{L:top_condition}
	Suppose $M$ is not a torus. A wired oriented CRSF is Temperleyan if and only if every component of $M\setminus \mathfrak{s}$ has the topology of an annulus.
\end{thm}

(Of course, in the case of an annulus, as already noted there are no punctures and so $\mathfrak{s} = \emptyset$, so the condition is automatically satisfied, as we already know.) {The criterion for a CRSF to be Temperleyan is thus just to say that the skeleton cuts the surface into disjoint annuli. \cref{F:pants} gives two examples on a surface $M$ (the `pair of pants') where $M\setminus \mathfrak{s}$ consists of topological annuli, where $\mathfrak s$ is the paths $p,q$ in that figure. (We get three annuli in the first example, and two in the second example).}

\begin{proof}
	%This is clear from \cref{lem:annulus}.
	%First note that a component can only have the topology of a torus if the manifold is a torus in which case there is nothing to prove by \cref{lem:annulus}.
	%and no vertex was removed so that  $(\cup_{i=1}^k\mathfrak B_i \bigcup e_i)$ is empty. In a torus, it is easy to see that all CRSF are Temperleyan.
	
	Let $t$ be a Temperleyan CRSF and let $t^\dagger$ be its dual with a choice of orientation. Note that the vertices (in $G'$) of $\mathfrak{s}$ are all matched with each other in the dimer configuration associated to $(t, t^\dagger)$. Therefore in each component of $M \setminus \mathfrak s$, all vertices are also matched with each other. By \cref{lem:annulus}, this implies that these components are annuli.
	
	Conversely, note that by definition $t^\dagger$ cannot cross $\mathfrak s$,
	and so $t$ can be restricted to each component of $M \setminus \mathfrak s$ to form a wired CRSF.
	By the other implication in \cref{lem:annulus}, it follows that $t$ is Temperleyan in each such component and so is globally Temperleyan.
\end{proof}

%The significance of this criterion is as follows. By \cref{L:top_condition}, a Temperleyan forest can be thought of as a wired CRSF conditioned on the event in the statement of that proposition. However, wired CRSF are easier objects to understand owing to the fact that they can be sampled through a version of Wilson's algorithm, as will be recalled in \cref{sec:Wilson}.

%In summary, after sampling the branches $\mathfrak B_1,\ldots.\mathfrak B_k$, the rest of the wired oriented CRSF are automatically Temperleyan in each component. Thus each component can be sampled using the usual Wilson's algorithm as described in \cref{sec:Wilson}. Thus the total information of the height form on the manifold is encoded by the height form on each annulus, their conformal radii, and the topology of their gluing. In effect, \cref{L:top_condition} provides a way to decompose the height function into parts which are more amenable to analysis.

\subsection{Wilson's algorithm to generate wired oriented CRSF}\label{sec:Wilson}

Recall the measure $\Ptemp$ from \cref{eq:oriented CRSF}, i.e. the law of $(\cT, \cT^\dagger)$ on Temperleyan pairs which is naturally associated with the dimer model. We will not study directly $\Ptemp$ but rather a slightly altered version which can be sampled through Wilson's algorithm and is defined as follows: we first sample a wired oriented CRSF of $\Gamma$ with law
$$
\Pwils (\cT = t) = \frac1{\Zwils} 1_{\{ t \text{ Temperleyan} \} }\prod_{e \in t} w(e);
$$
and given $\cT = t$, we pick $\cT^\dagger$ an oriented dual uniformly among all possibilities of orientation of the dual. %Thus $\Pwils$ can be viewed also (with a small abuse of notation) as a measure on $\crsf$.

%The two laws look similar but are in fact 
Clearly the only difference between the two is due to the fact that any cycle of the dual $t^\dagger$ of a Temperleyan oriented CRSF $t$ can be oriented in two possible ways to determine a dual pair $(t, t^\dagger )$. We deduce the following relationship:

% Let $\cT_{\bf m}$ be a Temperleyan wired oriented CRSF in $\Gamma^\dagger$ obtained by applying
%Temperleyan bijection to the dimer configuration $\bf m$ on $G$ sampled from
%the above described measure. Let $\cT_{\bf w}$ be a Temperleyan wired oriented CRSF picked from the measure given by \eqref{eq:oriented CRSF}. Notice that the wired oriented CRSF determines the dual free \emph{unoriented} oriented CRSF and hence determines all the non-contractible cycles. The only undetermined information is the orientation of each cycle for which there are two possible choices. Since all the dual edges have weight 1, we readily obtain
\begin{lemma}\label{lem:RN_dimer_CRSF}
  Let $(t, t^\dagger)$ be a Temperleyan pair such that $t^\dagger$ contains exactly $k^\dagger$ non-contractible cycles. Then the Radon--Nikodym derivative satisfies
  \begin{equation*}
    \frac{d \Ptemp}{d \Pwils} (t,t^\dagger) = \frac{\Zwils}{\Ztemp}2^{k^\dagger} % = %\frac{2^k}{\Ewils(2^{k})}.
  \end{equation*}
In particular, conditioned on having $k^\dagger$ non-contractible cycles for the dual forest, the law $\Ptemp$ and $\Pwils$ coincide. \end{lemma}

Let $\Gamma, \Gamma^\dagger$ be faithfully embedded on a nice Riemann surface $M$.
%Recall that we define a distinguished set of vertices $\partial$ to be the \textbf{boundary} of $G$.
%Recall that since we do not tackle the sphere and the plane for now, $M$ is simply connected only if it is a bounded domain for which the boundary $\partial$ is necessarily non-empty.
We now describe Wilson's algorithm to generate a wired (but not necessarily Temperleyan)
oriented CRSF on $\Gamma$. We
prescribe an ordering of the vertices $(v_0,v_1, \ldots )$ of $\Gamma$.
\begin{itemize}
\item We start from $v_0$ and perform a loop-erased random walk until
  a non-contractible cycle is created or a boundary vertex (i.e., a vertex in $\partial \Gamma$) is hit.
\item We start from the next vertex in the ordering which is not
  included in what we sampled so far and start a loop-erased random
  walk from it. We stop if we create a non-contractible cycle or hit
  the part of vertices we have sampled before.
\end{itemize}
There is a natural orientation of the subgraph created since from every non-boundary vertex there is exactly one outgoing edge through which the loop erased walk exits a vertex after visiting it. Let $\Pwwils$ be the law of the resulting wired oriented CRSF.

\begin{prop}\label{prop:unoriented_to_oriented}
%The subgraph created above has the same law as a wired oriented CRSF with a law given by \eqref{eq:oriented CRSF} in $G$ with boundary $\partial$.
We have
\begin{equation}
 \Pwwils(t) = \frac1{ \Zwwils} \prod_{e \in t} w(e).\label{eq:law_CRSF}
\end{equation}
In particular, $\Pwwils$ generates a wired oriented CRSF of $\Gamma$ as described by \cref{def:CRSF} with law given by \eqref{eq:law_CRSF}. Furthermore, conditionally on being Temperleyan, $\Pwwils$ coincides with the first marginal of $\Pwils$.
\end{prop}
\begin{proof}
This follows from Theorem 1 and Remark 2 of Kassel--Kenyon \cite{KK12}.
\end{proof}

{As a consequence of the result of this section, we obtain a relatively usable description of the law $\Ptemp$ of interest: start with $\Pwwils$, which can be sampled via Wilson's algorithm, and condition on the topological event described in \cref{L:top_condition} (in the second paper \cite{BLR_Riemann2} we denote this conditioning event by $\cA_{\cM}$). Finally bias the law by the Radon-Nikodym derivative coming from \cref{lem:RN_dimer_CRSF}. One of the main difficulties in  \cite{BLR_Riemann2} stems from the fact that $\cA_{\cM}$ is asymptotically degenerate (i.e., the probability of $\cA_{\cM}$ for a wired CRSF tends to zero as the ``mesh size'' $\delta$ tends to zero). The Radon-Nikodym derivative in \cref{lem:RN_dimer_CRSF} on the other hand is relatively benign and does not introduce much further complication.
}

\section{Winding and height function}\label{sec:winding_ht}

{In this section, we explain the connecting between winding of the Temperleyan forest and height function in our setup. In the process we develop a systematic way to relate these two notions in a very general setting which we believe is of interest even in  simply connected domains.

A first difficulty is that measuring angles and computing the total angle along say a simple loop is well known to be more complicated on a surface than on the plane. In the following, we will sidestep completely this issue by only establishing the connection for the lifts of the Temperleyan forest and the height function where we will be able to use the usual planar theory of angles and winding. This will also have the advantage that the dimer height become an actual function instead of a one form.

Using this approach, we are left with a situation analogous to the original generalised Temperley's bijection in \cite{KPWtemperley}: there is a fixed graph with a given planar embedding and we have to relate the dimer height to the winding of the associated ``tree'' at a deterministic level. (We also point out that a version on a torus already appeared in \cite{S16}.) In particular, the fact that in our current setting the graphs (and trees) arise as lifts and therefore have many symmetries, should not play a significant role (recall that \cite{KPWtemperley} is valid in great generality). 

There are however still significant technicalities to work out. First, we cannot assume that our embedding uses only straight lines (since in the hyperbolic case, the lift of an edge cannot be a straight line simultaneously in different copies of the fundamental domain). Second, our ``tree'' objects are actually forests so the winding between vertices belonging to different components must be properly defined. Third, we need to check that the objects defined on the lift can be reasonably mapped back to the surface and state any simplification of the general theory that would arise specifically in our setting.

Finally, the way the connection is phrased in \cite{KPWtemperley} is relatively unwieldy so it would be nice to make it  more ``user-friendly'', for example to make it symmetric with respect to the primal and dual forests. Recall that the winding field is naturally defined on either primal or dual vertices while the dimer height function is naturally defined on faces of the superposition graph, so we need some natural local rule to move between faces and vertices. Our solution for this technicality will be to extend both the primal and dual trees along the diagonal of each face, up to the midpoint of that face, a construction which results in what we call augmented trees. We will see that with this construction, the primal and dual trees play a completely symmetric role, as desired.

Our strategy in this section will be to proceed by successive generalisations from the definition of a winding field on a single deterministic embedded tree to dual pairs of spanning tree and the relation to  dimer height function in \cref{sec:tree_embedded}. \cref{sec:height_winding} covers the situation specific to our Riemann surface setting, with the main result of this section stated as \cref{lem:number_crossing}. Finally \cref{sec:instanton_convergence} treats the relation between the instanton component and the set of non-contractible loops.}

\subsection{Winding field of embedded trees and choice of reference flow}\label{sec:tree_embedded}
\medskip

\paragraph{{Embedded tree.}} We temporarily forget about dimers and Temperleyan forests, and focus on how to compute the \textbf{winding field} of a \emph{deterministic} tree.

{Consider a graph (finite or infinite) which is a  tree oriented towards either a leaf or one of its ends. Embed it on $\C$ with smooth edges (although having points of non-differentiability at the vertices, which we call \textbf{corners}, is allowed, so that the branches are only piecewise smooth). Call the embedding $\cT$, we want to think of it as a union of smooth curves in $\C$.} Recall that, using \cref{lem:int_to_top}, the intrinsic winding of any finite self avoiding path in this tree is well defined. 
From every point $x$ on $\cT$, we let $\gamma_x$ be the unique path from $x$ following the orientation (so it either stops at a leaf or at the end towards which the tree is oriented) and call it the branch of $x$.
Suppose for now that
\begin{equation}\label{tree_assn}
|W_\i(\gamma_x)| <\infty \text{ for all }x\in \cT.
\end{equation}
We can then define a winding field  $\{h_\cT(x): x \in \cT\}$ of the tree simply as $h_\cT(x) := W_\i(\gamma_x)$. 

The first generalisation step is to extend the definition of $h_\cT$ even if \eqref{tree_assn} is not satisfied which follows from the following elementary lemma.
%expresses the height in a way which later allows us to extend the definition of $h_\cT$ even if \eqref{tree_assn} is not satisfied.
If $x \notin \gamma_y$ and $y \notin \gamma_x$, notice that $\gamma_x$ and $\gamma_y$ eventually merge and $\gamma_y$ merges either to the right or to the left of $\gamma_x$ (since the tree is embedded in $\C$, this makes sense). Let $\gamma_{xy}$ be the unique path connecting $x$ and $y$ in $\cT$.

\begin{lemma}\label{lem:winding_field}
 In the above setup, if $\gamma_y$ merges with $\gamma_x$ to its right then
$$
h_\cT(x)  - h_\cT(y) = W_\i(\gamma_{xy}) + \pi.
$$
If $\gamma_y$ merges with $\gamma_x$ to its left, then
$$h_\cT(x)  - h_\cT(y) = W_\i(\gamma_{xy}) - \pi.$$
If $y \in \gamma_x$ and $y$ is not a corner (or $x \in \gamma(y)$ and $x$ is not a corner), then $$h_\cT(x )  - h_\cT(y)  = W_\i(\gamma_{xy}) .$$
\end{lemma}
\begin{proof}
Notice that the last assertion follows simply from additivity of intrinsic winding. Indeed, for example if $y \in \gamma_x$, there is no discontinuity of intrinsic winding at $y$ since $y$ is not a corner.

For the rest, take $m \in \gamma_x \cap \gamma_y$. Notice that by additivity of winding,
\begin{align*}
h_\cT(x) - h_\cT(y) & = W_{\i} (\gamma_{xm}) - W_{\i} (\gamma_{ym})\\
& = W_{\i} (\gamma_{xm}) + W_{\i} (\gamma_{my})\\
& = W_{\i} (\gamma_{xy}) +\ve \pi
\end{align*}
where $\ve = +1$ (or $\ve = -1$) if $\gamma_y$ merge with $\gamma_x$ to its right (or left). This is clear since we need to do a half turn to move from $\gamma_{xm}$ to $\gamma_{my}$ at $m$ and the turn is clockwise if $\gamma_y$ merges to the right of $\gamma_x$ and anticlockwise otherwise.
\end{proof}
We remark that the last assertion of the above lemma could be made to work even at corners by adding the appropriate angle, however in what follows, we avoid using winding field at corners so we don't need to introduce these additional difficulties. Finally the formulas for $h(x) - h(y)$ described in \cref{lem:winding_field} can be taken to be the definition of the winding field (or rather its gradient) for any tree embedded with piecewise smooth edges, even if \eqref{tree_assn} is not satisfied. This will be the typical situation for our setup.

\paragraph{{Dual pairs and dimers}.}
Take an infinite planar graph $\Gamma$, properly embedded in $\C$ with smooth edges and {suppose that the embedding is locally finite in the sense that every vertex and face has finite degree, every face is a topological disc, and any bounded set intersects only finitely many edges}. Let $\cT$ be a one ended spanning tree of $\Gamma$ ({we emphasise that we are still considering a spanning tree for the moment and not yet a forest}). Since $\Gamma$ is locally finite, $\cT$ admits a dual $\cT^\dagger$ which is also a one-ended spanning tree of the dual graph $\Gamma^\dagger$.
%\note{I am guessing this is in $\C$ already.}
Let $G$ be the superposition graph (in $\C$), as introduced in \cref{sec:setup} (with the obvious modifications to take into account that our graph is infinite and without creating any puncture).
It will be useful to \textbf{augment the trees} $\cT$ and $\cT^\dagger$ as follows. {Recall that in \cref{sec:setup}, for each face $f$ of $G$, we fixed a diagonal $d(f)$ which is just a smooth simple curve connecting the primal and dual vertex staying inside the face and a \textbf{midpoint} $m(f)$ which was just a point in $d(f)$.}
For each face, add to $\cT$ (resp. $\cT^\dagger$) the portion of the diagonal $d(f)$ connecting the point $m(f)$ to the unique primal (dual) vertex touching $d(f)$. This way the primal and dual trees meet in each face $f$ at the (smooth) point $m(f)$ on the diagonal $d(f)$ of that face. With a small abuse of notation, we will still denote by $\cT$ and $\cT^\dagger$ these augmented trees.

%We can extend the tree $\cT$ by adding the portion of $d(f)$ joining its primal endpoint and $m(f)$. We still call it $\cT$ by an abuse of notation, and observe that $\cT$ is still a one ended tree. Similarly, for $\cT^\dagger$, we add the portion of $d(f)$ joining its dual endpoint and $m(f)$ and we call this new tree $\cT^\dagger$.

Note that we can orient $\cT,\cT^\dagger$ towards their respective ends. This allows us to define two winding fields as above (one for $\cT$ and one for $\cT^\dagger$).\
Having oriented $\cT$ and $\cT^\dagger$ we can also apply the local operation of \cref{F:Temp_bij} as described in \cref{SS:Tempforest} to obtain a dimer covering of $G$. We will first need to define a suitable reference flow on $G$, which will then allow us to speak of the height function associated to the dimer configuration and then show the relation between this dimer height function and the two above winding fields.

%\note{Not sufficiently clear whether a forest or a tree.}

\begin{figure}
\centering
\includegraphics[scale = 0.7]{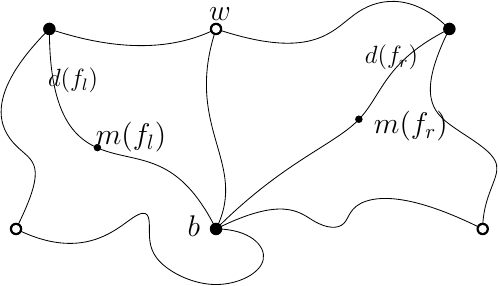}
\caption{The definition of the height function in terms of winding.}\label{fig:wind_ht}
\end{figure}

\begin{defn}[reference flow]\label{def:ref_flow}
Let $w,b$ be two adjacent white and black vertices and let $f_l, f_r$ be the faces to the left and right of the oriented edge $(bw)$.
Define
\begin{equation}
\omega_{ref} (wb) = \frac1{2\pi} \big(W_\i\big((m(f_r) , b) \cup (b,m(f_l))\big) + \pi \big),
\label{D:flowref}
\end{equation}
 where $(m(f_l) , b)$ is the portion of $d(f_l)$ joining $m(f_l)$ and $b$, and similarly $(m(f_r),b)$.
Define $\omega_{ref} (bw) =  - \omega_{ref} (wb)$ (see \cref{fig:wind_ht}).
\end{defn}
\begin{figure}[h]
\centering
\includegraphics[scale = 0.7]{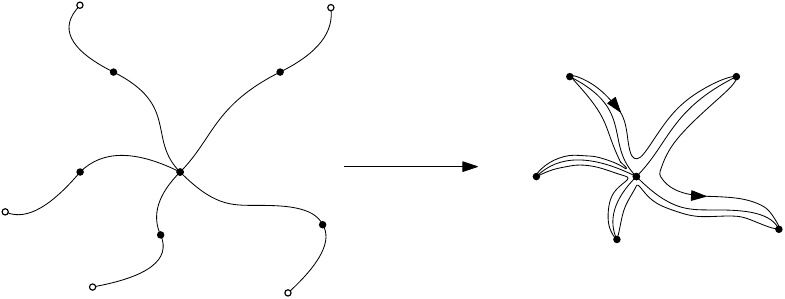}
\caption{Proof of \cref{lem:omega_ref}.}
\label{F:ref_flow}
\end{figure}
\begin{lemma}\label{lem:omega_ref}
$\omega_{ref}$ defined above is a valid reference flow. That is, the total mass sent out of any white vertex $w$, $\sum_{b \sim w} \omega_{ref}(wb)$, is equal to 1; and the total mass received to any black vertex $b$, $\sum_{w \sim b} \omega_{ref}(wb)$, is also 1.
\end{lemma}
\begin{proof}
Let $w$ be a white vertex. Notice that in $\sum_{b \sim w} \omega_{ref}(wb)$, the oriented diagonals form a clockwise loop whose total winding is
$- 2\pi$. Adding $+\pi$ for each of the surrounding black vertices, and dividing by $2\pi$, we see that the total flow out of $w$ is indeed 1 as desired.

For a black vertex $b$, the argument is better explained by considering a picture (see \cref{F:ref_flow}).  Fatten the ``star" formed by the half-diagonals incident to $b$ into a star shaped domain. Then notice that the total flow out of $b$ is  simply the limit of the total winding, divided by $2\pi$, of the boundary of this domain (again in the clockwise orientation this time), as the  domain thins into the ``star". Indeed, the $-\pi$ term in the definition of $\omega_{ref}$ in \eqref{D:flowref} counts the half-turn as we move from the left side to the right side of a half-diagonal. Since the total winding of such a curve is $-2\pi$, it follows that the total flow out of $b$ is $-1$, as desired.
%\note{It would be more natural to consider the flow received by $b$ and change the orientation of the picture.}
\end{proof}

We are now ready to relate the three notions of height function defined by the pair of dual spanning trees $\cT, \cT^\dagger$. To do so, note that we can extend the definition of the winding field of both $\cT$ and $\cT^\dagger$ from Lemma \ref{lem:winding_field} to the augmented trees.

\begin{prop}\label{prop:winding_field=ht}
In the above setup, let $h_\cT$ and $h_{\cT^\dagger}$ be the winding field of $\cT$ and $ \cT^\dagger$ respectively. Let $h_{\dim}$ be the height function corresponding to the dimer configuration obtained from $(\cT,\cT^\dagger)$ with reference flow $\omega_{ref}$.  For any two faces $f, f'$,
$$
h_\cT (m(f')) - h_{\cT}(m(f)) = h_{\cT^\dagger} (m(f')) - h_{\cT^\dagger}(m(f))  = 2\pi \big(h_{\dim } (f') - h_{\dim} (f)\big) .
$$
\end{prop}

\begin{remark}
  The winding fields in $\cT$ and $\cT^\dagger$ now play a completely symmetric role.
\end{remark}

\begin{proof}
 %Define $\gamma_v$ for the branch of $\cT$ starting from $v$ and similarly define $\gamma^\dagger_v$.
 For a face $f$, we define $\gamma_f$ (respectively $\gamma^\dagger_f$) to be the branch of the augmented tree $\cT$ (resp. $\cT^\dagger$) starting from $m(f)$. Fix faces $f$ and $f'$ and assume without loss of generality that $\gamma_{f'}$ is to the right of $\gamma_f$ {when they merge}.
 Note that this means that $\gamma^\dagger_{f'}$ is to the left of $\gamma^\dagger_{f}$. {Let $\gamma_{f f'}$ and $\gamma_{f'f}^\dagger$ denote the paths in $\cT$ and $\cT^\dagger$ connecting respectively $m(f)$ to $m(f')$} and $m(f')$ to $m(f)$. With the above convention on the relative position of $\gamma_f$ and $\gamma_{f'}$, the concatenation of $\gamma_{ff'}$ with $\gamma^\dagger_{f'f}$ has a to a simple clockwise loop {and there is no jump of the winding at $m(f)$ and $m(f')$ because by assumption the midpoint $m(f)$ is a smooth point of $d(f)$} so 
 $$
 W_{\i}(\gamma_{ff'}) + W_{\i}(\gamma^\dagger_{f'f})  = -2\pi
 $$
The {first equality} easily follows by applying the definition of winding field from \cref{lem:winding_field}.

For the last equality let $f_l,f_r$ be adjacent faces. Let the common (oriented) edge be $(bw)$ with $b$ being the black and $w$ the white vertex and let $f_r$ lie to its right. We assume {without loss of generality} that $b$ is a primal vertex. From the definition of $\omega_{ref}$ and recalling the sign convention of the flow defining the height function (\cref{sec:ht})
\begin{align*}
2\pi(h_{\dim } (f_r) - h_{\dim} (f_l)) & = \omega_{ref}(wb) - 2 \pi 1_{\text{$(bw)$ occupied by dimer}}  \\
&= W_\i((f_r , b) \cup (b,f_l)) + \pi - 2 \pi 1_{\text{$(bw)$ occupied by dimer}}.
\end{align*}
Note that that $\gamma_{f_r}$ and $\gamma_{f_l}$ merge at $b$. Also note from Temperley bijection that if $(bw)$ is occupied by a dimer, then $\gamma_b$ starts by using the (half) edge $bw$ which implies that $\gamma_{f_l}$ lies to the left of $\gamma_{f_r}$. Otherwise, $\gamma_{f_l}$ lies to the left of $\gamma_{f_r}$. We conclude using the definition of winding field from \cref{lem:winding_field} and the above equation.
\end{proof}

\paragraph{{From tree to forest}.}
{One can think of \Cref{prop:winding_field=ht} as giving a way to compute the (gradient of the) winding field by following the path between two vertices in the the tree. Before dealing with forests, we generalize the above by showing how to compute a height difference along an arbitrary path in $\Gamma$, which might use edges not in the tree.}

\begin{figure}[h]
\centering
\includegraphics[width = 0.8 \textwidth]{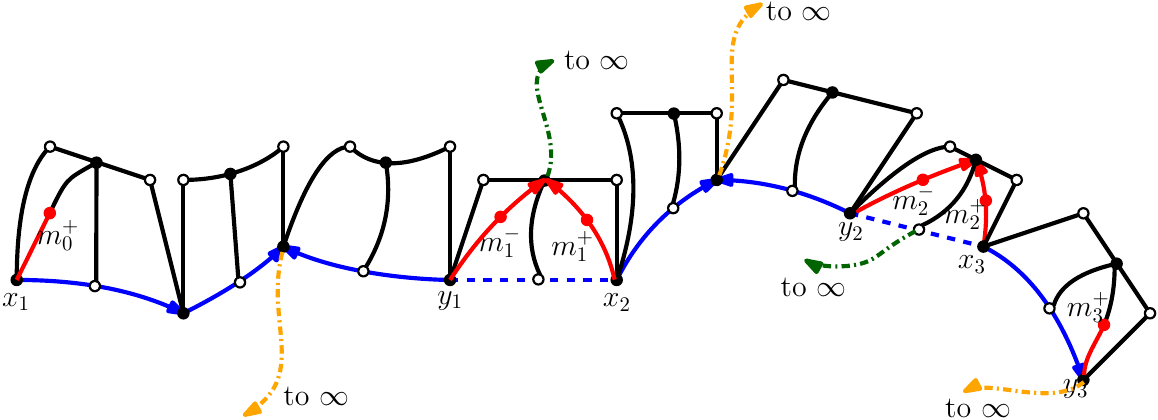}
\caption{{An example for the computation of the height along an arbitrary trajectory. The path $P$ in $\Gamma$ is the union of solid blue and dashed blue edges and here it is composed of $k=3$ partitions.} The solid blue denotes the segments in the tree $\cT$ with arrows giving the orientation of each edge. The union of solid red and solid blue edges is the modified path. The orange dotted paths denote the direction in which the primal tree goes off to $\infty$ and the green dotted path denotes the same for the dual tree. This determines the $\ve_i's$ and $\delta_i'$s. Here $\ve_1 = -1,\ve_2 = 1,\ve_3 = -1$ and $\delta_1=1, \delta_2=-1$.}\label{F:paths_modified}
\end{figure}

We suggest the reader to look at \Cref{F:paths_modified} while reading the following definitions.
Fix $f$, $f'$ two faces of $G$ and let $P$ be a self avoiding path starting at $m(f)$, going along $d(f)$ up to the vertex of $\Gamma$ adjacent to $f$, then moving along edges of $\Gamma$ up to $f'$ and finally following $d(f')$ up to $m(f')$.
 Partition this path into connected components belonging to $\cT$ (called segments from now on) separated by edges not belonging to $\cT$, allowing for the segments to be reduced to a single vertex (so any two segments are joined by a single edge). We call these segments and edges respectively $(P_1,\ldots, P_k)$ and $(e_1,\ldots,e_{k-1})$ and we let $(x_i,y_i)$ be the starting and ending points of $P_i$ and $z_i$ be the white intersection vertex in $e_i$.

If $x_i \neq y_i$, let $\ve_i = +1$ (resp. $-1$) if $\gamma_{y_i}$ lies to the right (resp. left) of $\gamma_{x_i}$ for $1 \le i \le k$. If $x_i = y_i$ (i.e. when $P_i$ is a single vertex), then let $\ve_i = +1$ (resp. $-1$) if $\gamma_{x_i}$ starts to the left (resp. right) of $P$.
{Note that by definition the edges $e_i$ are not in $\cT$ so $z_i$ are all in $\cT^\dagger$. We let $\delta_i = +1$ (resp. $-1$) if $\gamma_{z_i}^\dagger$ starts to the left (resp. right) of $P$.}

%\note{I dont like the current writing. In the previous version, the paths $\gamma$ involved the midpoints. THis is because the last line of \Cref{lem:winding_field} is only valid if $x$,$y$ are not corners. Also the notation of $m_i$s and $f_i$s got deleted and they are in the lemma below. Did you delete or comment those parts? I think it was good and correct the way it was before...}\note{B : The notations $m_i^\pm$ still appear in the proof of Lemma 4.6 which was not changed. Indeed the notation should be changed in the lemma so I will just do it.. The reason for my change is that we used to define the modified path first, then define the $\epsilon_i$ and $\delta_i$ from the modified path so since there is a simple way to define them directly from our main objects I decided to make the change. Note that I also moved the proof of Lemma 4.6 after its statement since in the previous version it was given in plain text before.}

\begin{lemma}\label{lem:winding_arbitrary_path}
Under the above setup,
$$ h_{\cT}(m(f)) - h_{\cT}(m(f')) = W_{\i} (P) + \pi(\sum_{i=1}^{k} \ve_i  +\sum_{i=1}^{k-1} \delta_i )= 2\pi(h_{\dim} (f) - h_{\dim} (f') ).$$
\end{lemma}
\begin{proof}
Modify $P$ as follows. Observe that the oriented edges $e_i = (y_i,x_{i+1})$ has two faces of $G$ to their left. Call the one incident to $y_i$, $f_i^-$ and the one incident to $x_{i+1}$, $f_i^+$ and $m_i^-, m_i^+$ the respective midpoints of the diagonals of these faces. By convention, also call $x_1 = m_0^+$ and $y_k = m_k^-$. We also join $y_i$ to $x_{i+1}$ using the diagonal segments of $f_i^-$ and $f_i^+$. Finally we delete the edges $e_i$. This completes the modification of the path $P$ and note that this modification does not change the intrinsic winding, so we are still going to denote the modified path $P$ (see \cref{F:paths_modified}).

Note that after this modification, connected components connecting diagonal midpoints staying in $\cT$ alternate with components connecting diagonal midpoints staying in $\cT^\dagger$ so we can apply  \cref{lem:winding_field} and \cref{prop:winding_field=ht} for all such connected components. The terms $\eps_i$ and $\delta_i$ match the $\pm \pi$ terms in \cref{lem:winding_field} so we obtain the result.
\end{proof}

We now define the winding field given by a dual pair of spanning forests. {We are in the same setup as above for the graphs $\Gamma,\Gamma^\dagger$, except $\cT,\cT^\dagger$ are not necessarily spanning trees but spanning forests, and they are dual to each other.}
{
\begin{defn}\label{def:winding_intermediate}
Suppose we are in the above setup and $(\cT, \cT^\dagger)$ is a dual pair of spanning forests where each component is oriented towards one of its ends. The winding field of $(\cT,\cT^\dagger)$ is a function defined on all diagonal midpoints which satisfies the first equality from \cref{lem:winding_arbitrary_path}. It is well defined and unique up to a global shift in $\R$.
\end{defn}
}
{\begin{proof}[Proof that \Cref{def:winding_intermediate} is consistent]
Once we show consistency, uniqueness up to a global shift by $\R$ is clear.
     To that end, note that one can apply Temperley's bijection to $(\cT, \cT^\dagger)$ to obtain a dimer configuration whose height function is always well defined up to a global shift in $\R$ and which can be computed along any path. The second equality of \cref{lem:winding_arbitrary_path} shows that the definition is consistent.
\end{proof}
}

\begin{figure}
\centering
\includegraphics[scale = 0.5]{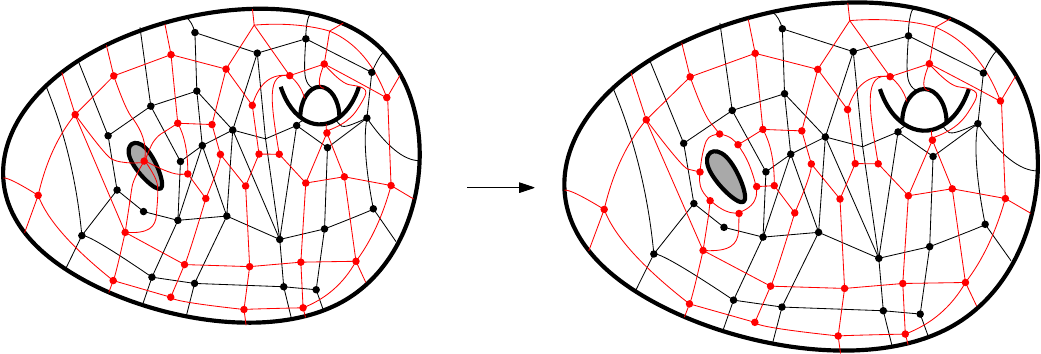}
\caption{Splitting of the boundary vertex of $\Gamma$ (red graph) into a cycle. The black graph is $\Gamma^\dagger$.}\label{fig:boundary_split}
\end{figure}
\begin{remark}\label{rmk:wind_generalize}
Note that there is no assumption about the number of ends of $(\cT,\cT^\dagger)$ in the above definition: it makes sense as long as an orientation is fixed towards a unique end in each component. {Further note that in the above proof, we only need the pair $(\cT, \cT^\dagger)$ to be associated to a dimer configuration by Temperley's bijection so \cref{def:winding_intermediate} also extends to the case where $\Gamma$ and $(\cT, \cT^\dagger)$ are obtained by lifting a graph and a Temperleyan pair $(t,t^\dagger)$ (satisfying the assumptions in \Cref{sec:setup}) from the manifold $M'$ to its cover. Note however that we need to be a bit careful with the boundary vertices of $\Gamma$ as they are somewhat special. In order not to deal with boundary components separately, we replace the boundary vertex by a cycle whose length is the same as the degree of the vertex, and with every vertex in the cycle having one neighbour in $\Gamma$ (see \Cref{fig:boundary_split}). It is not too hard to see that the neighbours can be chosen so that the embedding of $\Gamma$ still satisfies the first three items of the definition of faithful embedding (see \Cref{sec:setup}).   Let us call these the boundary cycles of $\Gamma$ (the boundary cycles defined in \Cref{sec:setup} belonged to $\Gamma^\dagger$, so this should not cause confusion). We now replace the boundary vertex in $t$ by this cycle in a natural way. In this setup, every component of $t,t^\dagger$ has a unique non contractible cycle, with the boundary component of $t$ corresponding to the boundary cycles we just defined. It is clear from the unique path lifting property that an oriented loop in $t, t^\dagger$ corresponds to a collection of bi-infinite simple paths in the lift, each one obtained by going along the loop infinitely many times in clockwise and anti-clockwise direction. Hence, after the above surgery of the boundary vertex, $t$, $t^\dagger$ lift to a forest $\cT,\cT^\dagger$ with every component having exactly two ends. Fixing an arbitrary orientation of the boundary cycles, we can also orient the lifts towards one of its ends. We can then readily apply \cref{def:winding_intermediate}.}
%Later we will see that because of topological reasons, each component (either of primal or dual) of the lifts of the Temperleyan forests has \emph{exactly two ends} almost surely\blue{ so we were slightly more general than strictly needed}. \note{add note that the setting above extends to the setting on the universal cover.}
\end{remark}

\subsection{Winding of CRSF and height function}\label{sec:height_winding}

{In this section, we give further details on the connections between winding and height functions which are specific to our setting, i.e graphs and spanning forests obtained as lifts. More precisely our main goal will be to provide for every pair of faces $f, f'$ a canonical path $\gamma_{f f'}$ that makes the formula from \cref{lem:winding_arbitrary_path} as simple as possible, see \cref{lem:number_crossing} for the final statement.}

At this point, we need to spare a few words related to the removal of white vertices from the graph $G$ to obtain a graph with a dimer cover. This operation can be interpreted as inserting certain discrete version of magnetic operators on the free field (e.g. in the sense of \cite{dubedat_torsion}). If we want to interpret the height function as winding, the height function would be additively multivalued where it picks up an additional $\pm 2\pi$ winding when it goes around a removed white vertex. Recall that we introduced a puncture corresponding to each face obtained by removal of a white vertex (cf. \cref{edge_removal}) and that the new manifold is called $M'$. According to the above discussion, it is natural to lift $G'$ to the universal cover $\tilde M'$ of $M'$ and not the cover of $M$ (for the rest of this section, when we speak of lift or cover, we will always refer to $\tilde M'$) and call it $\tilde G'$. Note that punctures are mapped to the boundary of the disc and
 every non-outer face of $G'$ is mapped to a quadrangle.% \note{Should we add information about where in the face is the puncture or maybe draw a picture of the local behaviour close to a puncture ?}

\begin{figure}
\begin{center}
\includegraphics[width= .9\textwidth]{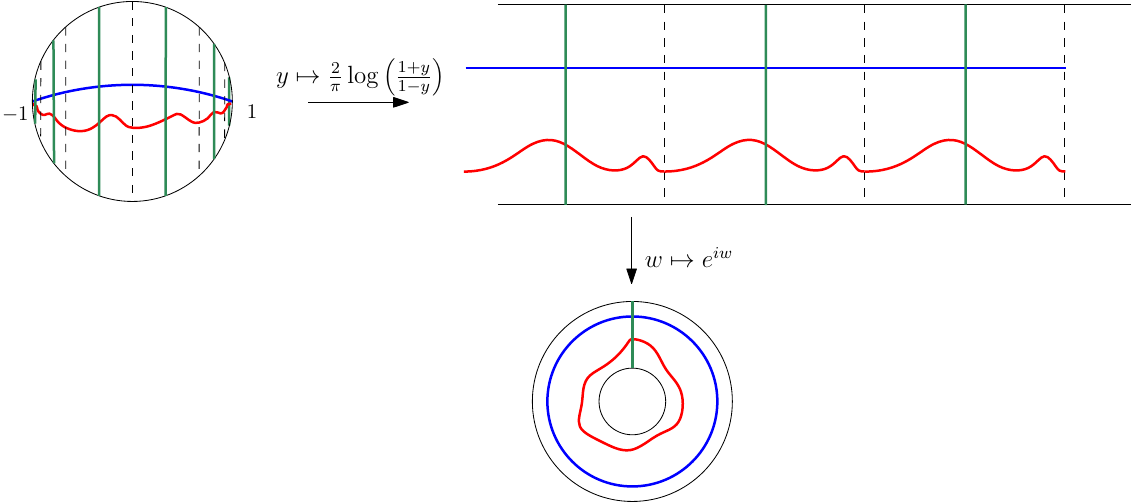}
\caption{The conformal maps between an annulus and its universal cover. The dotted lines separate different copies of the fundamental domain. The blue and red curves show two loops in the surface and the associated spines in the cover.}\label{fig:cover_annulus}
\end{center}
\end{figure}

Fix a Temperleyan CRSF $t$ of $\Gamma'$ and its oriented dual $t^\dagger$, both augmented by joining the midpoints of diagonals as explained in \cref{sec:tree_embedded}.  Each component of $t$ or $t^\dagger$ contains either a single cycle (which is oriented) or contains a boundary component. Using the surgery in \cref{rmk:wind_generalize}, each boundary component of $t$ also has a unique non-contractible cycle. We call each component of the lift of a cycle a \textbf{spine}. Note that by definition, each spine is a bi-infinite path oriented towards one of its ends.

We now state a useful lemma about the geometry of spines.

\begin{lemma}\label{lem:spine:boundary}
If $M$ is hyperbolic, then any spine $S$ is either a simple path in $\D$ connecting two points in $\partial \D$ or a simple loop containing a unique point in $\partial \D$.

{If $M$ is the torus, any spine $S$ is a bi-infinite periodic path and all spines and dual spines have the same asymptotic direction.}
\end{lemma}
The proof of the first part of \cref{lem:spine:boundary} uses tools from the theory of Riemann surfaces, so we postpone it to \cref{app:spine} in order to not disrupt the rest of the argument too much. 

\smallskip

By \cref{lem:spine:boundary}, we see that $(\cT, \cT^\dagger)$ form a dual pair of spanning forest of $\tilde \Gamma'$ and its dual so \cref{sec:tree_embedded} applies to them. We now turn to the definition of the terms involved in our main statement, \cref{lem:number_crossing}.

Given two spines $S$ and $S'$, we have a well defined region $\Omega_{S, S'}$ between them which is bounded by $S$, $S'$ and (in the hyperbolic case) two portions of $\partial \D$ (if $S$ and $S'$ are loops, then these portions of $\partial \D$ are understood as only prime ends associated with the same point). We say that a spine or dual spine $S''$ separates $S$ from $S'$ if it connects these two portions of $\partial \D$. Since the graph $\tilde G'$ only has accumulation points on $\partial \D$, it is easy to see that two spines can only be separated by finitely many others. Furthermore, two adjacent dual spines must be separated by a primal spine.%In \cref{F:spine}, the spines $S,S'$ are separated by a primal spine denoted by a dotted black curve.
\begin{figure}[h]
\centering
\includegraphics[scale = 0.5]{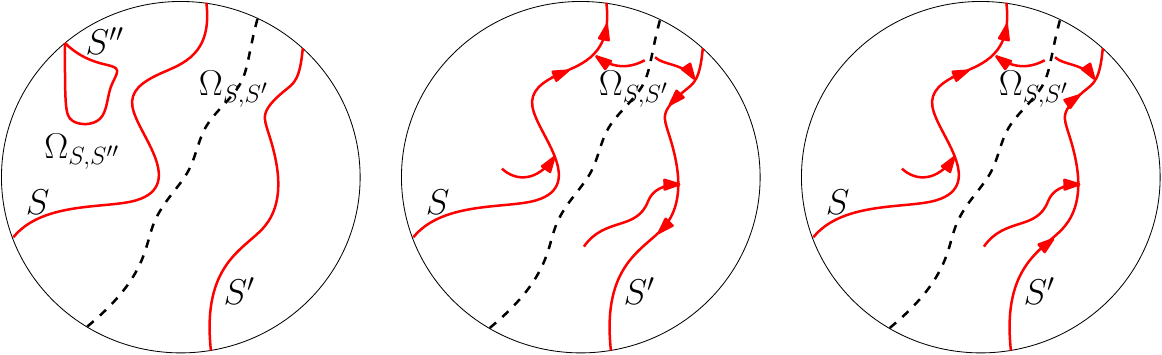}
\caption{An illustration of spines in the hyperbolic case. The spines $S$ and $S'$ are separated by a dual spine drawn in dashed blue.
This spine also separates $S''$ and $S'$ (left). The second and third figures from the left illustrate the choice of $\ve_S$ in \cref{lem:number_crossing}: in the second, $\ve_S = 1 = -\ve_{S'}$. In the third  figure (note that the orientation of $S'$ is reversed), $\ve_S = \ve_{S'} = 1$.}
%\note{The figure in the middle and on the right are identical!}}
\label{F:spine}
\end{figure}

 Now pick two faces $f,f'$ of $\tilde G'$ and let $S,S'$ be the spines of the components of $\cT$ containing $m(f), m(f')$. Let $\Omega_{S,S'}$ be the component between them as above. Draw a simple curve connecting $S$ and $S'$ in $\Omega_{S,S'}$, and oriented from $S$ to $S'$. For each primal spine $\sigma \subset \Omega_{S,S'}$ (hence not including $S$ and $S'$), let $\ve_\sigma$ be the algebraic number of times $\sigma$ is crossed by this curve (with the convention of counting +1 if $\sigma$ is crossed from its left to its right by the curve and -1 otherwise). If $\sigma$ is not crossed, define $\ve_\sigma$ to be 0. Define $\delta_\sigma$ similarly for primal spines. Notice that
\begin{equation}
\sum_{\sigma\subset \Omega_{S,S'}}\ve_\sigma + \delta_\sigma
\label{topterm}
\end{equation}
 is a topological term which does not depend on the choice of the curve (where the sum is over all primal and dual spines contained in $\Omega_{S, S'})$. Also, its only dependence on $f$ and $f'$ is through the spines $S$ and $S'$.

For any face $f$, recall that $\gamma_f$ is the infinite oriented path in $\cT$ started from $m(f)$ and going off to infinity along the unique outgoing oriented edges. Let $\zeta$ be the limit point of $\gamma_f$ (which exists due to \cref{lem:spine:boundary}).  Let $\zeta'$ be the limiting endpoint of $S'$ which lies in the same connected component of $\partial \tilde M \cap \Omega_{S,S'}$ as $\zeta$. Note that we are not defining $\zeta'$ to be the limit point of $\gamma_{f'}$ as it could be the case that $\zeta$ and $\zeta'$ might lie in different boundary components of $\partial \tilde M \cap \Omega_{S,S'}$ (depending on the orientation of $S'$).
%and eventually we would want to define the boundary arc $(\zeta,\zeta')$.
In case $S,S'$ are loops through the same boundary point, we want $\zeta $ and $\zeta'$ to be in the same prime end of $\Omega_{S,S'}$.
\begin{itemize}
\item If the limit point of $\gamma_{f'}$ does not lies in the same component of $\partial \tilde M \cap \Omega_{S,S'}$ as $\zeta$ and $f'$ is in $\Omega_{S, S'}$, then we set $\ve_{S'} =  -1$,
%of $\gamma_{f'}$ lies in the same component of $\partial \tilde M \cap \Omega_{S,S'}$, then define $\ve_{S'} = + 1$.
\item in all other case, define $\ve_{S'} = +1$.
%If not, then there are two cases. If $f'$ lies in $\Omega_{S,S'}$, then define $\ve_{S'} = -1$. Otherwise,  define $\ve_{S'} = +1$.
\end{itemize}
Finally define $\eps_S = + 1$.

Let $(\zeta,\zeta')$ be the arc joining $\zeta$ and $\zeta'$ (which could be a single point if $\zeta  = \zeta'$) in $\partial \tilde M$ in the hyperbolic case. Let $h_{\dim}$ be the dimer configuration corresponding to the pair $(\cT,\cT^\dagger)$ with reference flow given by \cref{def:ref_flow}. We can now state the following final result relating the dimer height function to the winding of trees:

\begin{thm}\label{lem:number_crossing}
In the hyperbolic case, let $\gamma := \gamma_{ff'}$ be the curve formed by concatenating $\gamma_f$, $(\zeta,\zeta')$ and the path in $\cT$ joining $\zeta'$ to $m(f')$. Orient $\gamma$ from $m(f)$ to $m(f')$. We have the following deterministic relation:
\begin{equation}\label{windingrel_hyp}
h_\cT( m(f')) -h_\cT( m(f)) = 2 \pi(h_{\dim}(f') - h_{\dim}(f) ) = W(\gamma, m(f)) + W(\gamma, m(f'))+ \pi\sum_{\sigma \in \bar \Omega_{S, S'}} (\ve_\sigma+\delta_\sigma).
\end{equation}
Here the sum $\sum_{\sigma \in \bar \Omega_{S, S'}} \ve_\sigma  + \delta_\sigma$ is as in \eqref{topterm} but also includes $S$ and $S'$.

{In the case of the torus, we have
\begin{multline}\label{windingrel_torus}
h_\cT( m(f')) - h_\cT( m(f)) = -2 \pi(h_{\dim}(f') - h_{\dim}(f) ) =  \\W(\gamma_f, m(f)) - W(\gamma_{f'} , m(f)) +  W(\gamma_{f}, m(f') )  - W(\gamma_{f'}, m(f'))+ \pi\sum_{\sigma \in \bar \Omega_{S, S'}} (\ve_\sigma+\delta_\sigma).
\end{multline}}
\end{thm}

% \begin{remark}
%   The same statement holds in the case of the torus, but with a few obvious modifications, where we think of $\zeta $ and $\zeta'$ being the same point at infinity (so starting from $f$ and $f'$ both paths go to infinity in the same direction). Then \eqref{windingrel_hyp} is also true but obviously without the winding term between $\zeta$ and $\zeta'$.
% \end{remark}

\begin{proof}
%Consider the winding field $h$ corresponding to $(\cT,\cT^\dagger)$ defined as in \cref{lem:winding_arbitrary_path}. See also the discussion following that lemma for the definition when $\cT$ is a forest. Notice that the winding field differences between the diagonal midpoints are exactly the dimer height differences between the corresponding faces, hence it is enough to prove the lemma for the winding field.

Let us consider the hyperbolic case first. Take two vertices $x,x'$ in $S,S'$ respectively. Observe that we can find a path joining $x,x'$ which goes along the primal components, and moves from one component to the other by ``jumping" over dual spines (i.e. going along edges whose dual belong to a dual spine). Also one can make sure that the path is minimal, in the sense that every component between those of $x,x'$ (i.e. components which separate $S,S'$) is visited at most once. Call this path $(x,x')$. Now letting $x \to \zeta$ and $x' \to \zeta'$, we can also ensure that the portion of $\Omega_{S,S'}$ bounded by $(x,x'), (\zeta,\zeta')$ and the spines $S,S'$ contains none of $f$ or $f'$. Note that for this path, the choice of the $\ve_{S}, \ve_{S'}$ exactly matches with the choice defined for \cref{lem:winding_arbitrary_path}. We refer the readers to the second figure of \cref{F:spine} for the case when the limit point of $\gamma_{f'}$ does not lie in the same component of $\partial \tilde M \cap \Omega_{S,S'}$ and $f'$ lies in $\Omega_{S,S'}$, in which case $\ve_{S'} = -1$.

Let $\tilde \gamma$ be the path obtained by concatenating $(m(f), x), (x,x'), (x', m(f')) $. From \cref{lem:winding_arbitrary_path}, it is clear that
\begin{equation*}
2 \pi( h_{\dim}(f') - h_{\dim}(f) ) = W_\i(\tilde \gamma) + \pi \sum_{S} (\ve_S+\delta_S ).
\end{equation*}
 Indeed, notice that since the path is minimal, the $\ve_S, \delta_S$ terms are defined so as to match with the definition of $\ve_i,\delta_i$ in \cref{lem:winding_arbitrary_path}. Furthermore, $W_\i(\tilde \gamma) = W_\i( \gamma)$ because of the above choice of $x,x'$. We finish the proof using \cref{lem:int_to_top} {and the fact that the topological winding of $\tilde\gamma$ and $\gamma$ are equal.}.

For the torus case, the same proof applies by noticing that $W((x,x'), m(f)) + W((x,x'), m(f'))$ converges
to 0 as $x$ and $x'$ go to infinity along $S,S'$ in the same asymptotic direction {(the signs in \cref{windingrel_torus} come from the fact that one goes along $\gamma_{f'}$ in the opposite direction)}.

Finally, the fact that the term $\pi \sum_{S} (\ve_S+\delta_S )$ converges follows simply from the fact that the number of non-contractible components is a.s. finite in the limiting CRSF (and hence so is the number of spines which separate $S,S'$ into different components).
\end{proof}

\subsection{Mapping back to the manifold}\label{sec:instanton_convergence}

In this subsection, we briefly explain how the notion of height function constructed above on the universal cover corresponds to a height one-form on the surface coming from the dimer cover.  We also describe how the instanton component is a function of only the non-contractible loops of the Temperleyan CRSF. 

%perform a few sanity checks on our %construction, i.e, we %
%describe how to interpret the height function and winding defined on the cover in terms of  a  height one-form on $M'$ coming from a dimer cover. 

{ Suppose $\gamma$ is a path in $M'$ connecting two faces $f,f'$ and let $h(f,f')$ denote the height difference in $M'$ calculated along $\gamma$. Now imagine two copies of the lift of $\gamma$ to the universal cover $\tilde M'$ and observe that using our formula (\Cref{windingrel_hyp} or \eqref{windingrel_torus}) for computing the height difference along the lifts in the universal cover, it is not clear how these quantities are related for different lifts and how these formulas map back to the surface $M'$. In our next lemma, we show that this difference among different copies can be explicitly calculated and is deterministic, so it does not contribute to the fluctuations of the height function, which is the quantity of interest in this article.
 
 Recall from \Cref{SS:universalcover,sec:setup} that the covering map $p$ maps $\tilde M'/F$ to $\tilde M'$ where $F$ is a subgroup of $\Z^2$ in the torus case, and is a Fuchsian group (a discrete subgroup of the M\"obius group) in the hyperbolic case. Thus one can map one lift of a path in $ G'$ to another copy simply by using an element from $F$,
which is either a translation in the torus case, or a M\"obius map in the hyperbolic case. In the following lemma we show how this mapping affects the winding. We encourage the reader to recall the formulas in \Cref{sec:winding} before reading the lemma.

 %The proof of the following lemmas require some elementary facts about Riemann surfaces which can be derived from the Fuchsian theory, together with assumptions on the graph in \cref{sec:setup} and some basic properties of winding. We therefore encourage the reader to consult \cref{app:spine,sec:setup,sec:winding} before reading this proof.

\begin{lemma}\label{lem:different_lifts}
Suppose we are in the hyperbolic case, so $\tilde M' = \D$. For any M\"obius map $\phi$ from $\D$ to itself, mapping $\tilde G'$ to itself, for any two diagonal midpoints $m, m'$, 
\[
h_{\cT}( \phi(m') ) - h_\cT( \phi(m) ) = h_\cT( m') - h_\cT(m)+ \arg_{\phi'(\D)}(\phi'(m(f) ) - \arg_{\phi'(\D)}( \phi'(m(f') ).
\]
In particular, $h_{\cT} - \arg( \phi')$ can be mapped back as a one-form on the dual of $G'$, away from vertices corresponding to punctures. In the parabolic case the same identity holds with no argument terms. 
\end{lemma}
\begin{proof}
    We focus on the hyperbolic case since the case of the torus is trivial. If $\phi$ is a conformal map mapping $\tilde G'$ to itself, then it must also map $\cT$ to itself and in particular $\gamma_{\phi(f) \phi(f')} = \phi( \gamma_{ff'} )$ in the notations of \cref{lem:number_crossing}. Clearly the topological terms $\epsilon_{\sigma}, \delta_{\sigma}$ from \cref{lem:number_crossing} are invariant under $\phi$:
    \[
\sum_{\sigma \in \Omega_{S, S'} } (\epsilon_\sigma + \delta_\sigma) = \sum_{\sigma \in \Omega_{\phi(S), \phi(S')}}  (\epsilon_\sigma + \delta_\sigma).
    \]
    For the term $W(\gamma, m(f) ) + W(\gamma, m(f'))$, it transforms as an intrinsic winding under conformal maps so by \cref{lem:winding_change_conformal}
    \begin{multline}
W( \phi(\gamma), \phi(m(f)) ) + W( \phi(\gamma), \phi(m(f')) ) = W( \gamma, m(f) ) + W( \gamma, m(f') ) \\
+ \arg_{\phi'(\D)}(\phi'(m(f) ) - \arg_{\phi'(\D)}( \phi'(m(f') ),
    \end{multline}
  which concludes the proof.  
\end{proof}
}
\begin{remark}\label{rem:fluctuation_lift}
A consequence of \Cref{lem:different_lifts} is the following. Fix any two faces $f,g$ of $G'$ and any path $\gamma$ in the dual of $G'$ starting at $f$ and ending at $g$. Let $h(f,g)$ be the sum of the height one form calculated along $\gamma$. Suppose $\tilde \gamma$ and $\tilde \gamma'$ be two copies of the lifts of $\gamma$ with endpoints $\tilde f,\tilde g$ and $\tilde f',\tilde g'$ respectively. Recalling the definition of height function $h_{\dim}$ from \Cref{sec:tree_embedded,sec:height_winding}, observe that 
$$
h_{\dim}(\tilde f) - h_{\dim}(\tilde g) - \E(h_{\dim}(\tilde f) - h_{\dim}(\tilde g) )  = h_{\dim}(\tilde  f') - h_{\dim}(\tilde g') - \E(h_{\dim}(\tilde  f') - h_{\dim}(\tilde g') ).
$$
since the $\arg$ terms cancel.
In other words, the dimer height fluctuation does not depend on the choice of the copy of $\gamma$ used in the cover, and hence can be mapped back into $M'$ without any ambiguity.
\end{remark}
%\begin{remark}
    %For the same it is not hard to check that the reference flow from  \cref{def:ref_flow} gives rise to a reference flow on $G'$ to $M'$ by applying the 
    
    %. by adding an $\arg( \phi')$ correction. Furthermore, up to the usual factor of $2\pi$, the projection of the winding field matches the dimer height form defined directly on $M'$ using that reference. Also note that since the correction is deterministic, fluctuations of the winding-field or height function can be mapped to $M$ without any correction.
%\end{remark}
{
Let us now consider the instanton component. Intuitively, it should be completely determined by the homology class of the non-contractible loops of the Temperleyan forests (and in particular it should be determined by the spines). To that end, take a rooted loop $\lambda$, and consider the height difference  along a particular choice of the lift of the loop, starting and ending at copies of the root, call this $h(\lambda)$.
\Cref{lem:different_lifts} tells us that we can compute $h(\lambda)$ unambiguously on any choice of the lift but we should in fact be able to say more. In particular, computing $h(\lambda)$ using \cref{windingrel_hyp,windingrel_torus} a priori require more precise geometric information involving windings. The following lemma shows that in fact the computation of $h(\lambda)$ does not involve any winding term and that the instanton component is at least a ``reasonable'' function of the non-contractible loops.
}

%Let us now consider the instanton component and how it maps back to the surface via our formulas in \cref{windingrel_hyp,windingrel_torus}. Intuitively, it should be completely determined once we know the location of the non-contractible loops of our Temperleyan forests. Also, it should not depend on the winding of the lifts of the non-contractible loop. We prove a version of this statement in the lemma below. Notice that we are insisting on a weaker dependency on the lift here compared to \Cref{lem:different_lifts}.

%Consider a non-contractible (continuous, rooted) closed curve $\lambda$ in the surface. To compute the height difference along $\lambda$ we first lift $\lambda$ to the cover $p^{-1}(\lambda)$ and we then compute the height difference between two copies of the root of the loop, say $\tilde x$ and $\tilde y$. A priori this might depend on the choice of the copy, and in addition on the winding of the spines corresponding $\tilde x$ and $\tilde y$. In the case of the torus, it is clear however that the winding terms cancel out as the spines of the starting and ending points are translates of each other, hence the choice of the copy is irrelevant.  It turns out the same is true in the hyperbolic case but a slightly more involved argument is needed as the map from one copy of the lift to another is a M\"obius map, which perturbs the winding. 

Take an ordered finite set of continuous simple loops which forms the basis of the first homology group of $M'$, all endowed with a fixed orientation. Let $\mathsf H $ be the finite set of numbers which denote the height change along these loops. It is well-known (see \cref{thm:hodge}) that $\mathfrak h$ is completely determined by $\mathsf H$ (at the discrete level).\  

{\begin{lemma}\label{lem:instanton_global}

Let $(\mathfrak C, \mathfrak C^\dagger)$ be the set of oriented non-contractible loops of the primal and dual Temperlayen CRSF.
Let $\mathfrak h$ be the instanton component of the height 1-form of the dimer configuration corresponding to this CRSF pair given by the extended Temperley bijection. Then $\mathfrak h $ is a function of $(\mathfrak C, \mathfrak C^\dagger)$ only.
%In other words, given a Temperleyan CRSF pair with the same $(\mathfrak C, \mathfrak C^\dagger)$ yields the same instanton component $\mathfrak h$.

\end{lemma}
}
\begin{proof}
{
We focus first on the hyperbolic case. Let $\lambda$ be a loop from the basis of the homology group considered above, without loss of generality we can assume that $\lambda$ starts in a diagonal midpoint. Choose an arbitrary lift of $\lambda$ and call its starting and ending points $m$ and $m'$ and the corresponding faces $f$ and $f'$.

Let $\phi$ be the Mobius map sending $m$ to $m'$ and mapping $\tilde G'$ to itself. By \cref{lem:number_crossing} and \cref{lem:int_to_top}, $h_\cT(m') - h_\cT(m) = W_{\text{int}}(\gamma) + \pi \sum_{\sigma \in \bar \Omega_{S,S'}} (\epsilon_\sigma + \delta_\sigma)$ with $\gamma$ obtained in this case by concatenating $\gamma_f$, the arc $(\zeta, \zeta')$ and the time reverse of $\gamma_{f'}$ (we assume for a brief moment that $\gamma$ is smooth close to $\zeta$ and $\zeta'$ and therefore piece-wise smooth everywhere). By additivity of the intrinsic winding we have $W_{\text{int}}(\gamma) = W_{\text{int}}(\gamma_f) + W_{\text{int}}((\zeta, \zeta')) - W_{\text{int}}(\gamma_{f'})$ but since $\gamma_{f'} = \phi(\gamma_f)$, by \cref{lem:winding_change_conformal}, $W_{\text{int}}(\gamma_f) - W_{\text{int}}(\gamma_{f'}) = -\arg(\phi'(\zeta)) + \arg(\phi'(m))$. Clearly, the above argument extends by taking an arbitrary regularisation of $\gamma$ if it is not smooth near $\zeta$ and $\zeta'$ (since the right hand side is independent of the regularisation procedure). Thus we have proved that
\begin{equation}
h_\cT(m') - h_\cT(m) = \arg(\phi'(m)) -\arg(\phi'(\zeta)) + W_{\text{int}}((\zeta, \zeta')) + \pi \sum_{\sigma \in \bar \Omega} (\epsilon_\sigma + \delta_\sigma). \label{eq:instanton_lift}
\end{equation}
The function $\phi$ only depends on $\lambda$ and the choice of lift $m$ for its starting point. Also clearly $\zeta, \zeta', \bar \Omega_{S,S'}$ and the $\epsilon_\sigma, \delta_\sigma$ only depend on the set of spines which concludes the proof of the hyperbolic case.

% For the convergence statement, just observe that $\zeta, \zeta', \bar \Omega_{S,S'}$ and the $\epsilon_\sigma, \delta_\sigma$ are all continuous functions in the Schramm topology, which is the one used for the convergence of Temperleyan forest in the companion paper \cite[Theorem 1.2]{BLR_Riemann2}. Moreover, since the number $K$ of nontrivial cycles has superexponential tail as shown in the companion paper \cite{BLR_Riemann2}, no two cycles are likely to be close to one another by Wilson's algorithm and the crossing assumption. From these facts, convergence of $\mathsf H^\d$ follows.

The proof for the parabolic case is identical and in fact even simpler since $\gamma_f$ and $\gamma_{f'}$ are translate of each other and there is no term for the winding of $(\zeta, \zeta')$ in \cref{lem:number_crossing}.
}\end{proof}
% For the next lemma, we need to rely on results from the second article in the series \cite{BLR_Riemann2}. % particularly \cite[Corollary 3.14 and Proposition 6.1]{BLR_Riemann2}, which says that the set of non-contractible loops of the Temperleyan CRSF $(\mathfrak C^\d, (\mathfrak C^\dagger)^\d))$ converges in law to a set of almost surely disjoint simple loops $(\mathfrak C, \mathfrak C^\dagger)$ as the mesh size $\delta \to 0$. 
%  We record now the following corollary regarding the weak limit of $\mathsf H^\d$ as $\delta \to 0$ where .

{We record now the following corollary regarding the weak limit of $\mathsf H$ as the mesh size goes to $0$. For this corollary, we need to rely on results from the second article in the series \cite{BLR_Riemann2} (and is not a deterministic result like the rest of this section) which will be reproduced later in the paper in \cref{thm:CRSF_universal}. We write a superscript $\delta$ to account for the dependence in $\delta$ as in \cref{sec:setup}.
\begin{corollary}
    $\mathsf H^\d$ described above converges in law as $\delta \to 0$ to some $\mathsf H$ which is measurable with respect to the limit of $(\mathfrak C^\d, (\mathfrak C^\dagger)^\d)$.
\end{corollary}
\begin{proof}
The term $\arg(\phi'(m))-\arg(\phi'(\zeta)) + W_{\text{int}}((\zeta, \zeta'))$ is clearly a continuous function of $(\mathfrak C, \mathfrak C^\dagger)$. Since $(\mathfrak C^\d, (\mathfrak C^\dagger)^\d)$ converges to an almost surely disjoint set of loops using \cref{thm:CRSF_universal}, $\sum_{\sigma \in \bar \Omega} (\epsilon_\sigma + \delta_\sigma)$ in the right hand side of \eqref{eq:instanton_lift} also converges in law, thereby completing the proof.
\end{proof}
}

\section{Local coupling}
\label{sec:local_coupling_grand}
In this section we prove a local discrete coupling result which extends ideas of \cite{BLR16} to the setup of Riemann surfaces. Roughly speaking, the goal of such a result is to show that the local geometry in a small neighbourhood of a Temperleyan CRSF is given by that of a uniform spanning tree in the surface or, alternatively (and more usefully), in some reference planar domain. Moreover, locally around a finite number of given points, the configurations can be coupled to independent such USTs.

{Recall that the actual Temperleyan CRSF's cannot be completely sampled using the standard Wilson's algorithm. However, due to \cref{L:top_condition}, given the skeleton $\mathfrak{s}$ emanating from either side of the punctures, the rest of the Temperleyan CRSF can be sampled from Wilson's algorithm. The main result in this section is   \cref{lem:exp_tail}.

%\note{Previous sentence: In this section, we only deal with a CRSF generated by Wilson's algorithm. A Temperleyan CRSF needs additional assumption because of the special branches (\cref{L:top_condition}). This assumption is outlined in \cref{assumption_precise}.}

%\note{This statement is ambiguous too. It sounds like you are only stating this valid only for the CRSF, but perhaps also valid for Temperleyan forest maybe under Assumption 8.1, which I don't think is the case -- or at least away from the special branches in the Temperleyan forest. And actually that statement is kind of important too for our future work, isn't it?}

%\note{Add explanation of what happens when sampling from $\Ptemp$}

%\note{To replace the introduction of this section and the first paragraph of the next}\noteb{The purpose of this section is to show that in a CRSF, locally around a finite number of points, the configurations look like independent full plane uniform spanning tree (which we abbreviate wired UST). This which is made sense of rigorously as a coupling result in \cref{lem:exp_tail} which is just a generalisation of \cite{BLR} to our more general setting. Let us therefore discuss quickly what needs to be added in our context.

%First of all, the whole method to even define the coupling uses Wilson's algorithm which cannot be used to sample a Temperleyan CRSF, therefore in this section we will only work on the local coupling for a true CRSF.

The argument follows the same line of arguments as in \cite{BLR16} so let us first recall the strategy there. This consists of two main steps. Consider $k$ points $v_1,\ldots, v_k$ in $(\Gamma')^\d$. In the first step, we choose cutsets at a small but macroscopic distance around each of the $k$ points, such that the cutsets separate the points from each other and from the rest of the graph. We reveal \emph{all} the branches emanating from these cutsets. This leaves $k$ unexplored subgraphs $\Gamma_i^\d$, one around each point. In this step the key point is to make sure that the $\Gamma_i^\d$ are macroscopic (e.g., contain a ball of a radius roughly of the same order of magnitude as the distance to the cutsets). Clearly,  the conditional law of the tree in each $\Gamma_i^\d$ is that of a wired UST. Moreover, these wired USTs are also independent conditionally given the cutset exploration. (Of course, unconditionally there is still some dependence). The second step is then to say that in each $\Gamma_i^\d$, one can couple a wired UST with a full plane one, which shows, among other things, that the unconditional distribution is close to being independent.

To adapt this strategy to Riemann surfaces, if the $\Gamma_i^\d$ are sufficiently small so that it is simply connected, then the conditional law of the CRSF in each $\Gamma_i^\d$ will also be that of a wired UST, so that the second step can be used directly. For the first step however (cutset exploration), we will need to redo parts of the proof to take into account the possible loops in the CRSF.

We first state few useful lemmas from the second article (\cite[Lemmas 2.5 and 2.6]{BLR_Riemann2}) %\note{double check precise reference}.
\begin{lemma}[Beurling type estimate]\label{lem:Beurling}
For all $r,\ve>0$ there exists $\eta>0$ such that for any $\delta <\delta(\eta)$ and for any vertex $v \in \Gamma^\d$ such that $\eta/2 < d_M(v,\partial M) <\eta$, the probability that a simple random walk exits $B_M(v,r ) := \{z\in M: d_M(z,v) <r\}$ before hitting $\partial \Gamma^\d$ is at most $\ve$.
\end{lemma}
\begin{lemma}\label{lem:uniform_avoidance}
Let $K_0 \subset K'_0 \subset M$ be open connected sets and let $K \subset K_0',K'$ be compact sets which are the closures of $K_0,K'_0$. Also assume $K$ contains a loop which is non-contractible in $M$. Then there exists a $\delta_{K,K'}>0$ and $\alpha_{K,K'}>0$ depending only on $K,K'$ such that for all $\delta<\delta_{K,K'}$ and all $v \in \Gamma^\d$ such that $v \in K$, simple random walk started from $v$ has probability at least $\alpha_{K,K'}$ of forming a non-contractible loop before exiting $K'$.
\end{lemma}

{Next, we recall a lemma from \cite{BLR_Riemann2} about the regularity of the skeleton branches away from the punctures. As mentioned above, the skeleton $\mathfrak{s}$ cannot be directly sampled using Wilson's algorithm, however the following (technical) result from \cite{BLR_Riemann2} guarantees that it will not be qualitatively different from a regular loop-erased random walk away from the punctures.
Let $\bar \lambda_\cM$ denote the law of $\mathfrak s$. Let $\bs \eta_r := \{\bs \eta_r^{(i)} = (\eta^{(i)}_{r,1},\eta^{(i)}_{r,2})\}_{1 \le i \le {\mathsf{k}}}$ denote the portion of the branches of the skeleton until they exit a ball of radius $r$ around the puncture, where $r$ is some fixed small quantity. For each such path $\eta^{(i)}_{r,j}$  (with $1 \le i \le \mathsf k$, $j=1,2$) run independent simple random walks from their tips, where each walk is conditioned to exit the ball of radius $r'>r$ before hitting $\eta^{(i)}_{r,j}$, for another small choice of $r'$. After hitting $r'$ the walks are allowed to run unconditionally until they either hit $\bs \eta_r$ or $(\partial \Gamma')^\d$ or their loop erasure creates a non-contractible loop. Let $\tilde \mu_{r,r'}$ denote the joint law of the loop erasures of the resulting pair of walks.
\begin{prop}[Proposition 6.5 of \cite{BLR_Riemann2}]\label{prop:lemma_abs_cont_special}
For any event $E$ which is measurable with respect to $\mathfrak s \setminus \bs \eta_r$,
\begin{equation}\label{eq:special_upper_independent}
	\underline{\lambda}_{\cM}( E) \leq C \left(\sqrt{ \sum_{j=1}^{2{\mathsf{k}}}\sup_{\bs\eta_{r/2^j}} \tilde \mu_{\mathcal{M}_{r/2^j,r/2^{j-1}}}(E)} \right)
	\end{equation}
	where the supremum is over all possible $\bs \eta_{r/2^j}$.
 \end{prop}
 Thanks to that proposition, we will be able to use Wilson's algorithm to prove estimates about all the branches in Temperleyan CRSF, including the skeleton.
}
%Coming back to Temperleyan CRSF, note that even if Wilson's algorithm cannot be used directly it is not so far from working : there are a finite number of special branches which separates the surface into annuli and given these branches we just need to add independent regular CRSF in each annulus. To transpose the results of this section to Temperleyan CRSF, it will therefore be enough to control that the annuli created by the special branches are large enough, see \cref{assumption_precise} for a precise formulation.}\note{Maybe we can just replace this paragraph by a reference to Remark 7.7 when we first mention Temperleyan CRSF.}

\subsection{Cutset exploration} \label{sec:local_coupling}

 We now describe the construction more precisely. Fix $k$ points $v_1, \ldots, v_k$ in $(\Gamma')^\d$ and let $N_{v_i}$ be small enough neighbourhoods around each $v_i$ such that $\{N_{v_i}\}_{1 \le i \le k} $ do not intersect each other or the pre-specified boundary $\partial \Gamma^\d$ of $\Gamma^\d$ (if the boundary is non-empty)
 and for each $1 \le i\le k$, $p^{-1}(N_{v_i})$ is a disjoint union of sets in $\tilde M$ such that $p$ restricted to each component is a homeomorphism to its image. Also for each $i$, fix one pre-image $\tilde v_i$ of $v_i$ and let $\tilde N_{\tilde v_i}$ denote the component of $\tilde{v_i}$ in $p^{-1}(N_{v_i})$.
Let $R(\tilde v_i,\tilde N_{\tilde v_i})$ be the inradius seen from $\tilde v_i$ in $\tilde M$ with respect to the Euclidean metric, and call
 \begin{equation} \label{def_r}
 r = \min_{1\le i \le k} R( \tilde v_i; \tilde N_{\tilde v_i}).
 \end{equation}
 {(We point out that in our setup it is natural to define these quantities with respect to Euclidean geometry on $\tilde M$ and project to $M$, because our assumption and in particular the uniform crossing condition are stated with respect to this geometry; whereas the intrinsic geometry of $M$ does not really appear in our setup).}
 %Pick any set of preimages $\{\tilde v_i\}_{1 \le i \le k}$ and choose $r$ small enough such that $B(\tilde v_i,10r)$ is contained in the component containing $\tilde v_i$ of $p^{-1}(N_{v_i})$.

We denote by $B_{\Euc}(\tilde v,r)$ the Euclidean ball of radius $r$ around $\tilde v  $ and for $r'>r$, let $A_\Euc(\tilde v , r,r') = B_{\Euc}(\tilde v,r') \setminus B_{\Euc}(\tilde v,r)$ be the Euclidean annulus of inradius $r$ and outer radius $r'$.
Let $H_i$ be a set of vertices in $p(A_\Euc(\tilde v_i,r/2,r)) \cap (\Gamma')^\d$ which disconnects $v_i$ from $\partial N_{v_i}$.  %I think that we want to be in the manifold at this point to have a more intrinsic definition
{The \textbf{cutset exploration} is done in two steps
\begin{enumerate}[a.]
\item First sample the skeleton according to the correct law (i.e. according to CRSF branches with the global conditioning that the branches divide the manifold into annuli).
\item Then, reveal all the branches from $H_i, 1 \le i \le k$ using (unconditioned) Wilson's algorithm, resulting in a subgraph $\cT_{H_i}^\d$, $1 \le i \le k$.
\end{enumerate}
}
We say a vertex $ v_j$ has
\textbf{cutset isolation radius} $6^{-k}r$ at scale $r$ if
$$  p( B_{\Euc}(\tilde v_j,6^{-k}r) )
\text{ does not contain any vertex from $\cT_{H_i}^\d$}.$$
Let us define $J_{v_j}$ as the the minimum value such that $v_j$ has isolation radius $6^{-J_{v_j}}r$ and let $J = \max_j J_{v_j}$.
%Let $J_{ v_j}$ be the minimum such $k$ after performing the cutset exploration. Let $J = \max_j J_{ v_j}$.

We want to show that $J$ has exponential tail (which results in polynomial tail of the isolation radius). To this end, we rely on the following bound on the distance between a loop-erased walk and a point, which is a version of Proposition 4.11 in \cite{BLR16}.
  \begin{lemma}
    \label{lem:coming_close}
    Let $\gamma$ be a loop-erased random walk starting from a vertex $v$ until it hits $(\partial \Gamma')^\d$ or a non-contractible loop is created. Let $u \neq v$ be another vertex and let $\tilde u$ be one of the images under $p^{-1}$ of $u$. Let $r$ be small enough such that $U:=\bar B_\Euc(\tilde u,r)$, is contained in the pre-image of $N_{u}$ which contains $\tilde u$.
%Let $\tilde v$ be a lift of $v$ and let $\tilde \gamma \subset \tilde M$ be the lift of $\gamma$ starting from $\tilde v$.
Then there exist constants $\alpha,c>0$ (depending only on the initial assumptions of the graph) such that for all $\delta <\delta_U$ and all $n \in (0,\log_2(\frac{Cr\delta_0}{\delta} )-1)$ where $\delta_U,\delta_0$ are as in the crossing condition (assumption (\ref{crossingestimate})),

%    $$
% \P(|\tilde \gamma-\tilde u|<2^{-n}r_\Euc) <c \alpha^n
%$$
%for some $c>0$. In particular \noteb{no, as written this would need summing over infinitely many pre-images}, for the above choice of $\delta$,
$$
\P(\gamma \cap  p(B_{\Euc}(\tilde u,2^{-n}r))  \neq \emptyset )< c \alpha^n .
$$
\end{lemma}

%  \begin{lemma}
%    \label{lem:coming_close}
%Let $\gamma$ be a loop erased random walk starting from a vertex $u$ in
%$\Gamma^\d$. Let $r = |u-v| \wedge \dist(u,\partial) \wedge \dist (v,\partial)$. Let $v$ be another vertex in $\Gamma^\d$. For all $n \ge 1$ and Then with $\delta_0$ as in the uniform
%crossing assumption \ref{crossingestimate}, for all $\delta \in (0,C\delta_0]$ for
%some universal constant $C>0$, for all $n \in (0,\log_4( Cr\delta_0 /\delta-1))
%\cap \N$,
%$$
% \P(|\gamma-u|<4^{-n}r) < \alpha^n
%$$
% for some universal constant $\alpha \in (0,1)$. The same holds if
% $u,v$ in $\tilde \Gamma^\d$ and $\gamma$ is the loop erased random walk
% from $v$ until it exits some large domain $D$ and $r = |u-v| \wedge
%\dist(v,\partial D) \wedge \dist(u,\partial D)$.
%  \end{lemma}

%\notet{We emphasise here that the introduction of $\tilde u$ above is to have a proper description of the annuli so that we can apply the uniform crossing estimate and the choice of $\tilde u$ is immaterial except this fixes a choice of $\delta$ . So this choice of $\tilde u$ is slightly artificial in this sense.}

  \begin{proof}
    {This is almost identical to Lemma 4.11 in \cite{BLR16} except we now need to take into account the topology as well. To emphasise the differences with Lemma 4.11 in \cite{BLR16}, we recall the strategy there. The idea is that if the loop erased walk on the manifold comes inside $p(B_{\Euc}(\tilde u,2^{-n}r))$ then some lift of the random walk must necessarily come within $B_{\Euc}(\tilde u,2^{-n}r)$ -- call this region (exponential) scale $n$. Furthermore, after the last such visit, this (lift of the) random walk must cross $n$ many annuli without making a loop around $ \tilde u$ (which we called a ``full turn"). This has a probability bounded by $e^{-cn}$.

\begin{figure}[t]
\begin{center}
\includegraphics[width=.7\textwidth]{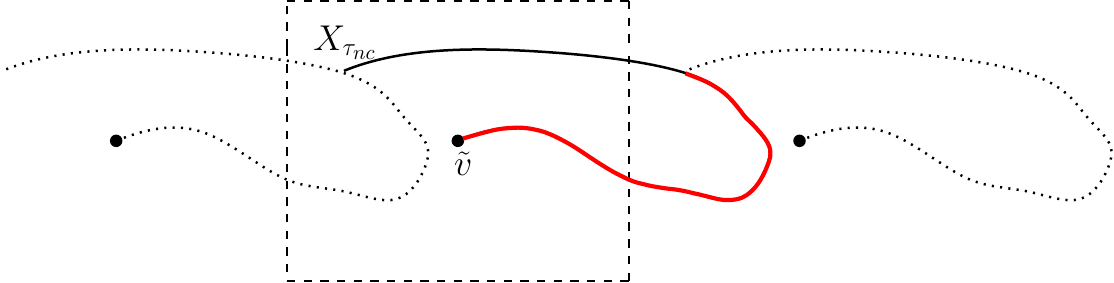}
\caption{A schematic representation (in solid line) of the lift of the loop erasure of the random walk on the torus until a non-contractible loop is formed. Call this path $\tilde \gamma$ and let $\gamma  = p(\tilde \gamma)$. The dashed  square denotes the fundamental domain and the dashed paths denote some other lifts of $\gamma$. In this case, $\tilde \gamma$ stops inside the fundamental domain.}
\label{fig:gamma_tilde}
\end{center}
\end{figure}

In the current situation however, the random walk might form a non-contractible loop before exiting $U$, and therefore its relevant lift described above does not necessarily have to cross $n$ many annuli (see e.g. \cref{fig:gamma_tilde}) after coming within Euclidean distance $2^{-n}r$ of $\tilde u$ before we stop it. Thus we cannot simply apply the argument of the previous paper.
We overcome this using the following idea which we first explain informally. The random walk on the manifold may visit $p(B_{\Euc}(\tilde u, r))$ multiple times. Consider one such time when it enters $p(B_{\Euc}(\tilde u, r))$ (not necessarily the first time) and suppose its relevant lift $\tilde X$ which also  enters $B_{\Euc}(\tilde u, r)$  at that time. Let $X'$ be the portion of the loop-erasure of $X$ on $M$ at that time such that, if the random walk hits $X'$ later on, then a non-contractible loop in $M$ is formed. %\note{Clarified the definition of $X'$ here to live solely on the manifold. Previous definition was the lift from the manifold, which also works, but is probably confusing. }
Assume that $X'$ comes inside $p(B_\Euc(\tilde u, 2^{-m}r) )$. In order for the whole loop-erasure $\gamma$ to come inside $p(B_\Euc(\tilde u, 2^{-n}r))$ the following events must take place on one such visit.
First, the lift $\tilde X$ has to cross the first $m$ scales without performing a full turn (otherwise it would close a non-contractible loop on the manifold and stop); this has a probability bounded by $e^{-cm}$. Then $\tilde X$  needs to come within Euclidean distance $2^{-n}r$ of $\tilde u$, but we bound crudely the probability of this event by 1. Finally $\tilde X$ needs to come back out after the last visit to $B_\Euc(\tilde u, 2^{-n} r)$ in such a way that no full turn occurs between distances $2^{-n}r$ and $2^{-m}r$ (for the same reason as in the simply connected case). The probability of this event can be bounded by $e^{-c(n-m)}$. The intersection of these three events gives the right overall upper bound: $e^{-cm} e^{-c(n-m)} = e^{-cn}$.}

   Let us fill in the details. Let $C_i$ be the Euclidean circle of radius $r_i:= 2^{-i}r$ around $\tilde u$ for $i \ge 0$. We start a random walk $X$ from $v$ (to emphasise again: $X$ is on the manifold $M$). Let $\{\tau_k\}$ be a sequence of stopping times defined inductively as follows. Set $\tau_0 = 0$. Having defined $\tau_k$ to be a time when the random walk $X$ crosses or hits some circle $p(C_{i(k)})$, define $\tau_{k+1}$ to be the smallest time after $\tau_k$ when the random walk crosses or hits either $p(C_{i(k)-1})$ or $p(C_{i(k)+1})$; this defines $\tau_k$ by induction for every $k \ge 1$. If $i(k) = 0$, define $\tau_{k+1}$ to be the smallest time after $\tau_k$ when the random walk crosses or hits $p(C_{i(k)+1})$.
 Let $k_{\text{exit}}$ be the smallest integer such that in the interval $[\tau_{k_{\text{exit}}},\tau_{k_{\text{exit}}+1}]$ either a non-contractible loop is created by the loop-erasure of the random walk or the random walk hits the boundary $\partial \Gamma^\d$. Let $\kappa_0,\kappa_1,\ldots, \kappa_N$ be the set of indices $k<k_{\text{exit}}+1$ when $X_{\tau_k}$ has crossed or hit $p(C_0)$.
 {Also note that the portions $ X[\tau_{\kappa_0 + 1}, \tau_{\kappa_1} ] $, $X[ \tau_{\kappa_1 +1}, \tau_{\kappa_2}] , \ldots$
 are contained inside $p(C_0)$; while the portions $ X[\tau_{\kappa_0}, \tau_{\kappa_0 +1} ],  X [ \tau_{\kappa_1}, \tau_{\kappa_1 + 1}], \ldots$ are contained outside $p(C_1)$}.
 %\note{NB what does that mean on a surface? also X is on the surface but C on the cover...}
%\note{G: A $p$ was missing in the sentence.}
     Our first claim is that $N$ has exponential tail, i.e.,  $\exists \alpha_1\in (0,1)$, such that for all $n \ge 1$,
     \begin{equation}
  \P(N >n) < \alpha_1^n. \label{E:N_exp}
\end{equation}
Indeed, using \cref{lem:uniform_avoidance}, every time the walk hits $p(C_0)$, there is a positive probability independent of the past that the walk creates a non-contractible loop before returning to $p(C_1)$. Therefore, $N$ has geometric tail. This proves \eqref{E:N_exp}.

Let $\cS$ be the set of points $\{X(\tau_k)\}_{k \ge 1}$. Note that  given $\cS$, the pieces of random walk $\{X[\tau_k,\tau_{k+1}]\}$ are independent of each other. {We call such pieces \textbf{elementary pieces} of the walk.} If $i(k) \neq 0$, $X[\tau_k,\tau_{k+1}]$ is (given $\cS$) a random walk starting from $X_{\tau_k}$ conditioned to exit the annulus $p(A_\Euc(\tilde u,r_{i(k)-1},r_{i(k)+1}))$ at $X_{\tau_{k+1}}$. If $i(k)=0$, $X[\tau_k,\tau_{k+1}]$ is a random walk which is conditioned to next hit $p(B_\Euc(\tilde u,r/2))$ at $X_{\tau_{k+1}}$. (Note that this property is lost if we solely condition on $\{X(\tau_k)\}_{k\le \kappa_N}$ since the conditioning on $N$ is complicated.) By Lemma 4.7 in \cite{BLR16}, conditionally on $\cS$, each random walk piece $X[\tau_k,\tau_{k+1}]$ has a uniformly positive probability to do a \textbf{full turn} in the annulus $p(A_\Euc(\tilde u,r_{i(k)-1},r_{i(k)}))$ for $\delta \le \delta_U$ and given range of $n$ (where $\delta_U$ comes from the uniform crossing assumption). %\note{The crossing assumption contains no $\delta(U,V)$. There is perhaps a $\delta_K$. Is what is meant? What are $U$ and $V$ here?}\note{V is a Remnant from previous writing. Also fixed notation to match with the uniform crossing. G.}
Here we define a full turn to be the event that the random walk crosses every curve joining the inner and outer boundary in the specified annulus. {Indeed, although Lemma 4.7 in \cite{BLR16} gives the uniform positive probability estimate for crossing in the Euclidean plane, the estimate is valid here by considering the relevant lift of $X[\tau_k, \tau_{k+1}]$ which is inside $\bar B_\Euc(\tilde u,r)  = U $ and applying the uniform crossing estimate in the universal cover. Let us also point out that $\delta$ is chosen small enough so that the uniform crossing estimate is valid inside $U$.}

\medskip We now define certain low probability events $\mathsf E_1,\ldots, \mathsf E_n$ such that one of them must take place if the loop-erasure of $X$ is to enter $B_n:=p(B_{\Euc}(\tilde u,2^{-n}r))$. Let us condition on $\cS$. We say the event $\mathsf E_j$ occurs if:

%The idea is to say that for all scales $1\le i \le n$, for the loop-erasure of $X$ to enter $B_n$, a full turn in the corresponding annulus cannot happen (either on the way to $B_n$ or on the way out, depending on the scale). Indeed, on the way in, full turns are forbidden in any annulus which contains a portion of the past walk such that hitting it would create a non-contractible loop. On the way out, full turns are prevented in any annulus which does not contain such a portion.

%More precisely, we define the event $\cB_j$ as follows.
\begin{enumerate}[{(}i{)}]
\item The portion $X[\tau_{\kappa_{j-1} +1},\tau_{\kappa_{j}}]$ intersects $B_n$. We then let $ \lambda_j$ be the smallest $k$ in $[\tau_{\kappa_{j-1} +1},\tau_{\kappa_{j}}]$
such that $X_{\tau_k}$ crosses (or hits) $p(C_n)$.

\item
Let $X_j'$ be the portion of the loop-erasure of $X[0,\tau_{\kappa_{j-1} +1}]$ such that if the walk $X[\tau_{\kappa_{j-1} +1},\tau_{\kappa_{j}}]$ intersects $X_j'$, a non-contractible loop would be created.
% \note{Also it would be better to recall in the notation that $X'$ depends on $j$, thus $X'_j$ instead of $X'$. }\note{Both $X,X'$ is on the manifold, hopefully it is clearer now? }
%Note that, using only the information of $X[0,\tau_{\kappa_{j-1} +1 }]$ it is possible to identify the portion $X'$ of the loop erasure of $X[0,\tau_{\kappa_{j-1} +1}]$ such that if the walk $X[\tau_{\kappa_{j-1} +1},\tau_{\kappa_{j}}]$ intersects $X'$, a non-contractible loop would be created at that point.
Let $m$ be the maximum index such that $X'$ intersects the circle $p(C_m)$. %Note that any full turn by $X$ in the annuli $A_\Euc(u, r_i, r_{i+1})$ for any $i \le m-1$ would result in a non-contractible loop (by assumption on $r_\Euc$).

Then for the event $\mathsf E_j$ to hold we also require, in addition to the previous point, that: for any $\kappa_{j-1} +1 \le k\le \lambda_j$ if $i(k)<m$, then the walk $X[\tau_k,\tau_{k+1}]$ does not perform a full turn in $p(A_\Euc(\tilde u,r_{i(k)-1},r_{i(k)}))$. If no such $m$ exists (e.g., if $X_j'$ is empty), we do not require anything further.
\item Let $ \ell_j$ be the last $k$ before $\kappa_j$ such that $X_{\tau_k}$ crosses (or hits) $p(C_n)$. Let $\ell'_j$ be the first time after $\ell_j$ that the walk intersects $p(C_{m+1}).$ Then in addition to the previous two points, for $\mathsf E_j$ to hold we require that the walk $X[\tau_k,\tau_{k+1}]$ does not perform a full turn for any $\ell_j \le k \le \ell_j'$.
\end{enumerate}

From the discussion above, it is clear that we have the following lemma.
\begin{lemma}\label{lem:events}
We have $$\{\gamma \cap  B_n \neq \emptyset\}  \subseteq  \{N \le n, \cup_{j=1}^n \mathsf E_j \} \cup \{N >n\}.$$
\end{lemma}
%\begin{proof}
%Recall that the walk $Y$ and $ X$ are coupled only up to the first time a non-contractible loop is created. Let $j$ be the first index when the walk $Y[\tau_{i_{2j-1}},\tau_{i_{2j}}]$ intersects $B_n$. If a non-contractible loop is created before $Y[\tau_{i_{2j-1}},\tau_{i_{2j}}]$ intersects that $B_n$ we are done. In particular, if point $(ii)$ of the event is violated, $\gamma$ remains outside $B_n$. On the other hand, if point $(iii)$ is violated, the portion of the walk intersecting $B_n$ is erased and the loop erased random walk does not intersect $B_n$ and also a non-contractible loop cannot be created in the process, so $X$ and $Y$ continue to be the same. Thus we have shown that on the complement of the event on the right hand side, we are outside the event in the left hand side.
%\end{proof}

Thus all we need to show is that the event on the right hand side of \cref{lem:events} has exponential tail bound. Now we claim that
$$
\P(\mathsf E_j | \cS) \le e^{-c'n} \text{ and so } \P( \mathsf E_j) \le e^{- c'n}.
$$

This is justified using the uniform positive probability of the walk $X$ performing a full turn, even conditionally given $\cS$.
Indeed, conditioned on $\cS$, we have that the event (ii) in the definition of $\mathsf E_j$ has probability at most $e^{-c'm}$, and conditioned on event in (ii), the event in (iii) has probability at most $e^{-c''(n-m)}$ since conditionally given $\cS$, the random walk portions $X[ \tau_k, \tau_{k+1}]$ are independent. {Thus the overall probability given $\cS$ is at most $e^{-cn}$ with $c = c' \wedge c''$. }
%\note{More care needed since $\ell_j$ not a stopping time as usual?} \note{G: This is conditionally on $\cS$!}
All in all, using \cref{E:N_exp}, \cref{lem:events} and a union bound, we obtain
\begin{equation*}
\P(\{\gamma \cap  B_n \neq \emptyset\}) \le ne^{-c'n} +e^{-c'n} \le ce^{-c'n}.
\end{equation*}
thereby concluding the proof.
\end{proof}

\begin{lemma}\label{lem:special_coming_close}
Let $\mathfrak s$ denote the skeleton.
Let $u $ be a vertex which is not one of the punctures and let $\tilde u$ be one of the images under $p^{-1}$ of $u$. Let $s$ be small enough such that $U:=\bar B_\Euc(\tilde u,s)$, is contained in the pre-image of $N_{u}$ which contains $\tilde u$.
%Let $\tilde v$ be a lift of $v$ and let $\tilde \gamma \subset \tilde M$ be the lift of $\gamma$ starting from $\tilde v$.
Then there exist constants $\alpha,c>0$ (depending only on the initial assumptions of the graph) such that for all $\delta <\delta_U$ and all $n \in (0,\log_2(\frac{Cs\delta_0}{\delta} )-1)$ where $\delta_U,\delta_0$ are as in the crossing condition (assumption (\ref{crossingestimate})),

%    $$
% \P(|\tilde \gamma-\tilde u|<2^{-n}r_\Euc) <c \alpha^n
%$$
%for some $c>0$. In particular \noteb{no, as written this would need summing over infinitely many pre-images}, for the above choice of $\delta$,
$$
\P(\mathfrak s \cap  p(B_{\Euc}(\tilde u,2^{-n}s))  \neq \emptyset )< c \alpha^n .
$$

\end{lemma}
\begin{proof}
This is a corollary of \Cref{prop:lemma_abs_cont_special} and \cref{lem:coming_close}. 
Consider the case ${\sf k} =1$ and recall the notations from from \cref{prop:lemma_abs_cont_special}. It is easy to see using the Markov property of the independent random walks used to sample $\tilde \mu_{r/2, r}$ that after hitting $r$, the result of \cref{lem:coming_close} kicks in, yielding the required exponential bound for $\tilde \mu_{r/2, r} (E)$. The same argument obviously applies to $\tilde \mu_{r/4, r/2} (E)$, yielding the desired statement.% For this we need to introduce a few notations used in \cite{BLR_Riemann2}.
% Let $\bar \lambda_\cM$ denote the law of $\mathfrak s$. Let $\bs \eta_r = (\eta_r^1, \eta_r^2)$ denote the portion of the branches of the skeleton until they exit a ball of radius $r$ around the puncture, for a very small choice of $r$. For each such path $\eta^i_{r}$  (with $i=1,2$) run independent simple random walks from their tips, where each walk is conditioned to exit the ball of radius $r'>r$ before hitting $\eta^i_{r}$, for another small choice of $r'$. After hitting $r'$ the walks are allowed to run unconditionally until they either hit $\bs \eta_r$ or $(\partial \Gamma')^\d$ or their loop erasure creates a non-contractible loop. Let $\tilde \mu_{r,r'}$ denote the joint law of the loop erasures of the resulting pair of walks.
% For any event $E$ which is measurable with respect to $\mathfrak s \setminus \bs \eta_r$, \cite[Proposition 6.5]{BLR_Riemann2} asserts that

% \begin{equation}\label{eq:special_upper_independent}
% 	\underline{\lambda}_{\cM}( E) \leq C \left(\sqrt{ \sup_{\bs\eta_{r/2}} \tilde \mu_{r/2,r}(E) + \sup_{\bs\eta_{r/4}} \tilde \mu_{r/4,r/2}(E) } \right)
% 	\end{equation}
% 	where the supremum is over all possible $\bs \eta_{r/2}$ (resp. $\bs \eta_{r/4}$)

%   It is easy to see using the Markov property of the independent random walks used to sample $\tilde \mu_{r/2, r}$ that after hitting $r$, the result of \cref{lem:coming_close} kicks in, yielding the required exponential bound for $\tilde \mu_{r/2, r} (E)$. The same argument obviously applies to $\tilde \mu_{r/4, r/2} (E)$, yielding the desired statement.
  The case ${\sf k}>1$ is analogous. We skip the details here.
\end{proof}

We now state the result showing an exponential tail of $J$, which is a combination of \cref{lem:coming_close,lem:special_coming_close} and Schramm's lemma, and is identical to the proof of Lemma 4.20 in \cite{BLR16} (see also the proof of Theorem 4.21 in \cite{BLR16}}); as it is identical we skip the proof here.

\begin{lemma}\label{lem:cutset isolation}
There exist constants $c,c'>0$ such that the following holds. Let $\tilde D$ be a compact set containing $B(\tilde v_i,r)$ for $1 \le i \le k$ where $r$ is as in \eqref{def_r}.
Then for all $\delta \in (0,\delta_{\tilde D})$ and for all $ m \in (0,\log_6(\delta_0r/\delta) -1)$,
$$
\P(J>m) \le ce^{-c'm}.
$$
\end{lemma}

Finally, we state a lemma which says that with exponentially high probability, a branch of the CRSF, after entering an exponential scale $t$ does not backtrack to a smaller scale. Such a lemma for SLE curves can be found in \cite{SLE} (see also Lemma 3.4 in \cite{BLR16} in the simply connected case).

\begin{lemma}
\label{lem:backtrack}
Fix $u,U, \delta_U$ as in \cref{lem:coming_close}.
Suppose $\gamma$ is the loop-erasure of a simple random walk in $\Gamma^\d$ started from vertex $u$ until it hits the boundary or creates a non-contractible loop. Suppose $\tilde \gamma$ is the lift of $\gamma$ started from $\tilde u$ (parametrised from $\tilde u$ to its endpoint). There exist constants $c,c'>0$ (depending only on the initial assumptions of the graph) such that for all $\delta <\delta_U$ and all $n \in (0,\log_2(\frac{Cr\delta_0}{\delta} )-1)$
$$
\P(\text{$\tilde \gamma$ enters  $B_{\Euc}(\tilde u, r 2^{-n})$ after exiting $B_{\Euc}(\tilde u, r 2^{-n/2})$}) \le ce^{-c'n}.
$$
\end{lemma}
\begin{proof}

Let $\mathsf E$ be the event that $\tilde \gamma$ enters  $B_{\Euc}(\tilde u, r 2^{-n})$ after exiting $B_{\Euc}(\tilde u, r 2^{-n/2})$.  Let $r$ be as in \cref{lem:coming_close} and assume $n$ is even without loss of generality. The argument for this is very similar to \cref{lem:coming_close} and in fact simpler, so we content ourselves with a sketch. Let $r,C_i,B_i$ be as in the proof of \cref{lem:coming_close}. We look at the lift $\tilde X$ of the simple random walk $X$ started at $\tilde u$ and we stop when either $X$ hits the boundary or a non-contractible loop is created. Let $\tau_k$ be the set of stopping times defined as in the proof of \cref{lem:coming_close} but for the lift $ \tilde X$ instead of $X$ and let $\cS$ be the set $ \{\tilde X(\tau_k)\}_{k \ge 1}$. Observe that lift of the loop erasure of $X$ is the loop erasure of $\tilde X$ since erasing contractible loops commutes with lifting to the universal cover. If the loop erasure of $\tilde X$ has to backtrack to scale $n$, the random walk has to necessarily enter scale $n$ after leaving scale $n/2$. Let $J$ be the largest $j$ such that $\tilde X_{\tau_j}$ crosses or hits $C_n$ after leaving $C_{n/2}$ %\note{stopping time is weird here, also how does this depend on $j$?} \note{Actually $j$ is random here, changed to $I$ and $J$ to emphasize this.} %
and let $I$ be the largest index $i$ smaller than $J$  when $\tilde X_{\tau_i}$ crosses or hits $C_{n/2}$. Conditioned on $\cS$, if $\mathsf E$ occurs, then $\tilde X$ enters $B_n$ at least once after leaving $B_{n/2}$, and none of the elementary pieces of the walk between $\tau_I$ and $\tau_J$ can perform a full turn. But again, conditioned on $\cS $ there is a uniformly positive probability to do a full turn for each elementary piece. Since there are at least $n/4$ such elementary pieces contained in $[\tau_I, \tau_J]$, we conclude the proof of the lemma applying the upper bound on the full-turn estimate on each elementary piece.
\end{proof}

\subsection{Full coupling} \label{sec:fullcoupling}

{The results of the previous section covered the first step in the proof of the coupling.\ As we mentioned above, the second step is identical to the simply connected case so we only recall the main statements.}

We first recall the result which we will need from \cite{BLR16}. We use the notations and assumptions already in force.
Let $\tilde D \subset \tilde D' \subset \tilde M'$ be two simply connected compact domains and fix $\tilde v \in \tilde D^\d$. Let $\cT^{\tilde D'}, \cT^{\tilde D}$ denote the wired UST respectively in  $(\tilde D')^\d$ and $\tilde{D}^\d$. Let $r_{\tilde v} $ denote $R(\tilde v, \tilde N_{\tilde v})$ as in \cref{def_r} (the largest Euclidean radius so that $p$ is injective).

\begin{lemma}[Theorem 4.21, \cite{BLR16}]
  \label{lem:isolation_rad}
  There exists $c,c'>0$ such that the following holds. Fix $\tilde v , \tilde D ,\tilde D'$ as above.
There exists a coupling between $\cT^{\tilde D}$ and $\cT^{\tilde D'}$ and a random variable $R'>0$ such that
$$
 \cT^{\tilde D'}\cap {B_\Euc(\tilde v,R')^\d} = \cT^{\tilde D}\cap {B_\Euc(\tilde v,R')^\d}.
$$
Furthermore, for all $\delta < \delta_{\tilde D'}$ and for all $m \in (0,\log_6(\delta_0r_{\tilde v}/\delta) -1)$, if we write $R' = 6^{-K_v}r_{\tilde v,\tilde D}$ where $r_{\tilde v, \tilde D}$ is the minimum of $r_{\tilde v}$ and the Euclidean distance between  $\tilde v$ and $ \partial \tilde D$, then
$$
\P(K_{v} \ge m) \le ce^{-c'm}.
$$
%
%Let $ I_{\tilde v_j}$ be as above. Then there exists a constant $\delta' = \delta'(B_i,B_j)$ where $B_i = \bar B(\tilde v_i,10r)$, $B_j = \bar B(\tilde v_j,10r)$ and positive constants $c,c'>0$
%such that for all $\delta \in (0,\delta')$ and for all $i \in (0,\log_6(\delta_0r/\delta) -1)$,
%and  for all $m \ge 1$
%$$
%\P(I_{\tilde v_i} \ge m| \text{ coupling fails }) \le ce^{-c'm}
%$$
%We emphasise that $i=j$ is allowed.
\end{lemma}
We now put together the cutset exploration with the above one-point coupling from the simply connected case as follows. Recall the notations from \cref{sec:local_coupling}. Let $\cT_{H_i}^\d$ be the branches revealed in the cutset exploration around each $v_i$.
Let $\cT_H^\d = \cup_{i=1}^k \cT_{H_i}^\d$. Let $\Gamma_i^\d$ be the connected component of $(\Gamma')^\d\setminus \cT_H^\d$ containing $v_i$. Observe that
given $\cT^\d_H$, the law of $\cT^\d\cap \Gamma_i^\d$ are independent wired UST in $\Gamma_i^\d$ (with the natural boundary), by the generalised Wilson algorithm. Applying the coupling of \cref{lem:isolation_rad} in the lift of $\Gamma_i^\d$ to $\tilde M'$ and some fixed compact $\tilde D$ in $\tilde M'$ containing some lifts $\tilde v_1,\ldots, \tilde v_k$ of points $v_1,\ldots, v_k$, we obtain a coupling of the oriented CRSF in $\Gamma^\d$ and independent wired USTs  $(\cT^{\tilde D}_i)_{1\le i \le k}$ in $\tilde D$. Furthermore, the $r$ in \eqref{def_r} and $r_{\tilde v}$ in \cref{lem:isolation_rad} differs by a constant multiplicative factor which depends only on the choice of lifts of the vertices. Furthermore, we may (and will) choose the $(\cT^{\tilde D}_i)_{1\le i \le k}$ to be independent of $\cT_H^\d$. In fact, note that for each $1\le i \le k$, $\cT^{\tilde D}_i$ may be chosen independent of the restriction of $\cT^\d$ to $\cup_{j \neq i} p(B(\tilde v_j, r/2))$.

\begin{prop}
  \label{lem:exp_tail}
%On the event $\cA$ that we do not abort the full coupling,
The above coupling has the following properties:
%\begin{itemize}
%\item the oriented CRSF $\cT^\d$ on $\Gamma^\d$ with boundary $\partial \Gamma^\d$,
%
%\item and independent copies of uniform spanning trees $ \cT^\d_D(j)$ in
%$ D^\d$ for $1 \le
%j \le k$
%\end{itemize}
%such that, on the event $\cA$ that the coupling is successful (i.e., we did not abort in the full coupling):
there exists random variables $R_1, \ldots, R_{k}$ such that
\begin{equation}\label{coincide}
\cT^\d
\cap p( B( \tilde v_i, R_i) )= p( \cT^{\tilde D}_i
\cap B( \tilde v_i, R_i) ).
\end{equation}
Furthermore if we write $R_i = 6^{- I_{v_i}} r_i$ where $r_i $ is the minimum of $r$ as in \eqref{def_r} and the distance between  $\tilde v_i$ and $\partial \tilde D$,
then for all
$\delta \le \delta_{\tilde D}$, %and $n\in (0,\log_6(\delta_0 r/\delta))$
and for all
$1\le i\le k$, for all $n  \in (0,\log_6(\delta_0r_i/\delta) -1)$,
\begin{equation}
  \label{eq:32}
  \P (I_{v_i} \ge n) \le ce^{-c'n}
\end{equation}
for some constants $c,c'>0$ (depending only on the initial assumptions on the graph). In particular, $\P(I_{v_i} \ge n) \le ce^{-c'n} \vee \delta^{c'}$.

The set of non-contractible loops of $\cT^\d$ is measurable with respect to $\cT_H^\d$.
In particular, $(\cT_i^{\tilde D})_{1 \le i \le k}$ are also independent of the set of oriented non-contractible loops in the oriented CRSF.
\end{prop}
Observe that when $I_{v_i}$ is very big or $r_i$ is very small, it is possible that the ball $B(v_i, R_i)$ is reduced to a point so the statement \eqref{coincide} is trivial (that is, \eqref{coincide} holds for a single vertex). This happens with a probability which is at most $\delta^{c'}$ for some $c'$.

\begin{proof}
This follows immediately from \cref{lem:cutset isolation} and \cref{lem:isolation_rad} once we observe that $I_{v_i}$ is within $O(1)$ of the sum of $J_{v_i}$ in \cref{lem:cutset isolation} and $K_{v_i}$ in \cref{lem:isolation_rad}, both of which have exponential tails.

The proof of the final assertion is a topological fact. Indeed, conditioned on $\cT_H^\d$, it is a deterministic fact that none of the oriented non-contractible loops pass through the portion of the CRSF in $\Gamma_i^\d$ (which is a wired UST). Indeed, otherwise we would have a path joining two points of the wired boundary of a wired UST which is impossible. Thus, the wired UST in each $\Gamma_i^\d$ is conditionally independent of the non-contractible loops and hence so are $\cT_i^{\tilde D}$s.
\end{proof}

% \subsection{Extension to Temperleyan forests}

% {\color{gray}\begin{remark}\label{rmk:special_branch_coupling}
% Recall from \cref{L:top_condition} that {to sample a Temperleyan CRSF} for a Riemann surface with non-zero Euler characteristic, we need to sample the special branches (called $\mathfrak B_i$ in \cref{L:top_condition}) emanating from the punctures which must decompose the surface into annuli. Also recall that these branches cannot be sampled using the usual Wilson's algorithm, as the condition conditioning on these special branches becomes degenerate in the limit (we will tackle this in a future paper). In this context, when we apply \cref{lem:exp_tail}, we pick $k$ points macroscopically away from the punctures, then first sample the special branches and assume that they do not come too close to the $k$ chosen points (this is part of \cref{assumption_precise}). After this step, we know that the generalised Wilson algorithm is applicable to sample the rest of the Temperleyan forest and hence so is \cref{lem:exp_tail}. See \cref{sec:winding_convergence} where we use this coupling in the general case to prove the convergence of height function assuming convergence of the special branches.
% \end{remark}
% }
\section{Convergence of height function and forms}\label{sec:ht_convergence}

In this section, we precisely state our main result (this is \cref{thm:main_precise}) and then prove it. Recalling the sketch from \cref{sec:organization}, we see that what remains to be done is to go from the convergence of the Temperleyan CRSF to the convergence of its winding field using the coupling of \cref{sec:local_coupling_grand}. The global idea is the same as in \cite{BLR16} but some of the estimates on winding of LERW have to be redone. The primary difficulty is that spines are made of multiple copies of LERW, and hence Wilson's algorithm cannot be used to estimate the winding of all the copies. These are dealt with in \cref{sec:apriori,sec:full_spine}. The  proof of the main result is finally concluded in \cref{sec:winding_convergence}.

\subsection{Precise statement of the result}\label{sec:statement}

From now on, we work with the manifold $M'$. Recall that it is obtained by removing $2g+b-2$ points from the interior of $M$ and that the white vertices removed from $G^\d$ to obtain a Temperleyan graph in \cref{L:top_condition} will converge to these points. Denote by $\tilde M'$ the universal cover of $M'$; in the sequel ``lift" refers to lift on $\tilde M'$.
 Let $h^\d_{\di}$ be the height function defined as in \cref{sec:winding_ht} sampled using \cref{Gibbs}. { Recall that $h^{\d}_{\di}$ is a function from the the dual of $ (\tilde G')^\d$ to $\R$ which is defined up to a global additive constant and $\bar h^{\d}_{\di} := h^{\d}_{\di}  - \E(h^{\d}_{\di})$ is independent of the choice of the reference flow (see \Cref{rem:fluctuation_lift}).} Recall the probability measures $\Pwils $ and $\Ptemp$ from \cref{sec:Wilson}.

First, we extend $h^\d$ to a function $h^\d_{\ext}: \tilde M' \to \R$ by defining $h^\d_{\ext}(x)$ to be the value of $h^\d$ on the face containing $x$. Recall the graphs $(\Gamma^\d)',(\Gamma^{\dagger,\d})'$ from \cref{sec:setup}. Let $\cT^\d$ be the Temperleyan forest sampled using $\Ptemp$. { Recall that we endowed $M$ with a Riemannian  metric $d_M$ as in \Cref{sec:background}, and we equip $M'$ with the same metric via the inclusion map. This induces a volume form $\vol$ in $\tilde M'$ }

\begin{thm}\label{thm:main_precise}
Let $\tilde f: \tilde M' \to \R$ be a smooth compactly supported function with $\int_{\tilde M'} \tilde f \vol = 0$. Suppose the assumptions in \cref{sec:setup} hold. Then
$$
\left(\int \tilde f(x) \bar h^\d_{\ext}(x)  \vol(x) , \cT^\d\right)
$$
converges jointly in law as $\delta \to 0$. The first coordinate also converges in the sense of all moments. Furthermore, the limit of the first coordinate
%$$
%\lim_{\delta \to 0}\left(\int \tilde f(x) \bar h^\d_{\ext}(x)  \vol(x) , \mathsf H^\d \right)
%$$
is measurable with respect to the limit $\cT $ of $\cT^\d$, is universal (in the sense that it does not depend on the graph sequence $(G')^\d$), and is conformally invariant.
%Also $\int \bar h^\d_{M',\ext}(x)f(x) \vol(x)$ converges in the sense of all moments.
\end{thm}
To clarify, convergence in the sense of all moments means that for all $i$, denoting $X^\d := \int \tilde f(x) \bar h^\d_{\ext}(x) \vol(x)$,
$
\E(|X^\d|^i )
$
converges as $\delta \to 0$. Notice also that since $\int \tilde f\vol =0$, the fact that $\bar h^\d_{\ext}(x)$ is defined only up to a global additive constant is irrelevant.

{Let us explain briefly what is meant by conformal invariance in our setting. Suppose $|\chi| = {\sf k}$ and $(M, x_1,x_2,\ldots, x_{\sf k})$ and $(N, y_1,\ldots, y_{\sf k})$ are conformally equivalent in the sense that there is a map $\phi:M \to N$ so that $\phi(x_i) = y_i$ for all $1 \le i \le {\sf k}$ and $\phi$ is a conformal bijection between the two Riemann surfaces. Note that if $p_M$ and $p_N$ are their respective covering maps (satisfying the properties in \Cref{sec:setup}), then there exists a M\"obius map $\psi$ such that $p_N = p_M \circ \psi$.  Let $h_M$ (resp. $h_N$) denote the limit field obtained in the setup of \cref{thm:main_precise} with the punctures being $x_i$ and $y_i$ for $1 \le i \le {\sf k}$. Then $h_N = h_M \circ \psi$ in the sense that for any compactly supported test function $f$ on the universal cover as in \cref{thm:main_precise}, 
$\int h_M  f \vol$ has the same law as $\int h_N (f \circ \psi^{-1}) |(\psi^{-1})'|^2\vol$.}

\begin{remark}
Before we start giving the details, we remark that the upcoming proof of the theorem goes through if we can establish the convergence of the Temperleyan CRSF (in the Schramm sense) to a conformally invariant limit in which all curves are a.s. simple, together with a few additional ingredients. In addition to  the assumptions on the discretisation of the surface described in \cref{sec:setup} and in particular the crossing assumption, these are the following: for each fixed vertex $v \in (\Gamma')^\d$, let $\gamma_v$ the branch of the Temperleyan forest starting from $v$. Then we require that  $\gamma_v$ satisfies  \cref{lem:coming_close} (the branch does not come close to a given point) as well as the moments on winding of \cref{cor:full_winding_moment} and \cref{lem:tildeu}. 
Note that the convergence to Brownian motion in \cref{InvP} is not really needed here; instead, convergence of random walk to a process which is absolutely continuous with respect to Brownian motion would be sufficient. Stated this way, we note that \cref{thm:main_precise} is novel even in the simply connected case. For instance, this result is applied in \cite{BHS} and in the forthcoming \cite{BLiu}. %\note{reference missing!}.
\end{remark}

{For ease of reference, let us assemble in a single statement the convergence results from the second paper in the series that are needed in the following (see Theorem 3.1, Corollary 3.14 and Proposition 6.1 in \cite{BLR_Riemann2}).
\begin{thm}\label{thm:CRSF_universal}
Let $M$ be a Riemann surface satisfying the assumptions of \Cref{sec:surface_embedding}, which we equip with the distance $d_M$ described in \Cref{sec:surface_embedding}.
 Let $(\Gamma^\d)_{\delta >0}$ be a sequence of graphs with boundary $\partial \Gamma^\d$ faithfully embedded in $M$ satisfying the assumptions of \cref{sec:setup}.  Then the
limit in law as $\delta \to 0$ of the Temperleyan CRSF sampled using $\Pwils$ or $\Ptemp$ exists
{in the Schramm topology relative to $d_M$. Both limits are independent of the sequence $\Gamma^\d$ subject to the assumptions in \cref{sec:setup}. These limits are also conformally invariant.

Furthermore, let $K^\d$ (resp. $(K^\dagger)^\d$) denote the number of noncontractible loops in the CRSF (resp. its dual). Then for any $q>1$ there exists a constant $C_q>0$ independent of $\delta$ such that
$$\Ewils(q^{K^\d}) \le C_q,$$
 where $\Ewils$ denotes the expectation under $\Pwils$. The same bound holds also with $K^\d$ replaced by $(K^\dagger)^\d$.

Furthermore, let $(C_1^\d,\ldots, C_{K^\d}^\d)$ be the set of noncontractible cycles of the Temperleyan CRSF,
$$
 (C_1^\d,\ldots, C_{K^\d}^\d )
\xrightarrow[\delta \to 0]{(d)} (C_1,C_2, \ldots, C_{K})
$$
where $C_1, \ldots, C_{K^\d}$ are almost surely disjoint noncontractible simple loops in $M$. 
}
% Finally for all $\ve>0$, there exists $C(\ve)<\infty$ such that for all $\delta$ small enough (depending on $\ve$), for all $k \ge 1$,
% \begin{equation}
%  \P(K^\d>k) \le C(\ve) \ve^k. \label{eq:non-contractible_tail}
% \end{equation}
\end{thm}

}
\medskip

%\note{This doesn't seem correct. Should this be $(X^\d)^i \prod_{m=1}^k(Y_m^\d)^{j_m} $ on the RHS ?}\note{Yes fixed. Thanks.}

\subsection{Some a priori tail estimates on winding}\label{sec:apriori}

%This section and the next are similar in spirit to \cref{sec:local_coupling} in that we are essentially reproving the same lemmas as in \cite{BLR16} in
The goal of this section is to obtain tail estimates on the winding of a branch of a Temperleyan forest. This will be achieved in \cref{L:winding_tail_partial} which is the main result of this section. Conceptually the arguments are similar to Section 4 of \cite{BLR16}. However, as in \cref{sec:local_coupling_grand}, there are additional difficulties linked with the fact that Wilson's algorithm can stop because of the formation of a non-contractible loop. Furthermore we ultimately want to consider the winding of entire spines, yet only one copy of the lift of a loop is directly connected to a loop erased random walk path. We first treat the part directly obtained by loop-erasure in this section and defer estimates on the copies for the next section.

%order to take into account the possibility in Wilson's algorithm to stop because a non-contractible loop was formed. There is however an additional difficulty here because according to \cref{lem:number_crossing}, we want to consider the winding of entire spines while only one copy of the loop is directly connected to a random walk path. We first treat the part directly obtained by loop-erasure in this section and defer estimates on the copies for the next.

Throughout this section, we deal with a CRSF sampled from $\Pwils$ (cf. \cref{sec:Wilson}).
First we sample the skeleton (cf. \cref{L:top_condition}) which we denote by $\mathfrak s$.
Let $\tilde {\mathfrak s}$ be the lift of $\mathfrak s$ to $\tilde M'$. Recall that $\mathfrak s$ decomposes $M'$ into finite number of disjoint annuli and that, conditionally on $\mathfrak{s}$, in each of them we can sample the rest of the Temperleyan CRSF using Wilson's algorithm.

We now work conditionally on $\mathfrak s$. Let $\gamma_v$ be the branch of the CRSF started from $v$ (which can thus be sampled using Wilson's algorithm).
%\note{Are we now working conditionally given $\mathfrak{s}$?} \note{yes, I added a phrase in the sentence}
For any vertex $\tilde v \in p^{-1}(v)$, let $\tilde
\gamma_{\tilde v} $ be the lift of $\gamma_v$ starting from $\tilde v$ and up until the time when $\gamma_v$ closes a non-contractible loop or hits the boundary  $\partial \Gamma^\d \cup \mathfrak s$. Let $\tau_{nc}$ be the stopping time when the simple random walk generating $\tilde\gamma_{\tilde v}$ creates a  non-contractible loop or hits the  boundary.
Recalling the definition of spines from \cref{sec:winding_ht}, note that $\tilde \gamma_{\tilde v}$ will include a part of a spine whenever the random walk does not hit $\partial \Gamma^\d \cup \mathfrak s$ when it stops (see \cref{fig:gamma_tilde}).

%To be more precise, $\tilde \gamma_{\tilde v}$ is the path obtained in Wilson algorithm starting a random walk from $\tilde v$ and erasing all loops from $\tilde X[0, \tau_{nc}]$

 Let us orient $\tilde
\gamma_{\tilde v}$ starting from $\tilde v$ and going
away from it in some continuous manner in $[0,1]$. For $t>0$, if $\tilde
\gamma_{\tilde v}$ exits $B(\tilde v,e^{-t-1})$, let $t_1$ be the  first time $\tilde
\gamma_{\tilde v}$ exits $B(\tilde v,e^{-t-1})$ and let $t_2$ be the last time it exits $B(\tilde v,e^{-t})$. In this case, if $\tilde \gamma_{\tilde v}$ ends in $B(\tilde v,e^{-t})$ we set $t_2 = 1$.

We first state a simple deterministic lemma connecting the winding of curves avoiding each other.
\begin{lemma}\label{lem:wind_comp}
Consider the annulus $A= \{z\in \C: r < |z-x| < R\}$ where $r>0$ and $R \in (r,\infty]$ and let $\gamma_0$ be a simple curve in $A$ connecting the outer and inner boundaries of $A$, assumed to be parametrised in $[0,1]$. Then for any simple curve $\gamma$ in $A \setminus \gamma_0$, we have
$$
|W(\gamma, x)| \leq \sup_{0 \leq s_1 < s_2 \leq 1} |W(\gamma_0[s_1,s_2], x)| + 4 \pi.
$$
For $R = \infty$, the statement is interpreted as $\gamma_0$ connects the boundary of the $B(x,r)$  to $\infty$.
\end{lemma}
Note that in the above lemma we do not need any assumption on the regularity of the curves {beyond continuity}.

\begin{proof}
{Observe that any two points in an annulus can be joined by a smooth curve lying completely in a half annulus containing those two points. 
Join each endpoint of $\gamma$ to its nearest points (with its intrinsic metric inherited from the Euclidean plane) in $\gamma_0$ using smooth curves $\eta_1,\eta_2$ satisfying the above mentioned property. Assume that $\eta_1$ and $\eta_2$ intersect $\gamma_0$ at $\gamma_0[s_1]$ and $\gamma_0[s_2]$ respectively. Consider the loop $L$ obtained by $\gamma$, $\eta_1,\eta_2$ and $\gamma_0[s_1,s_2]$.  Note that $L$ might be non-simple, but any loop which is created by the union of $\gamma$ and the two geodesics lie in the annulus $A$ slitted by $\gamma_0$ which is simply connected and does not contain $x$. Hence all these loops contribute winding 0 around $x$. Furthermore, it is easy to see that erasing all these loops (say in a chronological manner started from one of the tips of $\gamma$) creates a simple loop. Since the winding of $L$ after erasing these loops is 0 or $\pm 2\pi$, the total winding of $L$ around $x$ is also either 0 or $\pm 2\pi$. Since by definition $\eta_1 $ and $\eta_2$ lie in the smaller half annulus containing the two points, its winding at most $\pi$. Using this observation, the proof of the lemma is complete.}
\end{proof}
%\noteb{In the end, we decided to keep the lemma only for simple paths as that seems to be all we need in what comes later. This gets around a lot of topological mess.}

%\note{I'm not sure I follow this argument. Probably we need to say that $x$ is outside of the reunion of $\gamma $ and $\gamma_0$ and so the loop (even if nonsimple) cannot wind around $x$.}\note{I am not sure I understand your question? The curves are in the annulus and $x$ is the center of the annulus, so by definition they must be disjoint?}\note{I am saying the argument needs this fact, which should be recalled. It's not completely clear because the loop is not simple so we cannot just use Jordan. Straight lines is clearly not the way to go by the way as $x$ could be inside the loop - should be replaced by circular arcs surely (chosen so that $x$ lies outside)? Also as currently written the proof doesn't use that the two curves don't intersect, is that really needed in fact?}

Assume without loss of generality and to simplify notations that $t$ is an integer. Let $r_{\tilde v}$ be such that $B(\tilde v, 10r_{\tilde v})$ does not intersect $\tilde {\mathfrak s}$ and $p$ is injective in $B(\tilde v,r_{\tilde v})$ (this is a slight modification of the previous definition of $r_{\tilde v}$). Let $e^{-t_0} =r_{\tilde v}$.
We consider concentric circles $C_j$ of radius $\{e^{-t_0-j}\}_{0 \le j \le t+2}$ (with $B_j$ the ball inside it)
around $\tilde v$. Take a
random walk $X$ in $(\Gamma')^\d$ starting from $\tilde v$ stopped when it hits $(\partial \Gamma')^\d \cup \mathfrak s$. In case $(\partial \Gamma')^\d \cup \mathfrak s = \emptyset$ (i.e., in the case of the torus),
%\note{what do you mean? What about the double torus etc?}
%\note{ We work conditionally on the special branches, so the only case with no boundary is the torus...I have rephrased the potentially confusing sentence.}),
 the random walk continues forever. Let $\tilde X$ be the lift of this walk starting from $\tilde v$. Now let $\{\tau_k\}_{k \ge 0}$ be the
set of stopping times as described in the proof of \cref{lem:coming_close} for the random walk $\tilde X$.
That is, if we hit or cross the circle $C_{i(k)}$ at time $\tau_k$, we wait until
we hit or cross $C_{i(k)\pm 1}$ at time $\tau_{k+1}$. If $i(k) =
0$ (or $t+2$), we wait until we hit or cross $C_{1}$ (or $C_{t+1}$). Let $B_i$ denote the disc inside $C_i$.

Let $\tau_{i_1},\tau_{i_2},\ldots$
be the successive times in the sequence $(\tau_k)_{k\ge 1}$ when the walk hits $C_0$. Note that the interval $[\tau_{i_j},\tau_{i_j+1})$ is spent completely outside $B_{1}$.
%Let $\tau_{i_0} = -1$.
We now claim that the random walk cannot wind too much outside $B_1$.
% Note that the walk outside $B_0$ is given by subsets of $\tilde X[\tau_{i_1},\tau_{i_1+1}], \tilde X[\tau_{i_2},\tau_{i_2+1}], \ldots$.

 %Let $\tilde V$ define a walk which also starts from $\tilde v$ but erases all loops for $p(V)$ even the non-contractible ones, but we stop when boundary $\partial G^\d$ is hit. If $\partial G^\d$ is empty, we continue forever. Let us couple $\tilde V$ with $\tilde X$ until the first time a non-contractible loop is formed (there may not be such a time, and in that case the whole walk $\tilde V$ is the same as $\tilde X$). Define a sequence $\tau_{i_j}$ in the same way as for $\tilde X$ except now $j $ takes value beyond $\N$.

\begin{lemma}\label{lem:macro_winding}
There exist $c,c'>0$ such that for all $\delta<\delta_{ {p(\bar B_0)}}$, $n \ge 1$, $j \ge 1 $ and $u \in B_1$ such that $\P(X_{\tau_{i_j+1}} = u) >0$,
\begin{gather*}
\P \Big(\sup_{\cY \subset \tilde X[\tau_{i_j},\tau_{i_j+1}]}|W(\cY ,\tilde v)| >n\Big | X_{\tau_{i_j+1}} = u\Big) \le ce^{-c'n},\\
\P \Big(\sup_{\cY \subset \tilde X[\tau_{i_j},\tau_{i_j+1}]}|W(\cY ,\tilde v)| >n\Big | X_{\tau_{i_j+1}} \in \partial \Gamma^\d \cup \mathfrak s \Big) \le ce^{-c'n}.
\end{gather*}
Here the supremum is over all continuous paths obtained by erasing portions of $\tilde X[\tau_{i_j},\tau_{i_j+1}]$. %\note{add stuff for $t_0$}
\end{lemma}
\begin{proof}{The proof of \cref{lem:macro_winding} is very similar to Lemma 4.8 in \cite{BLR16}. But it needs an input from Riemannian geometry to control the winding of the spines near the boundary of the universal cover. We postpone this proof to \cref{app:spine}.}
\end{proof}
% \note{Have to say I'm a little lost. Also what fact of Riemannian geometry do we need, and where is this done?}

%\note{The proof here seems to have been lost ! Gourab, did you delete it by mistake or intentionally? It is visible (with comments from me) on the last version of 24 May.}

%Indeed, by uniform crossing estimate (\cref{crossingestimate}) the walk from any point in $l^\d$ has a uniformly positive  probability of hitting $B(\tilde v,e^{-j_1})$ before its lift hits $\ell_-$ (i.e. the random walk has to wind in the other direction and hit $\ell$ which is even more unlikely) for the given choice of $\delta$. We conclude by iterating this bound. \note{we need a lemma conditioned on hitting the inner annulus here}

\cref{lem:macro_winding} takes care of the winding of the excursions outside $B_1$. For  excursions inside, most of the technical work was done in \cite{BLR16}.
Let $ \tilde Y^j$ be the loop-erasure of $\tilde X[0,j]$. For any $j$, we parametrise $\tilde Y^j$ in some continuous way away from $\tilde v$. Fix $m \ge 1$ and let $t_1$ be the first time $\tilde Y^{\tau_{i_m}}$ exits $B(\tilde v,e^{-t-1})$ and $t_2$ be the last time $\tilde Y^{\tau_{i_m}}$ exits $B(\tilde v,e^{-t})$.

\begin{lemma}\label{lem:winding_inner_annulus}
There exist $C,c>0$ such that for all $\delta<\delta_{ {p(\bar B_0)}}, t \in (t_0, \log (C \delta_0/\delta))$,  and $n \ge 1,m \ge 1 $,
\begin{equation*}
\P\Big( \sup_{\cY \subseteq \tilde Y^{\tau_{i_m}}[t_1, t_2] \cap B(\tilde v, e^{-t-1})^c}  |W(\cY,\tilde v)| >n  | \tau_{i_m} < \infty\Big) < Ce^{-cn}.
\end{equation*}
\end{lemma}
We emphasise that in the above lemma, $m$ is non-random as it is going to be important in what follows.
\begin{proof}
The proof of this lemma is identical to that of \cite{BLR16}, Lemma 4.13 and the only difference is in the setup, which we point out. In \cite{BLR16}, we were working with a simply connected domain and we waited until the random walk exited it. The argument proceeds by conditioning on the positions $\cS := \{\tilde X_{\tau_k}\}_{k \ge 1}$ and arguing that, in each interval of walk between successive points in $\cS$, the winding of any continuous subpath has exponential tail. Then we proved that we only need to look at a random number of intervals which itself has exponential tail. In the present setup, we can condition on $\cS$ so that $\tau_{i_m} < \infty $ is satisfied. The exponential tail of winding inside any inner annulus follows from Lemma 4.7 in \cite{BLR16} (this is exactly the same as the simply connected case) and for the outer annulus, we use \cref{lem:macro_winding}. The number of relevant intervals to consider also has exponential tail following verbatim the proof of Lemma 4.15 in \cite{BLR16}. \end{proof}

With these lemmas we can now state and prove the main result of this section, which controls the (topological) winding in an annulus.

\begin{lemma}\label{L:winding_tail_partial}
There exist constants $C,C', c>0$ so that for all $n \ge 1$, for all $\tilde v$, for all $\delta < \delta_{\bar B(\tilde v,r_{\tilde v})}$ and for all $0< t  < \log (C'r_{\tilde v}\delta_0 / \delta)$,
\begin{equation}\label{E:windingtailpartial1}
\P\Big(\max_{\cY \subseteq \tilde \gamma_{\tilde v}[t_1, t_2] \cap B(\tilde v, r_{\tilde v} e^{-t-1})^c}   |W(\cY,\tilde v)| >n
 \Big) <Ce^{-c n^{1/3}},
\end{equation}
where the maximum is taken over all continuous segments from $\tilde \gamma_{\tilde v}[t_1, t_2]$. Also
\begin{equation}\label{E:windingtailpartial2}
\P\Big(\max_{\cY \subseteq \tilde \gamma_{\tilde v}    \cap B(\tilde v, r_{\tilde v} e^{-1})^c}   |W(\cY,\tilde v)| >n
 \Big) <Ce^{-cn}.
\end{equation}
\end{lemma}
A few remarks are in order. Firstly, the stretched exponential tail is an artefact of the proof and we believe that an exponential tail bound could be proved with more care if necessary. Secondly, note that the intersection in the argument of the $\max$ above is used so that we don't need to look into what happens very close to $\tilde v$. This is a technicality which simplifies the proof, but later it is not going to matter. In the end, we can decompose the whole path $\tilde \gamma_{\tilde v}$ into a disjoint union over $t$ of $[t_1',t_2']$ where $t_1'$ is last exit of $B(\tilde v,e^{-t-1})$ and $t_2'$ is the last exit of $B(\tilde v, e^{-t})$ which will accomplish the desired moment bound of truncated winding by exploiting the moment bound for each of these segments. We also emphasise that, although the non-contractible loop in the component containing $v$ could be a branch from the skeleton, $\tilde \gamma_{\tilde v}$ itself does not contain any portion of the skeleton. This is simply because we have sampled the skeleton first before defining $\gamma_v$.

\begin{proof}[Proof of \cref{L:winding_tail_partial}]
The main idea is to use a union bound for the winding of each excursion between annuli. One difficulty with working with $\tilde \gamma_{\tilde v}$ is that if we condition on where or when a non-contractible loop is formed then we break the independence between the pieces of random walk. Indeed, a walk conditioned on not forming a non-contractible loop will avoid certain portions of its previous trajectory, which \emph{a priori} might bias it to have a lot of winding. Recall that at the point when a non-contractible loop is formed, the walk could be very close to the starting point.  Thus the idea is to run the random walk  until the boundary $(\partial \Gamma')^\d$ is hit (which could be potentially longer than dictated by Wilson's algorithm) and control the winding of its loop erasure uniformly over all scales using \cref{lem:macro_winding,lem:winding_inner_annulus}. After this, we need to separately compare the winding made by the last excursion into the ball to the previous estimate.

{Recall that $t_0$ satisfies $e^{-t_0 } = \tilde r_{\tilde v}$. Note that \eqref{E:windingtailpartial2} follows easily from \cref{lem:macro_winding}.}
% \note{confusing formulation. Better: we can easily from Lemma xxx that}
%$$
%\P\Big(\max_{\cY \subseteq \tilde \gamma_{\tilde v}    \cap B(\tilde v, e^{-t_0-1})^c}   |W(\cY,\tilde v)| >n
% \Big) <Ce^{-cn}.
%$$
Indeed, let us call $\tau_{nc}$ the first time where a non-contractible loop is created and $\tau_\partial$ be the first hitting of the boundary. Fix $N = \max \{ j : \tau_{i_j} < \tau_{nc} \wedge \tau_\partial \}$. Recall that $N$ has exponential tail since every time we hit $C_0$,
conditioned on what happened before, we have a positive probability to create a
non-contractible loop or hit the boundary {before hitting $C_{1}$}, by \cref{lem:uniform_avoidance}. Thus we can work on the event $N \le n$ at a cost which is exponentially small in $n$. Now \cref{lem:macro_winding} entails that on each of the $n$ pieces, the winding has uniform exponential tail. This completes the proof of \eqref{E:windingtailpartial2}. We emphasise here that to bound portions of the loop erasure outside radius $e^{-1}$, we can run the random walk until it hits the boundary $(\partial \Gamma')^\d$ and not worry about the case when the walk stops deep inside the ball, as the loop erasure outside radius $e^{-1}$ is already taken care of. 

%\note{B : That last sentence sounds like the type of things I would say but I can't figure out why we need to emphasize it.} \notet{I think the primary issue here is the case when the random walk stops deep inside the ball, and we point out that this part is `easy' as we dont care if the walk stops inside or not. G}

Now we turn to \eqref{E:windingtailpartial1}. As before, let $\cS := \{\tilde X_{\tau_k}\}_{k \ge 1}$. The idea is to compare the winding of $\tilde \gamma_{\tilde v}$ to the winding of $\tilde Y^{\tau_{i_m}}$ because there the conditioning on $\cS$
preserves the independence and \cref{lem:winding_inner_annulus} controls the winding. The main work will be to do a case-by-case analysis of what can be erased and created between $\tau_{i_m}$ and the creation of a non-contractible loop.

Let $N$ be as above and note that using the same idea of exponential tail of $N$ and exponential tail of winding up to a fixed number of hits of the outermost circle (\cref{lem:winding_inner_annulus}), we get
\begin{equation}
\P\Big( \sup_{\cY \subseteq \tilde Y^{\tau_{i_N}}[t_1, t_2] \cap B(\tilde v, e^{-t-1})^c}  |W(\cY,\tilde v)| >n \Big) < Ce^{-c'n}.  \label{eq:YiN}
\end{equation}
%\noteb{for some $C,c$ which depend only on the crossing estimate.}
%\note{I don't understand. Surely you want this to be uniform in $t$? But the argument above depends on fixing a scale $r$ so that beyond that scale we have a positive chance of making a loop.} \noteb{I added a statement about uniformity of $C,c$. Lemma 8.6 is uniform in $t$ so I am not sure I follow? The above sentence is just to remind the idea of Lemma 8.6 but we can remove it if it leads to confusion...}
%\note{Actually doesn't that follow from the above directly?}
%
\begin{comment}
To be more precise, observe that we are in a setting where we have a sequence of random variables $X_i$, not independent but with uniform exponential tails and a stopping time $N$ with exponential tail. Therefore we can bound $\P( X_N \geq n ) \leq \sum_{i = 1}^n \P( X_i \geq n) + \P(N \geq n) \leq nC e^{-cn} + C'e^{-c'n} \leq C'' e^{-c'' n}$.
\end{comment}

Let $\lambda_1$ and $\lambda_2$ be respectively the first time $\tilde Y^{\tau_{i_N}}$ exits $B(\tilde v, e^{-t-1})$ and the last time $\tilde Y^{\tau_{i_N}}$ exits $B(\tilde v, e^{-t})$. To ease notations, from now on, all maximums in this proof come with the additional condition that the paths stay outside $B(\tilde v, e^{-t-1})$ without writing it explicitly. We will also assume without loss of generality that the parametrisations of $\tilde \gamma_{\tilde v}$ and $\tilde Y^{\tau_{i_N}}$ are identical up to the first point where their traces differ.

If a non-contractible cycle is created or the boundary is hit in the interval $[\tau_{i_N}, \tau_{(i_N +1)}]$ then, since $\tilde X[\tau_{i_N}, \tau_{(i_N +1)}]$ does not intersect $B(\tilde v, e^{-t_0 -1})$, only pieces of $\tilde Y^{\tau_{i_N}}[0, \lambda_2]$ can be erased and nothing can be added  {in the time interval $[\tau_{i_N}, \tau_{(i_N +1)}]$} to get $\tilde \gamma_{\tilde v} [0, t_2]$. %\note{during which times?}.
 In particular we see that in this case
$$
\max_{t_1<s_1<s_2<t_2}   |W(\tilde \gamma_{\tilde v}(s_1,s_2),\tilde v)| \leq \max_{\lambda_1<s_1<s_2<\lambda_2}   |W(\tilde Y^{\tau_{i_N}}(s_1,s_2),\tilde v)|,
$$
and we are done using \eqref{eq:YiN}.
%\note{DANGER BEWARE: If we are looking at the last exit of $e^{-t-1}$ then this is not stable by inclusion so the interval $[t_1, t_2]$ might not be included in $ \lambda_1, \lambda_2$. This is why I change the beginning from last to first.}

\smallskip

The only other possibility is that a non-contractible cycle is created between $\tau_{({i_N} + 1)}$ and $\tau_{i_{(N + 1)}}$ (otherwise that would contradict the maximality of $N$). Since $p$ is injective in $B(\tilde v, e^{-t_0})$, this occurs if and only if the walk hits a copy inside $B( \tilde v , e^{-t_0})$ of a portion of the walk that is further away from $\tilde v$, as illustrated schematically in \cref{fig:lambda}. From now on we assume we are on this event.

First we make a topological observation. Let $\beta$ be the first exit time of $B( \tilde v, e^{-t_0})$ by  $\tilde Y^{\tau_{i_N}}$. We claim that $\tilde X[ \tau_{i_N+1}, \tau_{nc} ]$ cannot hit $\tilde Y^{\tau_{i_N}}[0, \beta]$. %, as this would contradict our assumption on $N$.
%\note{ should this be $\tilde X[ \tau_{i_{N} +1} ...]$ and not $X[ \tau_{i_N} + 1 ...]$?} \noteb{yes fixed.}
%\note{I am also confused here. By definition, during $[\tau_{i_N}, \tau_{i_{N +1}}] $ the walk stays inside $B(\tilde v, e^{-t_0})$ where $p$ is injective. So how can a non-contractible loop be formed in this interval?}
Indeed, suppose by contradiction that it does so at some time $T \in [ \tau_{i_N+1}, \tau_{nc} ]$. Then, after erasure, we are left with a path completely contained in $B( \tilde v, e^{-t_0})$ where $p$ is injective. No non-contractible loop can then be created before $\tau_{i_{N+1}}$, which contradicts the maximality of $N$ as explained above.

%By definition of the universal cover, if $\tilde X [\tau_{{i_N} + 1}, \tau_{i_{(N + 1)}} ]$ hits any vertex in $\tilde Y^{\tau_{i_N}}$ then the loops created that way are contractible.
 % {Note that by construction of $\beta$, $\tilde Y^T \subset B(\tilde v, e^{-t_0})$}. %\note{ditto}
%  Also $\tilde X [T, \tau_{i_{(N + 1)}} ] \subset \tilde X [\tau_{{i_N} + 1}, \tau_{i_{(N + 1)}} ] \subset B(\tilde v, e^{-t_0})$. Thus in this case $\tilde Y^{\tau_{i_{N+1}}} \subset B(\tilde v, e^{-t_0})$.  Since the projection map is injective on $B(\tilde v, e^{-t_0})$, no non-contractible loop can be created in $[T,\tau_{i_{N+1}}]$. This contradicts our assumption on $N$.
  %\note{I have re-read this five or six times and I still don't understand the argument. It seems all you are saying is that the ball $B(\tilde v, e^{-t_0})$ is small enough that $p$ is injective in it, so it's impossible to make a nontrivial loop while we are in it. Doesn't that mean it's impossible to have a nontrivial loop in this interval of time? }\noteb{I have rephrased a few things and underlined the key point here. Hopefully it makes more sense?}

%\note{what event, sorry? Ok, not much point continuing here...} \note{We defined the event two para before and we assume we are on this event from now on. Hopefully this is clearer... }

Let $k_{nc}$ be the index during which the non-contractible loop happens, i.e, the unique index $k$ such that $\tau_{nc} \in [\tau_{k}, \tau_{k+1}]$.  By the topological claim in the previous paragraph, no full turn can occur in any of the intervals $[\tau_k, \tau_{k + 1}]$ for $i_N \leq k < k_{nc}$. However, by Corollary 4.5 in \cite{BLR16}, a full turn can occur in any interval $[\tau_k, \tau_{k + 1}]$ independently with uniformly positive probability given $\cS$. Hence, there exist constants $c, C$ depending only on the crossing estimate such that
\begin{equation*}
\P(k_{nc} - i_N > n) \le \P(k_{nc} - i_N > n, N \le n) + \P(N >n) \le Ce^{-cn} .
\end{equation*}
Combining the above with Lemma 4.7 in \cite{BLR16} (which bounds the winding of the random walk during an interval of the type $[\tau_i, \tau_{i + 1}]$), we obtain the following stretched exponential tail bound:
\begin{equation}
\P\big (\max_{\cY \subset \tilde X[\tau_{i_N}, \tau_{nc}]}  |W (\cY ,\tilde v)| >n \big) \le Ce^{-c\sqrt{n}}\label{eq:RW_portion_added}
\end{equation}
for some $C,c>0$ depending only on the crossing estimate and where, as before, $\cY$ is any continuous portion obtained from $X[\tau_{i_N}, \tau_{nc}]$ which preserves the order of the random walk path. In particular, this gives a good control of the winding of the piece of $\tilde \gamma_{\tilde v}$ added to $Y^{\tau_{i_N}}$.

%\noteb{
%Actually let us remark on something slightly stronger. Writing $i_{\beta}$ for the index of $\tau$ where $\tilde Y^{\tau_{i_N}}[0,\beta]$ is added (i.e. not erased later), the argument above actually shows that $i_{nc} - i_{\beta}$ has an exponential tail. Indeed loops inside $B(\tilde v, e^{-t_0})$ are also forbidden between $i_{\beta}$ and $\tau_{i_N}$ because they would erase $\tilde Y^{\tau_{i_N}}_{\beta}$.
%} \note{do we need it?}

\begin{figure}[h]
\centering
\includegraphics[width = 0.75\textwidth]{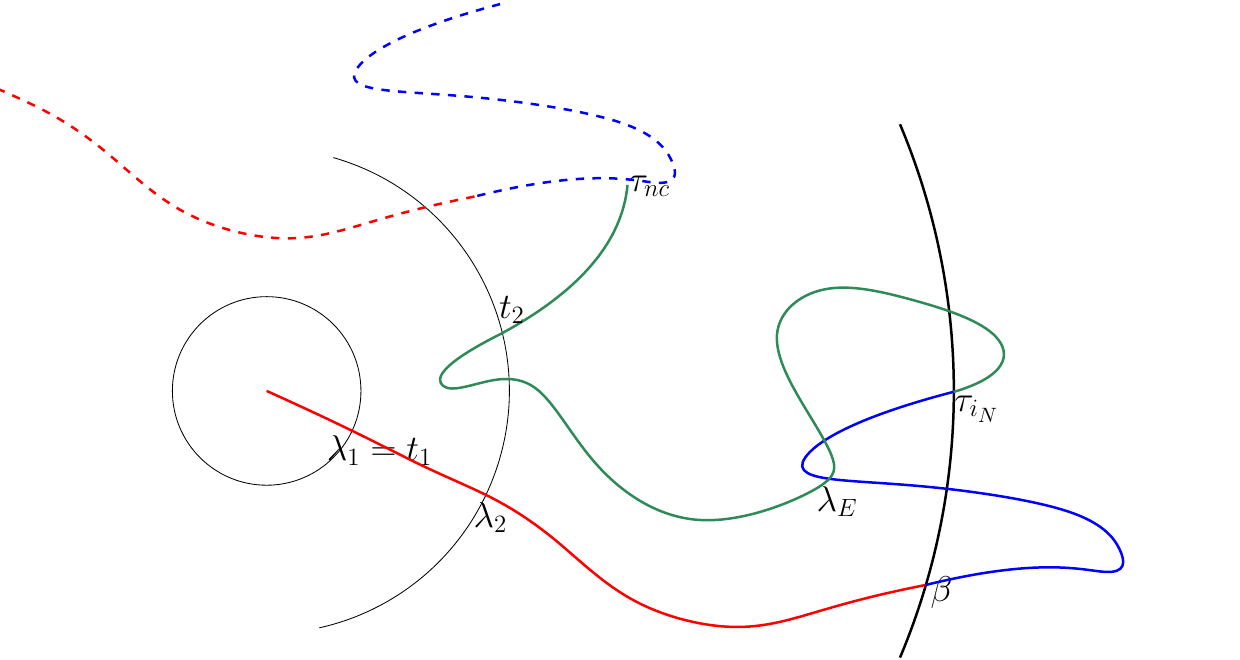}
\caption{The concatenation of the red and blue curves is $\tilde Y^{\tau_{i_N}}$ and the green part is $\tilde X[\tau_{i_N}, \tau_{nc}]$. The dotted path denotes a copy of $\tilde Y^{\tau_{i_N}}$ so overall the event represented is one where a non-contractible loop is formed inside $B(\tilde v, e^{-t_0})$.}\label{fig:lambda}
\end{figure}

%Let $\beta_1, \beta_2, \beta_3$ denote the \emph{first} times $\tilde Y^{\tau_{i_N}}$ exits respectively $B(\tilde v , e^{-t -1})$, $B(\tilde v , e^{-t})$ and $B(\tilde v , e^{-t_0})$ and let $\tau_\beta$ and $\i_\beta$ denote respectively the random walk time where $Y^{\tau_{i_N}}_{\beta_3}$ was added to $Y^{\tau_{i_N}}$ and the index of the corresponding piece of random walk. The the same argument as in the last paragraph shows that $i_{NC} - i_{\beta}$ has exponential tail

However, we are still not done as illustrated in \cref{fig:lambda}: the times $t_1,t_2$ (which were first entry, last exit times for $\tilde \gamma_{\tilde v}$) may be different from  $\lambda_1,\lambda_2$ (which were the first entry, last exit of $\tilde Y^{\tau_{i_N}}$), so we need additional arguments.    Let $$\lambda_{E} := \inf \{ \lambda : \exists t \in [\tau_{i_N}, \tau_{nc}] , \tilde Y^{\tau_{i_N}}(\lambda) = \tilde X(t) \} $$ be the last time in $\tilde Y^{\tau_{i_N}}$ which is not erased. Note that with this notation we showed above that $ \lambda_E \geq \beta$. This implies immediately that $t_1 = \lambda_1$. Now we need to consider several cases depending on where $\lambda_E$ is in relation to $\lambda_2$.
%This implies immediately that $\tilde Y^{\tau_{i_N}}[0, \beta] \subset \tilde \gamma_{}

% and that this implies immediately that
%and that by definition $\lambda_1 \geq \beta_1$.

%Let also $\beta_1$ be the first $\tilde Y^{\tau_{i_N}}$ exits $B(0, e^{-t-1})$.

\begin{itemize}
\item If $\lambda_E \in [\lambda_1, \lambda_2] $, then $\tilde \gamma_{\tilde v}[t_1, t_2]$ can be decomposed as the union of a piece of $ \tilde Y^{\tau_{i_N}}[\lambda_1, \lambda_2]$ and some erasure of $X[\tau_{i_N}, \tau_{nc}]$, so on that event
$$
\max_{t_1<s_1<s_2<t_2}   |W(\tilde \gamma_{\tilde v}(s_1,s_2),\tilde v)| \leq \max_{\lambda_1<s_1<s_2<\lambda_2}   |W(\tilde Y^{\tau_{i_N}}(s_1,s_2),\tilde v)| + \Delta,
$$
where $\Delta$ is a variable with stretched exponential tail using \eqref{eq:RW_portion_added}.

%\item Let $t_ E  = \inf\{t \in [\tau_{i_N}, \tau_{nc}] : Y^{\tau_{i_N}}(\lambda_E) = \tilde X(t)  \}$. If $\lambda_E \geq \lambda_2$ and the $\tilde X[t_E, \tau_{nc}]$ does not enter $B(\tilde v, e^{-t})$ then $\tilde \gamma_{\tilde v}[t_1, t_2] = \tilde Y^{\tau_{i_N}}[\lambda_1, \lambda_2]$ and there is nothing to prove.

\item  If $\lambda_E \geq \lambda_2$ and $\tilde \gamma_{\tilde v}(\lambda_E,t_2]$ (the loop erasure after $\lambda_E$) does not enter $B(\tilde v, e^{-t})$ then $\tilde \gamma_{\tilde v}[t_1, t_2] = \tilde Y^{\tau_{i_N}}[\lambda_1, \lambda_2]$ and there is nothing more to prove.

\item If $\lambda_E \geq \lambda_2$ and $\tilde \gamma_{\tilde v}(\lambda_E,t_2]$ enters $B(\tilde v, e^{-t})$ (this is the case pictured in \cref{fig:lambda}), then we decompose $\tilde \gamma_{\tilde v}$ as the union of $\tilde Y^{\tau_{i_N}}[\lambda_1, \lambda_2] $, $\tilde Y^{\tau_{i_N}}[\lambda_2, \lambda_E]$ and $\tilde \gamma_{\tilde v}[\lambda_E, t_2]$. The winding of any continuous portion of the first part has exponential tail by \eqref{eq:YiN}. The winding of any continuous portion of the last part has stretched exponential tail since it is bounded by $\Delta$ as in the first case. So we only need to take care of the middle part.

To that end, we decompose $\tilde Y^{\tau_{i_N}}[\lambda_2, \lambda_E]$ into excursions inside and outside $B_0$ as follows:
\begin{equation}
\tilde Y^{\tau_{i_N}}[\lambda_2, \lambda_E] = \tilde Y^{\tau_{i_N}} [\lambda_2, \beta_0] \cup \tilde Y^{\tau_{i_N}} [\beta_0, \beta_1]   \cup \tilde Y^{\tau_{i_N}} [\beta_1, \beta_2] \cup \ldots \cup  \tilde Y^{\tau_{i_N}} [\beta_{k_0}, \lambda_E],\label{eq:decomposition}
\end{equation}
where $\beta_0 = \beta$ and for $k \ge 1$, $\beta_{2k}$ is the first exit of $B(\tilde v, e^{-t_0})$ after $\beta_{2k-1}$ and $\beta_{2k-1}$ is the first entry into $B(\tilde v, e^{-t_0 -1})$ after $\beta_{2k-2}$. Note that any portion of $\tilde Y^{\tau_{i_N}} [\beta_{2k}, \beta_{2k+1}]$ is outside $B(\tilde v, e^{-t_0-1})$ so its winding has exponential tail using \cref{lem:macro_winding}. Note also that $\tilde Y^{\tau_{i_N}} [\beta_{2k+1}, \beta_{2k+2}]$ lies inside the annulus bounded between $C_t$ and $C_0$ and never intersects the curve $\tilde X[\tau_{i_N}, \tau_{nc}]$ by maximality of $\lambda_E$. Let $Y'$ be the loop erasure of the portion of $\tilde X$ from $\tau_{i_N}$ until its first hit of $C_t$. Hence using \cref{lem:wind_comp}
\begin{equation*}
\max_{\beta_{2k}<s_1<s_2<\beta_{2k+1}}   |W(\tilde Y^{\tau_{i_N}}(s_1,s_2),\tilde v)| \leq \max_{s_1<s_2}   |W(Y'[s_1,s_2],\tilde v)| +4\pi.
\end{equation*}
Note that the right hand side has stretched exponential tail by \eqref{eq:RW_portion_added}. Thus each term in the decomposition \eqref{eq:decomposition} has stretched exponential tail and clearly the number of terms is bounded by $N$ which itself has exponential tail. Combining these two, we can conclude that the absolute value of winding of any continuous portion of $\tilde Y^{\tau_{i_N}}[\lambda_2, \lambda_E] $ has stretched exponential tail (with exponent $1/3$), thereby completing the proof of this case.
\end{itemize}
This concludes the proof of \cref{L:winding_tail_partial}.
\end{proof}
%\subsubsection{Extension to Temperleyan forests}
\subsection{From partial path to the full spine.}\label{sec:full_spine}

As shown in \cref{lem:number_crossing}, the path $\tilde \gamma_{\tilde v}$ defined above is only a part of what is needed to compute the height function. Recall the notion of spine from \cref{sec:height_winding}.% which is just a bi-infinite path in the cover. As observed before, 
and that when we follow the outgoing edges in the Temperleyan forest from any given $\tilde v$, we always end up in a unique spine (since boundary loops have been added around every hole).

Recall from \cref{lem:number_crossing} that we are interested in the winding of the path starting from $\tilde v$, and then moving along the spine to infinity or the boundary of the disc. Observe that the initial portion of this path is $\tilde \gamma_{\tilde v}$ (\cref{fig:gamma_tilde}) and then it moves along copies of the non-contractible loop in the component of $\tilde v$ in the Temperleyan CRSF. In this section, we will call this path $ \mathsf p_{\tilde v}$ (i.e., $\tilde \gamma_{\tilde v}$ followed by a semi-infinite piece of the spine). Observe that this notation is in contrast to the notation used in \cref{lem:number_crossing}, where the path $\mathsf{p}_{\tilde v}$ was called
$\gamma_{f}$, whereas here we emphasise that $\tilde \gamma_{\tilde v}$ denotes what is sampled from Wilson's algorithm given the skeleton. When we apply Wilson's algorithm to sample $\tilde \gamma_{\tilde v}$, given the skeleton, we may form a new non-contractible loop, in which case we have discovered a portion of a spine (the unique spine attached to $\tilde v$). The winding of this portion is controlled by \cref{L:winding_tail_partial}. However, we also need to control the winding of the rest of $\mathsf p_{\tilde v}$. The purpose of this section is precisely to achieve this (done in \cref{L:winding_tail}). We also need to control the winding of the other semi-infinite piece of the spine attached to $\tilde v$, which is done in \cref{cor:full_winding_moment}. Finally, if we want to control the height gap between $\tilde u$ and $\tilde v$, we need to control the winding of $\mathsf p_{\tilde u}$ around $\tilde v$ as well, which is done in \cref{lem:tildeu}.

%can be sampled using Wilson's algorithm, and the whole spine. \note{If I'm not mistaken that notation has changed in Theorem \ref{lem:number_crossing}.} \note{In that Section, $\gamma$ denoted the path connecting two points in the cover. The point we are trying to make here is that the winding of the whole branch cannot be computed using Wilson and we need to treat separately. } \noteb{$p_{\tilde v}$ is what was called $\gamma$ there. But now $\gamma$ comes from Wilson.}
 \medskip Let us parametrise $\mathsf p_{\tilde v}$ in $[0,\infty)$ in any continuous manner away from $\tilde v$. Let $t_0$ be as in \cref{sec:apriori}, i.e., $e^{-t_0} = \tilde r_{\tilde v}$.
%\note{Not again!!!}\note{reintroduced}
For $t \ge t_0$, let $t_1$ be the first exit time of  $\mathsf p_{\tilde v}$ from $B(\tilde v, e^{-t-1})$ and let $t_2$ be the last exit time of the same from $B(\tilde v, e^{-t})$. Note again that here, it is possible that the unique spine attached to $\tilde v$ could be the same as the spine corresponding to a skeleton branch. %\note{What is the ``spine" of $\tilde v$?}

\begin{lemma}\label{L:winding_tail}
For all $k \ge 1$, there exists a constant $m>0$ so that for all $\delta < \delta_{\bar B(\tilde v, e^{-t_0})}$ and for all $ t_0 \le t < \log (C'\delta_0 / \delta)$,
$$
\E\left(\max_{\cY \subseteq   \mathsf p_{\tilde v} [t_1,t_2] \cap B(\tilde v, e^{-t-1})^c}   |W(\cY,\tilde v)|^k)
 \right) \le m,
$$
where the supremum is taken over all continuous segments. Also
$$
\E\left(\max_{\cY \subseteq   \mathsf p_{\tilde v} \cap B(\tilde v, e^{-t_0})^c}   |W(\cY,\tilde v)|^k)
 \right) \le m.
 $$
\end{lemma}
%Maximising over $t>0$, we obtain a stretched exponential bound uniform over $t$:
%\begin{equation*}
%\P(\max_{\cY \subseteq   \mathsf p [t_1,t_2] \cap B(\tilde v, e^{-t-1})^c}   |W(\cY,\tilde v)| >n
% ) <Ce^{-c'n^{\alpha'}}.
%\end{equation*}
%for some $C,c',\alpha'>0$ independent of $t$.

\begin{proof}
In this proof, we write $\mathsf p_{\tilde v} = \mathsf p$ to lighten notation. Let us first consider the case when $\mathsf p$ does not contain a portion of the spine corresponding to the skeleton.

% {\color{gray}Let us first consider the case when $\mathsf p$ contains a portion of the spine corresponding to a special branch.
% %\note{$\mathsf{p}_{\tilde v}$, replace below as well}.
% %\note{You mean $\tilde v$ on one of the special branches? Or $\mathsf p_{\tilde v}$ intersects a special branch?} \note{ rephrased. Also we do not use the subscript $\tilde v$ in this proof, so I think it is good to use this lighter notation}
% In that case, for any portion of $\tilde \gamma_{\tilde v}$, we use \cref{L:winding_tail_partial} and for the rest of the spine, we use \cref{assumption_precise} to obtain the required bound.}

%be such that $p$ is injective in $B(\tilde v, r)$. As in \cref{L:winding_tail_partial}, we shall prove
%$$
%\E(|\max_{\cY \subseteq   \mathsf p_{\tilde v} \cap B(\tilde v, e^{-t_0})^c}   |W(\cY,\tilde v)| |^k)
%$$
%and for $t \ge t_0$,
%$$
%\E(|\max_{\cY \subseteq   \mathsf p_{\tilde v} [t_1,t_2] \cap B(\tilde v, e^{-t-1})^c}   |W(\cY,\tilde v)| |^k)
% ) \le m.
%$$
%which are stronger statements.

%\note{Recall that I am assuming for now that we are on the torus.}

%We first focus on $\mathsf p = \mathsf p(\tilde v)$ and $t > t_0$ (that is, we consider portion of the path inside $B(\tilde v,e^{-t_0})$).
%Recall that $\mathsf p(\tilde v)$ is the path obtained by iteratively moving along the outgoing edge from every vertex as in \cref{sec:winding}.

%\note{reformulated above and below}.
We parametrise $\tilde \gamma_{\tilde v}$ in $[0,1]$ as before and assume that the parametrisation of $\tilde \gamma_{\tilde v}$ and $\mathsf p$ is the same until they start to differ.
We drop $\tilde v$ and write $\tilde \gamma_{\tilde v} = \gamma$ throughout the rest of this proof for notational clarity. To differentiate the first and last exit times for these two parametrisations we write  $t_1 (\mathsf p)$ (resp. $t_2(\mathsf p)$) for the first exit of $B(\tilde v, e^{-t-1})$ (resp. last exit of $B(\tilde v, e^{-t})$) of $\mathsf p$. We also write $t_1 = t_1(\gamma)$ and $t_2 = t_2(\gamma)$ for clarity.

Note that it is always the case that $t_1(\gamma) = t_1(\mathsf p)$ %\note{$\tilde \gamma$, below as well...} \note{notation simplified as written at the end of previous para}
and that $t_2( \gamma) \leq t_2(\mathsf p)$. On the event $t_2(\mathsf p ) = t_2( \gamma)$, the maximum in this lemma is over the same set as the one in \cref{L:winding_tail_partial} so there is nothing to prove. We focus now on the case where $t_2(\mathsf p ) > t_2( \gamma)$, meaning that a part of the spine comes back close to $ x$ on the cover $\tilde M'$ after the time that a non-contractible loop was created.

We start with the case of the torus.
Let us denote by $S$ the spine attached to $ \tilde v$ (meaning the full bi-infinite path). Since $S$ is a periodic path, it has a well defined direction $d$ (say represented by a unit vector in $\R^2$) and separates the plane in two sets right and left of $S$. Let us assume without loss of generality that the direction $d$ is horizontal and that $ \tilde v$ is below $S$. Let $\tau$ be a time at which the vertical coordinate of $\gamma$ reaches its maximum (for topological reasons, this has to occur on the spine $S$) and let us define $\gamma'$ by appending to $\gamma[0, \tau]$ a vertical segment going up to infinity.
%\note{I don't understand. First this segment will intersect $S$, right? Also, do you view $\gamma'$ as a curve? Does it come back from infinity? etc}

Now it is easy to see that the winding of a straight segment around a point is bounded by $\pi$, so
\begin{equation}
 \max_{\cY \subseteq  \gamma' \cap B(  \tilde v, e^{-t-1})^c}   |W(\cY, \tilde v)| \leq \max_{\cY \subseteq  \gamma \cap B( \tilde v, e^{-t-1})^c}   |W(\cY, \tilde v)| + \pi.\label{eq:replacement}
\end{equation}
Using \cref{lem:wind_comp}, this implies that we can always bound \begin{equation}
\max_{\cY \subseteq   {\mathsf p_{\tilde v}}  \cap B(  \tilde v, e^{-t-1})^c}   |W(\cY, \tilde v)| \leq  \max_{\cY \subseteq  \gamma' \cap B(  \tilde v, e^{-t-1})^c}   |W(\cY, \tilde v)| +4\pi \le \max_{\cY \subseteq  \gamma \cap B( \tilde v, e^{-t-1})^c}   |W(\cY, \tilde v)| + 5 \pi.\label{eq:asymp_direction}
\end{equation}
By \cref{L:winding_tail_partial} (the $\tilde \gamma_{\tilde v}$ there is  $\gamma$ here) applied to all scales up to scale $t$, the right hand side is a sum of $O(t)$ variables with uniform stretched exponential tails.
%\note{I think paramaterisation arbitrary so you probably mean going from one scale to another instead.}

Let us now bound $\P(t_2( {\mathsf p} ) > t_2(\gamma); \mathsf{p} \cap \mathfrak{s} = \emptyset)$. Note that on the event $t_2 (\mathsf p) > t_2(\gamma)$, $ \tilde v$ has to be at distance less that $e^{-t}$ of its spine. Using the fact that we can sample points in any order in Wilson's algorithm, we can bound that probability by first doing a cutset exploration around $\tilde v$ at scale $r_{\tilde v}$.  After this cutset exploration, the probability that any of the branches sampled intersects $B(\tilde v, e^{-t})$ is at most $C e^{-c(t - t_0)}$ for some constants $C, c$ (see \cref{lem:cutset isolation}). However, after the cutset exploration around $\tilde v$, the spine attached to $\tilde v$ is necessarily fully sampled (see e.g., \cref{lem:exp_tail}, final assertion). Hence
\begin{equation}
\P(t_2( {\mathsf p} ) > t_2(\gamma); \mathsf{p} \cap \mathfrak{s} = \emptyset) \le C e^{- c(t-t_0)}.\label{eq:spine_bound}
\end{equation}
Thus overall, either $\max_{\cY \subseteq   {\mathsf p  [t_1,t_2]} \cap B(  \tilde v, e^{-t-1})^c}   |W(\cY, \tilde v)|$ is the same as the variable in \cref{L:winding_tail_partial}, or with probability at most $Ce^{-c(t-t_0)}$, it is at most a sum of $O(t)$ variables with uniform stretched exponential tail. The moment bound now easily follows.

For the case where the universal cover is the unit disc, a similar argument can be done provided we can construct a path $\gamma'$ which is nice and avoids the spine. Let us also introduce $\tau_S$ as the first time $\gamma$ hits the spine and $\tau_{nc}$ as the ending time of $\gamma$.
As recalled in more details in \cref{app:spine},
it is a well known fact from Riemann surfaces that there exists a M\"obius transform $\phi = \phi_{\gamma,\tilde v}$ such that
$${\mathsf p} ( v) = \gamma[0,\tau_S] \cup \bigcup_{n \geq 0} \phi^{(n)}((\tau_S, \tau_{nc} ])
$$
where $\phi^{(n)} = \phi \circ \ldots \circ \phi$ and the union is disjoint. In other words, we keep applying the same M\"obius transform to obtain all the copies. Furthermore $\phi$ can be written as $\phi= \Phi^{-1} \circ \mu \circ \Phi$ where $\Phi$ is a M\"obius map from the unit disc to the upper half plane and $\mu$ is either the translation by $\pm 1$ or a multiplication by $\lambda > 1$.

If it is a translation we use the same argument as before. Otherwise, it is a scaling by $\lambda>1$ in which case $\phi(S)$ necessarily converges to infinity. Furthermore, let $\tau$ be the time at which $\Im(\Phi( \gamma ))$ reaches its minimum, we define $\gamma'$ by appending a straight vertical segment $\ell$ from $\Phi(\gamma(\tau))$ to $\R$ (note that this segment may not intersect $\Phi(\gamma)$ nor the subsequent scalings of the image of the portion of the spine $\Phi (\gamma([\tau_S, \tau_{nc}])$) and then mapping $\Phi(\gamma[0,\tau] \cup \ell)$ back to the disc by $\Phi^{-1}$. By construction, it is trivial to check that $\Phi(\gamma')$ does not intersects $\Phi(  {\mathsf p}  (\tau, \infty) )$ so $\gamma'$ does not intersect $ {\mathsf p} (\tau, \infty)$. On the other hand, $\Phi^{-1}$ is the image of a line segment by a M\"obius transform so it is a circular arc and therefore its winding around any point is bounded by $2 \pi$. We can then conclude using exactly the same reasoning as in the torus case, with only the constant $\pi$ replaced by $2 \pi$ on the right hand side of \eqref{eq:replacement} and hence obtain the analogue of \eqref{eq:asymp_direction}.

{Finally, let us consider the case when $\mathsf p$ contains a portion of the spine corresponding to the skeleton. Note that the whole argument was geometric and the only place where we used the probability was in \eqref{eq:spine_bound} which is provided by \cref{lem:special_coming_close}. }
\end{proof}

%If $t = t_0$, we use \eqref{eq:asymp_direction} (and its analogue in the hyperbolic case) and employ \cref{L:winding_tail_partial} with $t = t_0$.

%Finally, if $\mathsf p = \mathsf p(\tilde u)$, we use \cref{lem:wind_comp} to bound
%\begin{equation}
%\max_{\cY \subseteq   {\mathsf p} ( \tilde u) \cap B(  \tilde v, e^{-t-1})^c}   |W(\cY, \tilde v)| \leq \max_{\cY \subseteq  \mathsf p(\tilde v) \cap B( \tilde v, e^{-t-1})^c}   |W(\cY, \tilde v)| + 2 \pi.
%\end{equation}
%and conclude.

%\begin{corollary}\label{cor:tildeu}
%
%\end{corollary}

\begin{lemma}\label{cor:full_winding_moment}
For all $k \ge 1$, there exist a constant $C > 0$ such that the following holds. Fix a compact set $K \subset \tilde M'$ .
Suppose we are in the setup of \cref{L:winding_tail}. Let $\mathsf p = \mathsf p_{\tilde v}$ be as above and $ \bar{\mathsf p}$ be the curve which starts at $\tilde v$, hits the spine, and then goes to infinity in the direction opposite to the orientation of the spine. Orient and parametrise both $\mathsf p, \bar {\mathsf p}$ from $\tilde v$ to $\infty$ so that $t_2$ is the last time they exit $B(\tilde v, e^{-t})$. Then for all $ \delta < \delta_{K}, t \in[t_0,\infty], \tilde v \in K$,
\begin{align*}
\E\Big(\sup_{\cY \subset \mathsf p[t_2,\infty) }|W(\cY , \tilde v)|^k \Big) &\le C((1+t) \wedge \log (1/\delta))^{k}, \\  \E \Big(\sup_{ \cY \subset \mathsf {\bar p}[t_2,\infty)}|W(\cY , \tilde v)|^k \Big) &\le C((1+t) \wedge \log (1/\delta))^{k}.
\end{align*}
\end{lemma}
We emphasize here that the spine might be identical to a spine of a skeleton branch.
\begin{proof}
First we tackle the case of $\mathsf p$. If $t < \log (C'\delta_0 / \delta)$, we need to add $O(t)$ many variables given by \cref{L:winding_tail} and hence we are done by Minkowski's inequality. If $t > \log (C'\delta_0 / \delta)$, first we apply the previous bound up to $t =  \log (C'\delta_0 / \delta)$. For the remainder, we can simply bound the winding by the volume of the ball of radius $C'\delta / \delta_0$ around $\tilde v$ which is $O(1)$ by our assumptions on the graph (see assumption \eqref{boundeddensity} in Section \ref{sec:setup}). Indeed, this quantity is bounded by the number of times $\mathsf p$ crosses a straight line joining $\tilde v$ to the boundary of the ball, which is simply bounded by the volume of the ball. For $\mathsf {\bar p}$, we can use \cref{lem:wind_comp} and the bound for $\mathsf p$ and the proof is complete.
\end{proof}

%
%
%Notice that for $\mathsf p$, the winding of the curve from infinity until it intersects $B(\tilde v, e^{-t-t_0})$ can be written as a sum of $t$ terms each of which is upper bounded by the terms in \cref{L:winding_tail}. This completes the proof by an easy bound on the moment of sum of random variables. For $\tilde p$, we simply employ \cref{lem:wind_comp} and use the bound obtained on $\mathsf p$.
\begin{lemma}\label{lem:tildeu}
For all $k$ there exist $C,c>0$ such that for all points $\tilde u, \tilde v$, for all compact sets $\tilde K$ containing $\tilde u , \tilde v \in \tilde M'$ , for all $\delta  < \delta_{\tilde K}$,
$$
\E\Big ( |W(\mathsf p_{\tilde u}  , \tilde v  )|^k\Big) \le C(|1+\log  |\tilde u - \tilde v|| \wedge \log (1/\delta))^c.
$$
\end{lemma}
\begin{proof}
Fix $T =  \log (\frac{C'\delta_0}{\delta})$ with $C'$ as in \cref{L:winding_tail}. We parametrise $\mathsf p_{\tilde v}, \mathsf p_{\tilde u}$ as in \cref{cor:full_winding_moment}.      If $|\tilde u - \tilde v| <e^{-T}$, we use  \cref{cor:full_winding_moment,lem:wind_comp} to conclude.

%we first apply \cref{L:winding_tail} to bound the $k$th moment of $W(\mathsf p_{\tilde v}[T_2, \infty), \tilde v)$ by $CT^{2k} \le C\log^{2k} |\tilde u - \tilde v|$. Then notice that $W(\mathsf p_{\tilde v}(T, \infty), \tilde v)$ is deterministically bounded by the number of vertices in $B(\tilde v, e^{-T})$ which is bounded by a universal constant according to our assumption on the graph. Indeed, $W(\mathsf p_{\tilde v}(T, \infty), \tilde v)$ is bounded by the number of times it crosses a straight line joining $\tilde v$ to the boundary of the ball, which is simply bounded by the volume of the ball.

%Finally, we use \cref{lem:wind_comp} to say that $W(\mathsf p_{\tilde u}, \tilde v)$ is deterministically  bounded by $W(\mathsf p_{\tilde v}, \tilde v)$and conclude.

If $|\tilde u - \tilde v| \ge e^{-T}$, we take $t  = -\log |\tilde u - \tilde v| $. Then using \cref{lem:wind_comp}, we can write
$$
|W(\mathsf p_{\tilde u},  \tilde v)| \le 2\pi +  \sum_{k=t}^\infty \sup_{ \cY \subset \mathsf p_{\tilde v}[k_2, \infty)} |W(\cY, \tilde v)|1_{|\mathsf p_{\tilde u} -\tilde v| \in [e^{-k}, e^{-k-1} ]}
$$
where $k_2$ is the last exit from the $B(\tilde v, e^{-k})$ by $\mathsf p_{\tilde u}$.
Using \cref{lem:coming_close} or \cref{lem:special_coming_close} whichever is appropriate, we get an exponential bound on the expectation of the indicator event above (notice for scales $k > \log (C'\delta_0/\delta)$, the probability actually becomes 0 so this is a finite sum). Therefore, we can again use \cref{cor:full_winding_moment} and Cauchy--Schwarz to conclude.
\end{proof}

%\subsubsection{Extension to Temperleyan forests} \note{Probably we do not need this..}

\subsection{Convergence of height function}\label{sec:winding_convergence}

%\note{Remove the assumptions in all the lemmas and add the proof of the main theorem}

In this section we prove \cref{thm:main_precise}.

Let us describe informally the general structure of the proof. Most of the work will be in the universal cover to obtain the convergence of expressions of the form $\E \prod_i (h(\tilde z_i) - h(\tilde w_i))$ for distinct points $\tilde z_i$ and $\tilde w_i$ (\cref{L:err_winding_moment}). By integrating this expression, we will then obtain the convergence of the height function, seen as a function on the universal cover. 

For the first part, the idea is to introduce a regularised height function $h^\d_t$ which is a continuous function of the CRSF so that the convergence of $h^\d_t$ as $\delta \rightarrow 0$ is immediate. This leaves us with two issues: the comparison between $h^\d$ and $h^\d_t$ for fixed $\delta$ and the convergence of $h_t$ as $t \rightarrow  \infty$ in the limit. Both questions are actually solved simultaneously by the estimate of \cref{L:err_winding_moment} which compares $h^\d$ and $h^\d_t$ with an error term that becomes small both in $\delta$ and $t$ independently.

\paragraph{Setup and notations.}  Recall that our goal is to prove that
\begin{equation}\label{E:height_test}
\int_{\tilde M'} \overline h_{\ext}^\d(z)\tilde f(z)d\mu(z)
\end{equation}
converges in law and also in the sense of moments as $\delta \to 0$. (Recall that $\bar X$ denotes the centered random variable $\bar X = X - \E(X)$ whenever this expectation is well defined). Implicitly, $h^\d$ is sampled according to the dimer law \eqref{Gibbs}. In other words, by \cref{prop:temp_bij}, we need to first sample a Temperleyan pair $(\cT, \cT^\dagger)$ and apply the bijection $\psi$ in that theorem.

However, it turns out to be more convenient in the proof to work with a pair $(\cT, \cT^\dagger)$ sampled from $\Pwils$.
%More precisely, we can instead first sample a CRSF $T$ according to the Wilson law $\Pwils$, and pick an oriented dual $T^\dagger$ with a conditional law given $T = t$ given by $\Ptemp (T^\dagger = \cdot | T = t)$. In other words, we pick an oriented dual of $T = t$ uniformly over all possibilities of orientation.
To this pair $(\cT, \cT^\dagger)$ we can apply the bijection $\psi$ from \cref{prop:temp_bij} and obtain a dimer configuration $\mathbf{m}$ whose centered height function can be studied.
We will first prove \eqref{E:height_test} for $\Pwils$ and then later explain how this implies the same for $\Ptemp$.

Recall that since $\int_{\tilde M'}\tilde f(z)d\mu(z) = 0$, the expression in \eqref{E:height_test} is in fact well defined.
We wish to compute the moments of this integral but only in terms of height differences since the field is a priori defined only up to constant.
To do this, we use the following trick. Note that since $\int \tilde f d\mu = 0$, $\int \tilde f^+ d \mu  = \int \tilde f^- d \mu =: Z(\tilde f)$ where $f^+ = \max\{f,0\}$ and $f^- = - \min \{f,0\}$. Now we can write
\begin{equation}
\int_{\tilde K} \overline h_{\ext}^\d(z)\tilde f(z)d\mu(z) = \int_{\tilde K \times \tilde K} (\overline h_{\ext}^\d(z) - \overline h_{\ext}^\d(w) )\frac{\tilde f^+(z)\tilde f^-(w)}{Z(\tilde f)}d\mu(z)d\mu(w),\label{eq:int_double_int}
\end{equation}
which implies
\begin{equation}
(\int_{\tilde K} \overline h_{\ext}^\d(z)\tilde f(z)d\mu(z) )^k = \int_{\tilde K^{2k}} \prod_{i=1}^k(\overline h_{\ext}^\d(z_i) - \overline h_{\ext}^\d(w_i)   )\frac{\tilde f^+(z_i)\tilde f^-(w_i)}{Z(\tilde f)}d\mu(z_i)d\mu(w_i).\label{eq:kth_moment_int}
\end{equation}
Therefore we are interested in the $k$-point function, $\E [ \prod_{i=1}^k(\overline h_{\ext}^\d(z_i) - \overline h_{\ext}^\d(w_i)   )]$. Pick $k$ distinct pairs of points $( z_1,w_1), \ldots, ( z_k ,w_k)\in \tilde K$ and let $(f(z_1),f(w_1)),
\ldots, (f(z_k),f(w_k))$ be the faces containing them. Let $z_i^\d,w^\d_i$ be the midpoint of the diagonals of $f(z_i), f(w_i)$.

Now recall \cref{lem:number_crossing} which relates the dimer height difference between two faces $f$ and $f'$ to the winding of a specific path $\gamma_{f,f'}$ connecting $m(f)$ and $m(f')$ and additional terms of the form $\pm \pi$ associated with jumping over components of the CRSF.
%$$
%h_{\dim}(f') - h_{\dim}(f) = W(\gamma, m(f)) + W(\gamma, m(f'))+ \pi\sum_{\sigma \in \bar \Omega_{S, S'}} (\ve_\sigma+\delta_\sigma)
%$$
Let $\gamma_i^\d$ be the path $\gamma_{f,f'}$ when $f = f(z_i)$ and $f'= f(w_i)$, and orient it
from $z_i^\d$ to $w_i^\d$.
%(recall that $z_i^\d$ and $w_i^\d$ are the midpoints of the diagonals of $f(z_i), f(w_i)$).
%Now recall the notations of \cref{lem:number_crossing}. Let $S(z_i^\d), S(w_i^\d)$ be the spines of the components containing $z_i^\d,w_i^\d$. Let $\mathsf p(z_i^\d)$ be the oriented (lift of) the branch to infinity starting from $z_i^\d$. In the hyperbolic case, let $\zeta(z_i^\d)$
%%\note{bad notation, should this be $\zeta(z_i)$? }
%be the limit point of $\mathsf p(z_i^\d)$ in the direction of its orientation. Let $\zeta(w_i^\d)$ be the limit point of $S(w_i^\d)$ in the same direction as that of $\mathsf p(z_i^\d)$ (see the discussion before \cref{lem:number_crossing} for more precise definition). Let $(\zeta(z_i^\d), \zeta(w_i^\d))$ be the arc in $\partial \tilde M$ joining $\zeta(z_i^\d), \zeta(w_i^\d)$ in the boundary of the component of $\tilde M$ between $S(z_i^\d), S(w_i^\d)$ (this arc could be a single point). Let $\gamma_i^\d$ be the concatenation of $\mathsf p(z_i^\d) , (\zeta(z_i^\d), \zeta(w_i^\d))$ and the part of the forest joining $\zeta(w_i^\d)$ to $ w_i^\d$. In the torus case, $\gamma_i^\d$ is simply the union of $\mathsf p(z_i^\d)$ and the path in the forest joining $w_i^\d$ to infinity in the same asymptotic direction. Orient $\gamma_i^\d$ from $z_i^\d$ to $w_i^\d$.
%\note{I personally don't think we need to recall all this notation, especially since below we just say $\Psi_i$ is the same as in \cref{lem:number_crossing}.}
Then with these notations, \cref{lem:number_crossing} says that
\begin{equation}
h(z_i^\d) - h(w_i^\d) = W(\gamma_i^\d, z_i^\d) + W(\gamma_i^\d, w_i^\d) + \Psi^\d_i\label{eq:path_height}
\end{equation}
where $\Psi^\d_i$ is the $\pi \sum_S (\ve_S+\delta_S)$ term in \cref{lem:number_crossing}.
%Recall the probability measure $\Pwils,\Ptemp$ from \cref{sec:Wilson}.
We drop the superscript $\delta$ from now on for clarity.
Thus from now on we focus on proving convergence of
\begin{equation}
\Ewils\Big(\prod_{i=1}^k (\bar W(\gamma_i ,z_i) + \bar W( \gamma_i ,w_i) + \bar \Psi_i )\Big).\label{eq:height_to_winding2}
\end{equation}
%where $\mathsf K$ is the number of dual non-contractible loops and $\bar X = X - \Ewils (X)$.
Let $$\cX = \{z_i,w_i:1 \le i \le k\}, \text{ and assume } p(u) \neq p(v) \text{ for all }u \neq v \in \cX.$$
%Denote
%\begin{equation}
%Z = \frac{2^{\mathsf K}}{\Ewils(2^{\mathsf K})}\label{Z}
%\end{equation}
%Observe that the number of primal non-contractible loops is the same as the number of dual non-contractible loops. Indeed, the branches divide the manifold into annuli and each special branch has two annuli on two sides. Further each annulus must contain a dual non-contractible loop. This in particular shows, in light of \cref{thm:CRSF_universal,remark:assumption} that for all $k \ge 1$ and $\delta$,
%\begin{equation}
%\E(Z^k) <\infty \label{eq:Z}
%\end{equation}

Clearly, \eqref{eq:height_to_winding2} can be expanded as a sum of $2^k$ many terms of the form:
\begin{equation}
\Ewils(\prod_{x \in S} (\bar W(\gamma_x,x) + \bar \Psi_x/2) )=: \Ewils(\prod_{x \in S} \bar F_x(\gamma_x))\label{eq:top_int_eq1}
\end{equation}
where $S$ is a subset of vertices of $\cX$ with distinct indices and $\gamma_x$ is $ \gamma_i$ for some $i$ such that $x = z_i$ or $w_i$ (of course, the products in the expansion of \eqref{eq:top_int_eq1} has further restrictions, but we ignore that for clarity).

We are interested in estimating and proving convergence of \eqref{eq:top_int_eq1}. To that end we employ an idea similar to that in \cite{BLR16}: we truncate the CRSF branch at a macroscopic scale and deal with the truncated macroscopic winding and the remaining microscopic winding part separately.

We now fill in the details.
Parametrise $\gamma = \gamma_x$. Define $(\gamma_x(t))_{t \ge 1}$ to be $\gamma_x$ from the first entry time into the ball $B(x,e^{-t})$ (note that at the moment, the balls $B(x,e^{-t})$ might overlap for different $x \in \cX$.) We emphasise here that we parametrise $\gamma$ in the opposite direction compared to \cref{sec:apriori,sec:full_spine} to be consistent with \cite{BLR16}. With an abuse of notation, we will denote by $\gamma_x (-\infty, t]$ the whole path from the opposite end of $\gamma$ up until $\gamma_x(t)$. Define the regularised term and the error term as
\begin{align}
\bar F_x(\gamma_x, t) & := \bar W(\gamma_x (-\infty, t]); x) + \bar \Psi_x/2\label{Fx}\\
\bar e_x(t) & := \bar W(\gamma_x,x) - \bar W(\gamma_x (-\infty, t]),x) = \bar F_x(\gamma_x) - \bar F_x(\gamma_x, t)\label{ex}.
\end{align}
When we want to emphasise the role of $z,w$, we write $\gamma_{zw}, \Psi_{zw},F_z(\gamma_{zw})$ in place of the above. We start with a general moment bound for the truncated winding.

%Recall that we can choose a positive constant $A$ such that for all $x,r$ with $B(x,r) \subset \tilde K$, all $t >0$, all $\delta< A\delta_0e^{-t}r \wedge \delta_{\tilde K}$ (where $\delta_0$ is as in the crossing assumption \ref{crossingestimate}) the crossing assumption is valid outside $B(x,e^{-t-1}r)$ and inside $B(x,r)$. Fix such an $A$ throughout the rest of the section.

\begin{lemma}\label{L:reg_winding_moment}
For all $m$, there exists constants $c = c( \tilde K,m),\alpha=\alpha(m)$ such that for all $z,w, \delta < \delta_{\tilde K}, t \ge 0$,
$$
\Ewils(|\bar F_{z}(\gamma_{zw}, t)|^m) \le c|(1+t  + |\log |z-w||)  \wedge \log (1/\delta)|^\alpha.
$$
%and
%$$
%\Etemp(|\bar F_x(\gamma_x, t)|^m) \le c(K,m)|(1+t + \log r)|^\alpha
%$$
\end{lemma}
\begin{proof}
The proof is immediate from \cref{cor:full_winding_moment,lem:tildeu} and the fact that $\Psi_{zw}$ is bounded by a constant $c(K)$ times the number of non-contractible loops in the CRSF. Indeed, this is clear from the fact that for a fixed compact set $\tilde K \subset \tilde M$ there exists a constant $c(\tilde K)$ such that any curve $P \subset \tilde K$ will cross at most $c(\tilde K)$ copies of a given spine (or lift of a loop). Since the number of non-contractible loops has superexponential tail by \cref{thm:CRSF_universal}, the moment bound of $\Psi_{zw}$ is immediate.
\end{proof}

By compactness, choose $r_{\tilde K}$ small enough so that $p$ is injective in $B(z,r_{\tilde K})$ for all $z \in \tilde K$. Now let $r_{\cX}$ be defined as in \eqref{def_r} but for the set of vertices $p(\cX)$ (which are all distinct by assumption on $\cX$). We observe that $r_{\cX} \ge  c(\tilde K) \min_{x\neq y \in \cX} |x-y|$ for some constant $c(\tilde K)$. We set
\begin{equation}
\Delta =\Delta (\cX) =  \frac1{10}(r_{\cX } \wedge r_{\tilde K}). \label{eq:Delta}
\end{equation}
\begin{lemma}\label{L:err_winding_moment}
There exist constants $c,c'>0$ such that for all $m,m' \ge 1$ there exists $\alpha>0$ such that the following holds.
Let $S\subset \cX, S' \subset \cX$ be disjoint.  Also assume that $|S| = m , |S'| = m'\ge0$.  Let $\Delta$  be as in \eqref{eq:Delta}. Then for all $\delta < \delta_{\tilde K}$ and $t \ge \log (r_{\tilde K} /\Delta)$
$$
|\Ewils(\prod_{x \in S}\bar e_x(t)\prod_{x \in S'}\bar F_x(t))| \le c |1+t |^\alpha(e^{-c(t- \log(r_{\tilde K}  / \Delta))} + \delta^{c'}).
$$
\end{lemma}
%We emphasise here that we do not enforce how $c$ is dependant on $m$ as in \cref{L:reg_winding_moment}.

\begin{proof}
For simplicity, we first assume $S'$ is empty. Recall that we can sample a CRSF from $\Pwils$ by first sampling the branches $\mathfrak s$ and then the rest by Wilson's algorithm. Note that \cref{lem:special_coming_close} tells us that the isolation radius corresponding of each vertex in $S$ with respect to the skeleton has polynomial tail. Since after sampling the skeleton, the rest of the branches are sampled simply by Wilson's algorithm, we can conclude that the application of \cref{lem:exp_tail} is valid with $\Delta$ in place of $r_i$.

We perform the coupling in \cref{sec:fullcoupling} (in particular \cref{lem:exp_tail}) with points $p(\cX)$ and their lift $\cX$, and a compact domain $D \subset \tilde M'$ containing all points in $\cX$ so that the minimal Euclidean distance between any point in $\cX$ and $\partial D$ is at least $r_{\cX}$. Note that we can choose $r_\cX$ for the $r$ in \eqref{def_r} there. Call $(\cT^D_x)_{x \in \cX}$ the resulting independent UST in $D$. For $x \in \cX$ let $\gamma^D_x$ be the branch in the UST $\cT^D_x$
joining $x$ to $\partial D$. Let $\gamma_x^D(t)$ be parametrised so that $\gamma^D_x$ enters $B(x,e^{-t}r_{\tilde K})$ for the first time (going from the outside to $x$) at time $t$.

Let $R_x$ be the isolation radius of  $x$ in the application of \cref{lem:exp_tail} and write $R_x = e^{-I_x}r_{\tilde K}$. Let $I = \max_{x \in \cX} I_x$. (Note that for notational clarity, $I$ here is shifted by $\log (r_{\tilde K}/\Delta)$ from the one in \cref{sec:fullcoupling}.)
%We know that on the event the coupling is not aborted, $I > t/10$ has probability at most $e^{-ct}$ uniformly over all $\delta<\delta(\cX,t_0)$.
We now decompose
\begin{align}
e_x(t) & = \underbrace{\overline W(\gamma^D_x(t,\infty),x)}_{\alpha_x} +
\underbrace{(\overline W(\gamma_x(t,\infty),x) - \overline W(\gamma^D_x(t,\infty),x)) }_{\xi_x} .\end{align}
Therefore, we need to deal with expectation of products of $\alpha_x,\xi_x$ for different indices $x$. Let $\mathfrak C$ be the sigma algebra generated by the cutset exploration. We will first compute the conditional expectation and then take the overall expectation.
%Note that $Z$ is measurable with respect to $\mathfrak C$ (final assertion of  \cref{lem:exp_tail}). Also the moments of $Z$ are bounded by \cref{eq:non-contractible_tail}.
Note that $\cT_x^D$ are independent for different $x$ and also independent of $\mathfrak C$ and hence any term involving $\alpha_x$ is 0. Thus we only have to deal with terms involving $\xi_x, x \in S$.

Let $\Lambda $ be the \emph{last} time $\gamma_x$ enters $B(x,e^{-I}r_{\tilde K})$. In the coupling, $\gamma_x, \gamma_x^D$ agree inside $B(x,e^{-I}r_{\tilde K})$ so in particular $\gamma_x(\Lambda, \infty)  = \gamma_x^D(\Lambda, \infty)$, so
\begin{equation*}
\xi_x = \overline W(\gamma_x(t,\Lambda),x) - \overline W(\gamma^D_x(t,\Lambda),x).
\end{equation*}
Let $G$ be the event that $\xi_x$ without the bar (meaning without the expectation terms) is 0. Observe that if $G$ does not occur then one of the following events happen. Either $$I  - \log(r_{\tilde K}  / \Delta)> (t-\log(r_{\tilde K}  / \Delta))/2$$ which has probability at most $e^{-c(t - \log(r_{\tilde K}  / \Delta))} \vee \delta^c \le (e^{-c(t-\log(r_{\tilde K}  / \Delta))} + \delta^c)$  (\cref{lem:exp_tail}). Otherwise, $\gamma_x$ has to exit $B(x, e^{-I}r_{\tilde K})$ after hitting $e^{-t}r_{\tilde K}$. This also has probability at most $e^{-c(t -\log(r_{\tilde K}  / \Delta)) }$  (\cref{lem:backtrack}). Now we bound the moments of $\xi_x$. Notice that we can write
\begin{align*}
\Ewils(|\xi_x|^k)  & \le \sum_{j = t/2 + \log (r_{\tilde K} / \Delta )/2}^{\infty} \E(|\xi_x|^k 1_{j \le I \le j+1}) + c e^{-c(t -\log(r_{\tilde K}  / \Delta)) }\\
&\le \sum_{j = t/2 + \log (r_{\tilde K} / \Delta )/2}^{\infty} |1+j|^k (e^{-c(j - \log(r_{\tilde K}  / \Delta) )} +\delta^c)
\end{align*}
where we first used Cauchy--Schwarz and then used \cref{cor:full_winding_moment} to bound the moment and \cref{lem:exp_tail}. Thus overall
$$
\Ewils(|\xi_x|^k) \le |1+t|^k( e^{-c(t- \log(r_{\tilde K}  / \Delta) )} +\delta^c).
$$
We have a product of at most $m$ terms and so using H\"older's inequality, taking $\alpha =m$ works.
%Observe that $\beta_x = 0$ without the bar (i.e., without expectation terms) is 0 unless the coupling fails for all scales which has probability $\delta^c$ and $\beta_x$ is bounded by a universal constant (cf. \cref{lem:number_crossing}). Finally $\xi_x $ without the bar (i.e. without expectation terms) is 0 unless $I >t$ and on that event,
%\begin{equation*}
%\xi_x = \overline W(\gamma_x(t,I),x) - \overline W(\gamma^D_x(t,I),x)
%\end{equation*}
%which is a sum of $O(1+I)$ terms, each of which has exponential tail and hence the $m$th moment of $\xi_x$ is bounded by a constant. But $I>t $  uniformly over all $\delta < \delta_{\bar D}$. Thus using Holder, the moment of product of terms involving $Z, \beta$ and $\xi$ is bounded above by $(ce^{-c't } + \delta^{c})$ which completes the proof.

Finally, if $S'$ is non-empty, the proof is exactly the same as the vertices in $S'$ are distinct from $S$ and hence local independence from the coupling still holds. We then use \cref{L:reg_winding_moment} to conclude using Cauchy--Schwarz. Details are left to the reader.
\end{proof}

\begin{corollary}\label{L:winding_moment}
Let $S\subset \cX$ containing vertices with distinct indices such that $|S| = m$. There exists a constant $c=c(\tilde K,m),\alpha=\alpha(m)$ such that for all $\delta < \delta_{\tilde K}$,
$$
|\Ewils(\prod_{x \in S} \bar F_x(\gamma_x) )| \le c|1+ \log^\alpha(\Delta)|.
$$
\end{corollary}
\begin{proof}
Decompose
\begin{equation}
\Ewils( \prod_{x \in S} \bar F_x(\gamma_x) ) = \Ewils( \prod_{x \in S}  (\bar F_x(\gamma_x,t) + \bar e_x(t)) )\label{eq:winding_moment1}
\end{equation}
with $t = 2\log(r_{\tilde K} / \Delta)$. Then this is a straightforward application of \cref{L:err_winding_moment,L:reg_winding_moment}.
\end{proof}

Now we prove a convergence of the height function integrated against $f$ in the sense of moments (still for the Wilson law $\Pwils$). Recall that $X_\delta$ converges in the sense of moments as $\delta \to 0$ if for all $k$, $\E(X_\delta^k)$ converges as $\delta \to 0$.
\begin{lemma}\label{lem:kth_moment_conv}
$\int_{\tilde K} \overline h_{\ext}^\d(z)\tilde f(z)d\mu(z)$ converges in the sense of moments under the law $\Pwils$. Furthermore, the limit does not depend on the sequence $(G')^\d$.
\end{lemma}
\begin{proof}
 Using \cref{eq:kth_moment_int,eq:height_to_winding2,eq:top_int_eq1}, we need to prove the convergence as $\delta \to 0$ of
\begin{equation}
\int_{\tilde K^{2k}} \Ewils( \prod_{x \in S} \bar F_x(\gamma^\d_x) )\frac{\tilde f^+(z_i)\tilde f^-(w_i)}{Z(\tilde f)}d\mu(z_i)d\mu(w_i).\label{Sproj}
\end{equation}
We use Fubini to bring the expectation inside the integral in the above display. We first observe that integrating \eqref{Sproj} over all sets of vertices $S$ so that $p(x) \neq p(y)$ for all $x \neq y \in S$ is enough (recall $p: \tilde M' \mapsto M'$ is the covering map). Indeed, using \cref{L:reg_winding_moment} and the fact that $\|f\|_\infty < \infty$, we see that the integrand can be bounded by $O(\log^\alpha (1/\delta))$. Since the volume of $\{S \subset \tilde K^{2k}: p(x) = p(y) \text{ for some }x,y \in S\}$ is $O(\delta)$ (indeed the number of preimages of any vertex in $\tilde K$ is bounded by a constant depending only on $\tilde K$), we see that the integral over this set is $O(\delta\log^\alpha (1/\delta) )$.

Thus we now concentrate on the integral \eqref{Sproj} over
\begin{equation}
\cA (\tilde K) = \{ \text{sets of vertices } S  \text{ so that } p(x) \neq p(y)  \text{ for all } x \neq y \in S  \}. \label{cA}
\end{equation}

 We now use \cref{L:winding_moment} and dominated convergence theorem. Since $\|f\|_\infty < \infty$ and $\log^m(\Delta)$ is integrable for any $m>0$,  we need to prove convergence of the expectation inside the integral above. Now we claim that the regularised part
\begin{equation}\label{E:fgammat}
\Ewils(\prod_{x \in S} \bar F_x(\gamma^\d_x,t) )
\end{equation}
converges as $\delta \to 0$. This follows from our assumption of a.s. convergence of Temperleyan CRSF and because the term inside the expectation is a.s. continuous function of the Temperleyan CRSF. The same can be said about the skeleton using the absolute continuity statement given by \Cref{prop:lemma_abs_cont_special}. Indeed, this follows from a.s. continuity properties of SLE$_2$ and the fact that the CRSF branches are made a.s. from finitely many chunks of SLE$_2$, in particular for a fixed $t$, the SLE curve is a.s. not a tangent to the boundary of circle of radius $e^{-t}$ (we refer to \cite[Section 3]{BLR_Riemann2}) for details. \cref{L:reg_winding_moment} tells us that the random variable in \eqref{E:fgammat} is uniformly integrable which completes the proof of convergence of the regularised term \eqref{E:fgammat}.

Now given a fixed $S$, choose $t \ge \log(r_{\tilde K}/\Delta)$. Now we claim that that the error term satisfies
\begin{equation}
|\Ewils( \prod_{x \in S} \bar F_x(\gamma_x) )- \Ewils( \prod_{x \in S} \bar F_x(\gamma_x,t) )| <c |1+t |^\alpha(e^{-c(t- \log(r_{\tilde K}  / \Delta))} + \delta^{c'}).\label{eq:error_30}
\end{equation}
Indeed, \eqref{eq:error_30} follows by writing $\bar F_x(\gamma_x) = \bar F_x(\gamma_x,t) + \bar e_x(t)$ and then expanding and using the bounds in \cref{L:reg_winding_moment,L:err_winding_moment}.
To finish the proof of the lemma, fix $\ve>0$. First choose $t $ large enough and then $\delta$ small enough so that the right hand side of \eqref{eq:error_30} is less than $\ve$. Next choose a smaller $\delta$ if needed so that for all $\delta'<\delta$,
$$ |\Ewils(\prod_{x}\bar F(\gamma^{\d'}_x,t)) - \Ewils(\prod_{x}\bar F_x(\gamma^{\d}_x,t))| <\ve$$
via the convergence of the regularised term. This completes the proof as $\ve$ is arbitrary.

Observe that \eqref{eq:error_30} completes the proof that the limit does not depend on the sequence $((G')^\d)_{\delta>0}$ since the main term \eqref{E:fgammat} is measurable with respect to the limiting continuum Temperleyan CRSF which is universal by \cref{thm:CRSF_universal}.
\end{proof}

It remains to prove convergence in law (still for $\Pwils$ for now).  
For this, we need to alter the definition of truncation slightly. Given $t>0$,  take a cover $\{B_\Euc(x,r_{\tilde K}e^{-10t}):x \in \tilde K\}$ of $\tilde K$ and take a finite subcover. Let $\mathfrak c(z), \mathfrak c(w)$ denote the center of one of the balls (chosen arbitrarily) in the finite subcover in which $z,w$ belong. Define $\mathsf F_x(t), \mathsf e_x(t)$ to be the same as $F_x(t)$ and $e_x(t)$ but we cut off the first time $\gamma_x$ enters $B(\mathfrak c(x) , e^{-t})$ (so compared to the above, the centre of the cutoff ball is shifted by an amount which is at most $e^{-10t}$). However, in this definition, we still measure winding around $x$ not $\mathfrak c(x)$.

\begin{lemma}\label{cor:shift_center}
The statements of \cref{L:reg_winding_moment,L:err_winding_moment} still hold if we replace $F,e$ by $\mathsf F, \mathsf e$ everywhere.
\end{lemma}
\begin{proof}
For \cref{L:reg_winding_moment}, the proof identically follows from \cref{cor:full_winding_moment,lem:tildeu} by noticing that the supremum over all continuous subpaths contain the portion of the branch until it first hits the shifted ball. Notice also that we shift only by an additive term in the exponential scale which is $O(t)$.

For \cref{L:err_winding_moment}, the proof is also identical, in particular we still consider the coupling around the points in $S$. Because we still shift only by $O(t)$ in the exponential scale, the proof readily follows.
\end{proof}

We can now complete the proof of \cref{thm:main_precise}.
%\begin{thm}\label{lem:conv_in_law}
%Let $\cT^\d$ be the oriented Temperleyan CRSF coupled with the dimer configuration through Temperleyan bijection. Then the joint law of
%$$(h_{\ext}^\d(z),\cT^\d)$$ converges in the %sense of \cref{thm:main_precise}. Furthermore, the marginal of the limit's first coordinate is measurable with respect to the second coordinate, does not depend on the sequence $G'^\d$ and is conformally invariant.
%\end{thm}
\begin{proof}[Proof of \cref{thm:main_precise}]
Again, we first prove this under $\Pwils$ before explaining how to extend this and concluding the proof of \cref{thm:main_precise} to $\Ptemp$. We write $r = r_{\tilde K}$ and pick a $t > 0$ to be taken large later.  Recall from \cref{eq:int_double_int,eq:path_height},
\begin{align*}
\int_{\tilde K} \overline h_{\ext}^\d(z)\tilde f(z)d\mu(z)=\int_{\tilde K \times \tilde K} (\bar W(\tilde \gamma^\d_{zw},z) + \bar W(\tilde \gamma_{zw}^\d ,w) +  \bar \Psi_{zw}) \frac{\tilde f^+(z)\tilde f^-(w)}{Z(\tilde f)}d\mu(z)d\mu(w).
\end{align*}
Here $Z(\tilde f)$ is a deterministic constant so we can assume it to be 1 from now on without loss of generality. Firstly we recall that we only need to integrate over the set $\cA(\tilde K)$ as in \eqref{cA} as the integral over the remaining part is $O(\delta \log^c(1/\delta))$.  Let us denote by $Y = Y(\delta)$ the above integral over $\cA(\tilde K) $. Introduce
\begin{equation*}
X(\delta,t) :=  \int_{\cA(\tilde K)} (\bar \sF_{\mathfrak{c}(z)\mathfrak{c}(w)}(t) + \bar \sF_{\mathfrak{c}(w) \mathfrak{c}(z)}(t)   )1_{|\mathfrak{c}(z)- \mathfrak{c}(w)| >re^{-t/10}} \tilde f^+(z)\tilde f^-(w)d\mu(z)d\mu(w).
\end{equation*}
Note that $X(\delta,t)$ is a sum of regularised winding of finitely many branches and hence converges in law (as $\delta \to 0$ and $t$ is fixed) by our assumption.

Now we show that for all $t>0$ and $\delta  < C'\delta_0e^{-t}$,
$$
\E(Y(\delta) - X(\delta,t))^2 \le c(1+t)^\alpha e^{-c't}.
$$
We expand the above square to get an integral over $(z,w,z',w') \in \cA(\tilde K)^2$. We can again without loss of generality restrict to the the set $\cA_2(\tilde K)$ such that the projection $p$ of all four points $(z, w, z',w')$ maps to pairwise distinct points on $M$. In that case let
$\Delta = \Delta (z,w,z',w')$ be as in \eqref{eq:Delta}. We now argue that this integral over the set $\Delta \le e^{-t/11}$ is  exponentially small in $t$. Indeed, we are integrating over a set which has exponentially small measure with respect to $\mu^4$ and furthermore the moments are integrable using \cref{L:winding_moment,cor:shift_center}.

For the rest of the integral, we write
\begin{align*}
 (\bar W(\tilde \gamma^\d_{zw},z) + \bar W(\tilde \gamma_{zw}^\d ,w) + \bar \Psi_{zw}) =  \bar \sF_{zw}(t) + \bar \sF_{wz}(t) +\bar \se_{zw}(t) + \bar \se_{wz}(t)
\end{align*}
and we expand again the product inside the integral to get products of terms of the form
$$
 \bar \se_{zw}(t) + \bar \se_{wz}(t) ;  \quad \bar \sF_{zw}(t) + \bar \sF_{wz}(t) - \bar \sF_{\mathfrak c(z)\mathfrak c(w)}(t) - \bar \sF_{\mathfrak c(w) \mathfrak c(z)}(t).
$$
Note here that the indicator over $|\mathfrak c(z) - \mathfrak c(w) | >e^{-t/10}$ is included in $\Delta > e^{-t/11}$ so we can get rid of the indicator. Products containing at least one $\se$ are small because of \cref{cor:shift_center}. Also note that on $\Delta  > e^{-t/11}$, $|z-w| >c(\tilde K) e^{-t/11}$. Thus we need to show
\begin{equation*}
\E(  |\bar \sF_{zw}(t) -  \bar \sF_{\mathfrak c(w) \mathfrak c(z)}(t) |^2 1_{|z-w| > c(\tilde K)e^{-t/11}} ) \le c(1+t)^\alpha e^{-c't}.
\end{equation*}
Observe that
\begin{multline*}
 (\bar \sF_{zw}(t) -  \bar \sF_{\mathfrak c(z) \mathfrak c(w)}(t))1_{|z-w| > c(\tilde K) e^{-t/11}}   = (\bar \sF_{zw}(t) -  \bar \sF_{\mathfrak c(z) w}(t))1_{|z-w| > c(\tilde K) e^{-t/11}} \\+ (\bar \sF_{\mathfrak c(z)w}(t) -  \bar \sF_{\mathfrak c(z) \mathfrak c(w)}(t))1_{|z-w| > c(\tilde K) e^{-t/11}}).
 \end{multline*}
If $\mathsf p(z)$ and $\mathsf p(\mathfrak c(z))$ merge before exiting $B(\mathfrak c(z), e^{-5t})$, then the paths we need to consider are identical outside $B(\mathfrak c(z), e^{-t})$, only the centers around which we measure the topological winding are different. This difference in winding is therefore deterministically $O(e^{-9t})$. On the other hand, the paths we need to consider are not identical outside $B(\mathfrak c(z), e^{-5t})$
if and only if a simple random walk from $z$ does not hit $\mathsf p(\mathfrak c(z))$ before exiting $B(\mathfrak c(z), e^{-5t})$ which has probability exponentially small in $t$ by a Beurling-type lemma (\cref{lem:Beurling}). Hence we can conclude using Cauchy--Schwarz and \cref{cor:shift_center}. (Note that here the shifted cutoff is particularly useful since the cutoff point is the same for both $z$ and $\mathfrak c(z)$).

The proof of the rest of the statements is a simple consequence of the fact that the main term $X(\delta, t)$ is an a.s.\ continuous function of $\cT^\d$ for a fixed $t$. So trivially $(X(\delta, t), \cT^\d)$ converges jointly in law, the limit does not depend on the sequence $(G')^\d$ and is conformally invariant. It is a simple exercise to show that the $L^2$ bound on the error term $Y(\delta) - X(\delta,t)$ as shown above is enough to conclude the proof of the theorem (i.e. the proof of the same claim as in the previous sentence for the pair $(\bar h_{\ext}^\d(z),\cT^\d)$ ).

We have proved convergence in law of $(\bar h_{\ext}^\d(z),\cT^\d)$ for the $\Pwils$ law. We now explain how to convert the result so that it holds under the Temperleyan law $\Ptemp$.
Recall that if $\cT$ is sampled by performing Wilson's algorithm (after sampling the skeleton) and $\cT^\dagger$ is sampled from the uniform distribution among all oriented duals of $\cT$, then the Radon--Nikodym derivative of $(\cT, \cT^\dagger)$ with respect to $\Ptemp$ is $Z = 2^{K^\dagger}/ \Ewils (2^{K^\dagger})$ where $K^\dagger $ is the number of nontrivial cycles of $\cT^\dagger$ by \cref{lem:RN_dimer_CRSF}. Now, observe that $Z$ is measurable with respect to $\cT^\d$ so $(Z, (\bar h_{\ext} , \tilde f))$ jointly converges in law under $\Pwils$ by the above. Furthermore all moments of $Z$ and of $(\bar h_{\ext} , \tilde f)$ are bounded under $\Pwils$ and therefore we conclude that
$$
\left( \int_{\tilde M} (h_{\ext} (z) - \Ewils( h_{\ext} (z) )) \tilde f(z) \mu (dz) ; \cT^\d\right)
$$
converges in law and in the sense of moments under $\Ptemp$. In particular, taking expectation (under $\Ptemp$) of the first quantity, we also deduce that $\int_{\tilde M} (\Etemp (h_{\ext} (z)) - \Ewils( h_{\ext} (z) )) \tilde f(z) \mu (dz)$ converges. Taking the difference, it follows that
$$
\left( \int_{\tilde M} (h_{\ext} (z) - \Etemp( h_{\ext} (z) )) \tilde f(z) \mu (dz) ; \cT^\d\right)
$$
converges in law and in the sense of moment under $\Ptemp$, as desired. Furthermore, since the limit of $\cT^\d$ is independent of the chosen graph sequence (subject to the assumptions in \cref{sec:setup}), so is the limit of $\int \tilde f(x) \bar h^\d_{\ext}(x)  \vol(x)$.
\end{proof}
%
%\begin{proof}[Proof of \cref{thm:main_precise}]
%\cref{lem:conv_in_law} proves that $\int \tilde f(x) \bar h^\d_{\ext}(x)  \vol(x)$ converges in law under $\Ptemp$ and in the sense of all moments, and is measurable with respect to $\cT$. Furthermore, since the limit of $\cT^\d$ is independent of the chosen graph sequence (subject to the assumptions in \cref{sec:setup}), so is the limit of $\int \tilde f(x) \bar h^\d_{\ext}(x)  \vol(x)$.
%\end{proof}

\begin{remark}\label{rmk:instanton_jt_conv}
In fact, with some more work, the convergence in \cref{thm:main_precise} can be upgraded to specifically include information about the instanton component. More precisely, using \cref{lem:instanton_global} we obtain the following.

Take an ordered finite set of continuous simple loops which forms the basis of the first homology group of $M'$, all endowed with a fixed orientation. Let $\mathsf H^\d \in \R^{4g+2b-4}$ denote the vector of height differences along these loops (i.e., for each such loop, record the height accumulated by going along the loop once in the prescribed orientation). Consider the one-form
$$
d \bar h^\d_{\di}(e)  :=\bar h^\d_{\di}(\tilde e^+) - \bar h^\d_{\di}(\tilde e^-).
$$
It is well-known (see \cref{thm:hodge}) that the instanton component of the above one-form is completely determined by $\mathsf H^\d$. Then we have:
$$
\left(\int \tilde f(x) \bar h^\d_{\ext}(x)  \vol(x),  \mathsf H^\d , \cT^\d\right)
$$
converges jointly in law as $\delta \to 0$. The first two coordinates also jointly converge in the sense of all moments. Furthermore,
$$
\lim_{\delta \to 0}\left(\int \tilde f(x) \bar h^\d_{\ext}(x)  \vol(x) , \mathsf H^\d \right)
$$
is measurable with respect to the limit $\cT$ of $\cT^\d$. We do not go into the details of these claims for the sake of brevity.
\end{remark}

\appendix
\section{Geometry of spines}\label{app:spine}
In this section we prove \cref{lem:spine:boundary} and \cref{lem:macro_winding}. But before getting into the proofs, we remind the readers certain basic facts from the classical theory of Riemann surfaces.

 By the uniformisation theorem of Riemann surfaces (\cref{SS:universalcover}), recall that there exists a conformal map from the Riemann surface $\D / \mathfrak{g}$ to $M'$ where $\mathfrak{g}$ is a Fuchsian group which is a discrete subgroup of the group of M\"obius transforms. Furthermore, such a conformal map is unique up to conformal automorphisms (i.e. M\"obius transforms) of the unit disc. In other words, if $\mathfrak{g},\mathfrak{g}'$ are two Fuchsian groups such that $M'$ is conformally equivalent to both $\D / \mathfrak{g}$ and $\D/ \mathfrak{g}'$ then there exists a M\"obius map $\phi: \D \mapsto \D$ such that $\mathfrak{g}'  = \phi^{-1}\circ\mathfrak{g}\circ \phi$. Since we have fixed a canonical lift, we have defined $\Gamma$ uniquely.

It is further known that $\mathfrak{g}$ is isomorphic to the fundamental group $\pi_1(M',x)$ (topologically, this is also known as the group of deck transformations). This connection is described as follows: choose a base point in the manifold $x_0$ and a particular lift of $\tilde x_0$, then any simple loop $\ell$ based in $x_0$ can be lifted to a simple curve $\tilde \ell$ in $\D$ starting from $\tilde x_0$, with the endpoint $\tilde x_1$ depending only on the homotopy class of the loop. Then to $\ell$ we associate the map $\phi_{\ell,\tilde x_0} \in \mathfrak{g}$ that sends $\tilde x_0$ to $\tilde x_1$.
%This means that for any loop $\ell$ $\tilde \ell$ is a simple curve in $\H$ starting from $\tilde x$ and ending at $\phi_\ell(\tilde x)$ where $\phi_\ell$ is the M\"obius map in $\Gamma$ corresponding to the element in $\pi_1(M',x)$ specified by $\ell$.
Note then that $\phi_{\ell,\tilde x_0}(\tilde \ell)$ is a curve which projects via $p$ to the same curve $\ell$ in $M'$ since $M'$ is conformally equivalent to $\D / \mathfrak{g}$ (and in particular is homeomorphic) and it does not intersect $\tilde \ell$ since $\ell$ is a simple loop. Furthermore, since the endpoint of $\phi_{\ell,\tilde x_0}(\tilde \ell)$ is $\phi_{\ell,\tilde x_0} \circ \phi_{\ell,\tilde x_0} (\tilde x_0)$, by the unique path lifting property, the curve $\tilde \ell \cup \phi_{\ell,\tilde x_0}(\tilde \ell)$ is the unique lift obtained by going around $\ell$ twice in the same direction. Iterating, we obtain that the infinite path obtained by going around $\ell$ in the same direction is given by $\cup_{n=0}^\infty \phi_{\ell,\tilde x_0}^{(n)}(\tilde \ell)$ where $\phi_{\ell,\tilde x_0}^{(n)}$ is the $n$-fold composition of $\phi_{\ell,\tilde x_0}$ and that the union is a disjoint union.

We can actually say more using the classification of M\"obius maps according to their trace.
Recall that M\"obius maps preserving the unit disc have the form
$$
\phi(z)=e^{i \theta} \frac{z-a}{ \bar a z-1} ; \quad |a| < 1 ; \quad \theta \in [0, 2 \pi)
$$
and can be classified (up to conjugation with M\"obius transforms) depending upon the behaviour of the trace:
\begin{itemize}
\item If $|e^{i\theta} -1|^2 = 4(|a|^2-1)^2$, then $\phi(z)$ is conjugate to either $z+1$ or $z-1$ seen as maps from $\H$ to $\H$.
\item If $|e^{i\theta} -1|^2 > 4(|a|^2-1)^2$, then $\phi(z)$ is conjugate to $\lambda z$ where $\lambda>1$ seen as maps from $\H$ to $\H$.
\item If $|e^{i\theta} -1|^2 < 4(|a|^2-1)^2$, then $\phi$ is conjugate to a rotation of the unit disc.
\end{itemize}
Thus there is a M\"obius map $\Phi$ from $\D$ to $\D$ or $\H$ such that $\phi = \Phi^{-1} \circ \mu_{\phi} \circ \Phi$ where $\mu_{\phi}$ is either a translation by $\pm 1$ or a scaling by $\lambda$ (in case $\Phi$ maps $\D$ to $\H$) or a rotation (in case $\Phi$ maps $\D$ to $\D$).

We argue for any loop $\ell$, $\mu_{\phi_{\ell,\tilde x_0}}$ cannot be a rotation.
%This is because if we map the whole lift of loop to $\D$ by a fixed M\"obius map $\psi$, then the images of $\psi(\tilde x)$ under the iteration of pushforward of $\mu$ by $\psi$ forms a dense set in a circle.
Indeed, if a map $\phi_{\ell,\tilde x_0}$ was conjugate to a rotation, then its iterates would be either periodic, or a dense set (on the image of a circle). Being periodic is forbidden because of path lifting property and being dense is forbidden because $\mathfrak{g}$ is discrete.
From the two remaining cases, we can complete the proof of \cref{lem:spine:boundary}.
%\begin{lemma}\label{lem:spine:boundary}
%Any spine $S$ in the CRSF is either a simple path in $\D$ connecting two points in $\partial \D$ or a simple loop containing a unique point in $\partial \D$.
%\end{lemma}
\begin{proof}[Proof of \cref{lem:spine:boundary}]

Let $S$ be a spine, we can choose a point $x_0$ on $p( S)$ and a lift $\tilde x_0$ on $S$. Then we see that the previous general theory applies so we can find an open path $\tilde \ell$ (a certain lift of the path going once around a non-contractible loop) and a map $\phi_{p(S),\tilde x_0}$ such that
$$
S = \bigcup_{n \in \Z} \phi_{p(S),\tilde x_0}^{(n)}( \tilde \ell ),
$$
and the map $\phi_{p(S),\tilde x_0}$ is conjugate to either a scaling or a translation. The case of a scaling clearly gives a simple path between two different points (the image of $0$ and $\infty$) while the case of a translation gives a simple loop where the image of $\infty$ is the unique point point of $\partial \D$ on the loop.

In the case of the torus, the map $\phi_{p(S)}$ is simply a translation in $\C$ so the result is trivially true.
\end{proof}

\medskip

We now prove  \cref{lem:macro_winding} which provides bounds on the macroscopic winding of the spines. We repeat the statement of the Lemma here for convenience.
\begin{lemma*}[Restating \cref{lem:macro_winding}]
Fix a compact set $\tilde K \subset \tilde M'$. There exist constants $c,c'>0$ such that for all $\tilde v \in \tilde K$, for all $\delta<\delta_{ {p(\bar B_0)}}$, $n \ge 1$, $j \ge 1 $ and $u \in B_1$ such that $\P(X_{\tau_{i_j+1}} = u) >0$,
\begin{gather*}
\P \Big(\sup_{\cY \subset \tilde X[\tau_{i_j},\tau_{i_j+1}]}|W(\cY ,\tilde v)| >n\Big | X_{\tau_{i_j+1}} = u\Big) \le ce^{-c'n},\\
\P \Big(\sup_{\cY \subset \tilde X[\tau_{i_j},\tau_{i_j+1}]}|W(\cY ,\tilde v)| >n\Big | X_{\tau_{i_j+1}} \in \partial (G')^\d \cup \mathfrak s \Big) \le ce^{-c'n}.
\end{gather*}
Here the supremum is over all continuous paths obtained by erasing portions of $\tilde X[\tau_{i_j},\tau_{i_j+1}]$. %\note{add stuff for $t_0$}
\end{lemma*}

\begin{proof}
We are going to borrow the notations from \cref{sec:apriori}.
Take a non-contractible simple loop $\ell$ in $M'$ through $v$ and find a continuous path $\ell^\d$ in $G^\d$ which approximates this loop in the sense that it stays within distance $c\delta/\delta_0$ from $\ell$ (this is guaranteed to exist by the uniform crossing estimate for small enough $\delta$ depending on $\ell$). Notice that the lift starting from $\tilde v$ of a curve which winds infinitely many times in the clockwise (resp. anti-clockwise) direction of $\ell^\d$ defines an infinite path $\tilde \ell_+$ (resp. $\tilde \ell_-$) which converge to the boundary in the hyperbolic case (\cref{lem:spine:boundary}) or goes to infinity in an asymptotic direction in case of the torus. Let $\tilde \ell'_{\pm}$ denote the portion of $\tilde \ell_{\pm}$ from the last exit of $B_1$ to infinity (given an arbitrary parametrisation starting from $\tilde v$). Notice that $(\tilde \ell'_- \cup \tilde \ell'_+)$ divides the annulus $\tilde M' \setminus B_1$ into two simply connected domains. By compactness we can find a constant $C$ such that the winding of $\tilde \ell'_+ $ around $\tilde v$ is bounded by $C$, uniformly over all points $\tilde v \in \tilde K$ for a suitable choice of loop $\ell$.
Let $\tau_{+,1}$ be the first hitting time of $\tilde \ell'_+$ in the interval $[\tau_{i_j},\tau_{i_j+1}]$ and by induction define $\tau_{-,i}$ to be the first hitting time of $\tilde \ell'_-$ after $\tau_{+,i}$ and define $\tau_{+, i}$ the hitting time of $\tilde \ell'_+$ after $\tau_{-, i-1}$. Let $I_+$ the number of $\tau_{+,i}$ before $\tau_{i_{j}+1}$, i.e $I_+ = \abs{\{ i | \tau_{+,i} \leq \tau_{i_{j}+1}\}}$. We now use the deterministic bound
$$
\sup_{\cY \subset \tilde X[\tau_{i_j},\tau_{i_j+1}]} |W(\cY ,\tilde v)| \le (C+2\pi) (I_+ +1).
$$
and observe that conditionally on either $\{X_{\tau_{i_j+1}} = u\}$ or $\{X_{\tau_{i_j+1}} \in \partial (G')^\d \cup \mathfrak s\}$, $I_+$ has an exponential tail. The proof of this fact is essentially the same as the second item of Lemma 4.8 in \cite{BLR16}. Indeed we need to show that once the random walk intersects $\tilde \ell'^+$ outside $B_1$, there is a positive probability (uniform in $\delta$) for the walk to create a non-contractible loop without intersecting $\tilde \ell'^-$ (as in the proof of \cref{lem:uniform_avoidance}). This is intuitively clear, but a complete proof of this needs an input from Riemannian geometry. The issue at hand is that too close to the boundary, the uniform crossing ceases to hold uniformly in $\delta$. We control the winding of this portion of this walk as follows.

Consider a compact set $A \subset  M'$ containing $\ell$ which forms an approximation of $\ell$, in the sense that topologically $K$ is an annulus with $\ell$ being non-contractible in $K$. Recall from the proof of \cref{lem:uniform_avoidance} that the simple random walk has a uniform positive probability to create a non-contractible loop inside $K$ by winding around exactly once and we stop the simple random walk if this happens. But in this process, we can assume without loss of generality on $\ell$ that the lift of the walk stays inside at most four consecutive copies of $K$ (where these four copies are mapped to each other by M\"obius transforms). So if the walk is on $\tilde \ell_+$ at a copy of $\tilde A$  which is more than four copies away from $\tilde v$ and $\tilde \ell^-$, the lift of that walk cannot intersect $\ell^-$ in this process.
On the other hand, applying four copies of the corresponding M\"obius map to an arbitrary $\tilde v \in \tilde K$ yields by compactness of $\tilde K$ a slightly bigger compact $\tilde K'$. Applying the uniform crossing property for the given choice of $\delta$ to $\tilde K' $, we now simply apply the argument of Lemma~4.8 in \cite{BLR16}.
\end{proof}

\bibliographystyle{abbrv}
\bibliography{torus}

\begin{thebibliography}{10}

\bibitem{ahn2021lozenge}
A.~Ahn, M.~Russkikh, and R.~Van~Peski.
\newblock Lozenge tilings and the {G}aussian free field on a cylinder.
\newblock {\em arXiv preprint arXiv:2105.00551}, 2021.

\bibitem{aru2015}
J.~Aru.
\newblock {\em The geometry of the {G}aussian free field combined with {SLE}
  processes and the {KPZ} relation}.
\newblock PhD thesis, Ecole normale sup{\'e}rieure de {L}yon, 2015.

\bibitem{BHS}
N.~Berestycki and L.~Haunschmid-Sibitz.
\newblock Near-critical dimers and massive {SLE}.
\newblock {\em arXiv preprint arXiv:2203.15717}, 2022.

\bibitem{BLR16}
N.~Berestycki, B.~Laslier, and G.~Ray.
\newblock Dimers and imaginary geometry.
\newblock {\em The Annals of Probability}, 48(1):1--52, 2020.

\bibitem{BLR_Riemann2}
N.~Berestycki, B.~Laslier, and G.~Ray.
\newblock Dimers on {R}iemann surfaces {II}: conformal invariance and scaling
  limits.
\newblock Preprint, 2022.

\bibitem{BLiu}
N.~Berestycki and M.~Liu.
\newblock Piecewise {T}emperleyan dimers and a multiple {SLE}$_8$.
\newblock {\em https://arxiv.org/abs/2301.08513}, 2022.

\bibitem{BerestyckiPowell}
N.~Berestycki and E.~Powell.
\newblock {G}aussian free field, {G}aussian multiplicative chaos and
  {L}iouville quantum gravity.
\newblock {\em Book in preparation, available on authors' webpage}, 2015.

\bibitem{bott}
R.~Bott and L.~W. Tu.
\newblock {\em Differential forms in algebraic topology}, volume~82.
\newblock Springer Science \& Business Media, 2013.

\bibitem{BdT_torus}
C.~Boutillier and B.~De~Tili{\`e}re.
\newblock Loop statistics in the toroidal honeycomb dimer model.
\newblock {\em The Annals of Probability}, 37(5):1747--1777, 2009.

\bibitem{cimasoni}
D.~Cimasoni.
\newblock Discrete {D}irac operators on {R}iemann surfaces and {K}asteleyn
  matrices.
\newblock {\em J. Eur. Math. Soc.}, 14:1209--1244, 2009.

\bibitem{cimasoni2007dimers}
D.~Cimasoni and N.~Reshetikhin.
\newblock Dimers on surface graphs and spin structures. i.
\newblock {\em Communications in Mathematical Physics}, 275(1):187--208, 2007.

\bibitem{cimasoni2008dimers}
D.~Cimasoni and N.~Reshetikhin.
\newblock Dimers on surface graphs and spin structures. ii.
\newblock {\em Communications in mathematical physics}, 281(2):445--468, 2008.

\bibitem{ciucu}
M.~Ciucu and C.~Krattenthaler.
\newblock A factorization theorem for lozenge tilings of a hexagon with
  triangular holes.
\newblock {\em Trans. Amer. Math. Soc.}, 369(5):3655--3672, 2017.

\bibitem{CohnKenyonPropp}
H.~Cohn, R.~Kenyon, and J.~Propp.
\newblock A variational principle for domino tilings.
\newblock {\em J. Am. Math. Soc.}, 14(2):297--346, 2001.

\bibitem{costa2002dimers}
R.~Costa-Santos and B.~M. McCoy.
\newblock Dimers and the critical ising model on lattices of genus {$>1$}.
\newblock {\em Nuclear Physics B}, 623(3):439--473, 2002.

\bibitem{DKRV}
F.~David, A.~Kupiainen, R.~Rhodes, and V.~Vargas.
\newblock Liouville quantum gravity on the {R}iemann sphere.
\newblock {\em Communications in Mathematical Physics}, 342(3):869--907, 2016.

\bibitem{Coulomb_gas}
P.~Di~Francesco, H.~Saleur, and J.-B. Zuber.
\newblock Relations between the {C}oulomb gas picture and conformal invariance
  of two-dimensional critical models.
\newblock {\em J. Stat. Phys.}, 49(1-2):57--79, 1987.

\bibitem{donaldson2011riemann}
S.~Donaldson.
\newblock {\em Riemann surfaces}.
\newblock Oxford University Press, 2011.

\bibitem{dubedat_torsion}
J.~Dub{\'e}dat.
\newblock Dimers and families of {C}auchy {R}iemann operators {I}.
\newblock {\em J. Am. Math. Soc}, 28(4):1063--1167, 2015.

\bibitem{DubedatGheissari}
J.~Dub{\'e}dat and R.~Gheissari.
\newblock Asymptotics of height change on toroidal {T}emperleyan dimer models.
\newblock {\em Journal of Statistical Physics}, 159(1):75--100, 2015.

\bibitem{GKRV}
C.~Guillarmou, A.~Kupiainen, R.~Rhodes, and V.~Vargas.
\newblock Segal's axioms and bootstrap for {L}iouville theory.
\newblock {\em arXiv preprint arXiv:2112.14859}, 2021.

\bibitem{hsu2002}
E.~P. Hsu.
\newblock {\em Stochastic analysis on manifolds}, volume~38.
\newblock American Mathematical Soc., 2002.

\bibitem{Jost}
J.~Jost.
\newblock {\em Compact {R}iemann surfaces}.
\newblock Springer, 2006.

\bibitem{KK12}
A.~Kassel and R.~Kenyon.
\newblock Random curves on surfaces induced from the {L}aplacian determinant.
\newblock {\em Ann. Probab.}, 2012.

\bibitem{Kenyon_ci}
R.~Kenyon.
\newblock Conformal invariance of domino tiling.
\newblock {\em Ann. Probab.}, 28(2):759--795, 2000.

\bibitem{KenyonGFF}
R.~Kenyon.
\newblock Dominos and the {G}aussian free field.
\newblock {\em Ann. Probab.}, 29(3):1128--1137, 2001.

\bibitem{KenyonOkounkov2}
R.~Kenyon and A.~Okounkov.
\newblock Limit shapes and the complex burgers equation.
\newblock {\em Acta Math.}, 199(2):263--302, 2007.

\bibitem{KPWtemperley}
R.~W. Kenyon, J.~G. Propp, and D.~B. Wilson.
\newblock Trees and matchings.
\newblock {\em Electron. J. Combin.}, 7, 2000.

\bibitem{LacoinRhodesVargas}
H.~Lacoin, R.~Rhodes, and V.~Vargas.
\newblock Semiclassical limit of {L}iouville field theory.
\newblock {\em J. Funct. Anal.}, 273(3):875--916, 2017.

\bibitem{LSW}
G.~F. Lawler, O.~Schramm, and W.~Werner.
\newblock Conformal invariance of planar loop-erased random walks and uniform
  spanning trees.
\newblock {\em Ann. Probab.}, 32(1):939--995, 2004.

\bibitem{mercat07}
C.~Mercat.
\newblock Discrete {R}iemann surfaces.
\newblock In {\em Handbook of {T}eichm\"{u}ller theory. {V}ol. {I}}, volume~11
  of {\em IRMA Lect. Math. Theor. Phys.}, pages 541--575. Eur. Math. Soc.,
  Z\"{u}rich, 2007.

\bibitem{IG1}
J.~Miller and S.~Sheffield.
\newblock Imaginary geometry {I}: interacting {SLE}s.
\newblock {\em Probab. Theor. Rel Fields}, (164):553--705, 2016.

\bibitem{IG4}
J.~Miller and S.~Sheffield.
\newblock Imaginary geometry {IV}: interior rays, whole-plane reversibility,
  and space-filling trees.
\newblock {\em Probab. Theor. Rel. Fields}, (169):729--869, 2017.

\bibitem{petrov}
L.~Petrov.
\newblock Asymptotics of uniformly random lozenge tilings of polygons.
  {G}aussian free field.
\newblock {\em The Annals of Probability}, 43(1):1--43, 2015.

\bibitem{marianna}
M.~Russkikh.
\newblock Dimers in piecewise {T}emperleyan domains.
\newblock {\em Comm. Math. Phys.}, 359(1):189--222, 2018.

\bibitem{SLE}
O.~Schramm.
\newblock Scaling limits of loop-erased random walks and uniform spanning
  trees.
\newblock {\em Israel J. Math.}, 118:221--288, 2000.

\bibitem{GFFShe}
S.~Sheffield.
\newblock Gaussian free fields for mathematicians.
\newblock {\em Probab. Th. Rel. Fields}, 139(3-4):521--541, 2007.

\bibitem{S16}
W.~Sun.
\newblock Toroidal dimer model and {T}emperley's bijection.
\newblock {\em arXiv:1603.00690}, 2016.

\end{thebibliography}

\end{document}